\documentclass[fleqn,11pt,oneside]{article}

\usepackage{hyperref}
\usepackage{dsfont}
\usepackage{graphicx}
\usepackage{bbm}  
\usepackage{bm}  
\usepackage{amsmath,amsthm,amssymb,amscd}
\usepackage{caption,subcaption}
\usepackage{booktabs,multirow}
\usepackage{setspace} 
\usepackage[top=1.1in, bottom=1.1in, left=1.6in, right=1.1in]{geometry}
\usepackage[inline,final]{showlabels}
\usepackage{dashbox}
\usepackage{microtype}
\usepackage[medium]{titlesec}
\usepackage{algpseudocode}
\usepackage[ruled]{algorithm}
\usepackage{xpatch}
\usepackage{accents}
\usepackage{bbding}
\usepackage[table]{xcolor}
\usepackage{srcltx}
\usepackage{cleveref}

\newtheorem{theorem}{Theorem}[section]
\newtheorem{lemma}[theorem]{Lemma}
\newtheorem{corollary}[theorem]{Corollary}
\newtheorem{definition1}{Definition}[section]
\newtheorem{observe}{Observation}[section]
\newtheorem{remark1}[observe]{Remark}
\newtheorem{example1}{Example}[section]
\newtheorem{aside1}[observe]{Aside}

\newenvironment{observation}{\begin{observe} \rm}{\end{observe}}
\newenvironment{remark}{\begin{remark1} \rm}{\end{remark1}}

\def\qed{\hfill$\blacksquare$\\} \renewenvironment{proof}{\noindent {\bf 
Proof.}}{\qed}

\algrenewcommand\algorithmicindent{3.0em}

\makeatletter
\xpatchcmd{\algorithmic}{\itemsep\z@}{\itemsep=2.3ex}{}{}
\makeatother

\newif\ifshowboxes \showboxestrue

\renewcommand{\hat}{\widehat}
\renewcommand{\tilde}{\widetilde}

\newcommand{\inner}[1]{\ensuremath{ {\langle #1 \rangle} }}

\newcommand{\norm}[1]{\ensuremath{ {\lVert #1 \rVert} }}
\newcommand{\bnorm}[1]{\ensuremath{ {\bigl\lVert #1 \bigr\rVert} }}

\newcommand{\bbnorm}[1]{\ensuremath{ {\biggl\lVert #1 \biggr\rVert} }}

\newcommand{\abs}[1]{\ensuremath{ {\lvert #1 \rvert} }}

\def\R{\mathbbm{R}}
\def\1{\mathbbm{1}}

\begin{document}

\begin{center}
   \begin{minipage}[t]{6.0in}

In this paper, we describe an algorithm for approximating functions 
of the form 
$f(x)=\int_{a}^{b} x^{\mu} \sigma(\mu) \, d \mu$ over $[0,1]$, where
$\sigma(\mu)$ is some signed Radon measure, or, more generally,
of the form 
$f(x) = \inner{\sigma(\mu), x^\mu}$, where
$\sigma(\mu)$ is some
distribution supported on $[a,b]$, with $0 <a < b< \infty$.
One example from this class of functions is 
$x^c (\log{x})^m=(-1)^m \inner{\delta^{(m)}(\mu-c), x^\mu}$, where $a\leq c \leq b$
and $m \geq 0$ is an integer.
Given the desired accuracy $\epsilon$ and the values of $a$ and $b$,
our method determines a priori a collection of non-integer powers 
$t_1$, $t_2$, \ldots, $t_N$, 
so that the functions are approximated by series of the form
$f(x)\approx \sum_{j=1}^N c_j x^{t_j}$,
and a set of collocation points
$x_1$, $x_2$, \ldots, $x_N$, 
such that the expansion coefficients 
can be found by 
collocating the function at these points.
We prove that our method has a small uniform approximation error which is
proportional to $\epsilon$ multiplied by some small constants, and
that the number of singular powers and
collocation points grows as $N=O(\log{\frac{1}{\epsilon}})$. We
demonstrate the performance of our algorithm with several numerical
experiments.

\thispagestyle{empty}

  \vspace{ -100.0in}

  \end{minipage}
\end{center}

\vspace{ 2.60in}
\vspace{ 0.50in}

\begin{center}
  \begin{minipage}[t]{4.4in}
    \begin{center}

\textbf{On the Approximation of Singular Functions by Series of Non-integer
Powers}\\

  \vspace{ 0.50in}

Mohan Zhao$\mbox{}^{\dagger \, \diamond}$
and
Kirill Serkh$\mbox{}^{\ddagger\, \diamond}$ \\
University of Toronto NA Technical Report \\
              v2, \today

    \end{center}
  \vspace{ -100.0in}
  \end{minipage}
\end{center}

\vspace{ 2.00in}

\vfill

\noindent 
$\mbox{}^{\diamond}$  This author's work  was supported in part by the NSERC
Discovery Grants RGPIN-2020-06022 and DGECR-2020-00356.
\\
%


\vspace{2mm}

\noindent
$\mbox{}^{\dagger}$ Dept.~of Computer Science, University of Toronto,
Toronto, ON M5S 2E4 \\
Corresponding author. Email: mohan.zhao@mail.utoronto.ca \\

\noindent
$\mbox{}^{\ddagger}$ Dept.~of Math. and Computer Science, University of Toronto,
Toronto, ON M5S 2E4 \\
Email: kserkh@math.toronto.edu \\

\vspace{2mm}





\section{Introduction}
The approximation of functions with singularities is a central topic
in approximation theory. One motivating application is the efficient
representation of solutions 
to partial differential equations (PDEs) on nonsmooth geometries or
with discontinuous data, which are known to have 
branch-point singularities. 
Substantial progress has been made in this area,
with perhaps the most common approach being rational approximation and its variants.  
Alternative approaches include the use of 
approximation methods for smooth functions on the real line, applied after a
change of variables to 
ensure a rapid function decay and the translation of 
singularities 
to infinity, and schemes 
that make use of basis functions obtained
through the discretization of certain integral operators. 
If the dominant characteristics of the functions to be approximated
are known a priori,
a class of methods called  
expert-driven approximation can also be used.

Rational approximation is a classical and well-established method for approximating
functions with singularities. Rational functions 
$r(x)=p(x)/q(x)$ are said to be of type $(n,m)$ if $p$ and $q$ are polynomials of degree
at most $n$ and $m$, respectively, and the order of $r(x)$ is defined as 
$\max{(\deg(p),\deg(q))}$.
In 1964, 
Newman proved that there exists an $n$-th order rational
approximation to the function $f(x)=|x|$ 
on $[-1,1]$,
converging uniformly at a rate of 
$O(\exp(-C\sqrt{n}))$ for some constant $C>0$ \cite{newman}---compare this to
the best polynomial approximation to $|x|$ on $[-1,1]$, which can only
achieve a convergence rate no better than $O(n^{-1})$. 
Furthermore, he observed that the same approximation also applies to the
function $f(x)=\sqrt{x}$ and, more generally, to the
functions $f(x)=x^{\alpha}$ on $[0,1]$, where $\alpha >0$.
Notably, Newman's approximation utilizes poles that are clustered exponentially
and symmetrically
around zero along the imaginary axis.

Numerous papers have been 
published on rational approximation methods 
for functions with singularities since Newman's discovery (see, for example, 
\cite{rapid}, \cite{best}, \cite{real}, \cite{minimax}).
The best possible rational approximation is the
so-called minimax approximation, which minimizes
the maximum uniform approximation
error between the function and its rational approximate.
It was shown by Stahl in 1994 that error of the minimax approximation to
$f(x)=x^\alpha$ on $[0,1]$, where $\alpha > 0$, converges at the rate
$O(\exp(-2\pi\sqrt{\alpha n}))$~\cite{best}.
This minimax approximation is, however, not easy to find, 
and is not necessarily unique in the complex plane \cite{unique}.
In practice, by assuming that the singular functions being
approximated fall into certain regularity classes, the poles of a rational
approximation can often be determined a priori,
similar to those employed in Newman's method,
in order to achieve a root-exponential 
convergence rate.
One such method is Stenger's approximation \cite{stenger2}, which
involves interpolating 
the functions at a set of preassigned points
exponentially clustered near the endpoints of the interval, 
using a rational function of type $(2n+2, 2n+1)$ with 
poles 
that are likewise exponentially clustered
at the endpoints. 

While Stenger's method uses explicit interpolation formulas for 
approximating functions falling into certain regularity classes, rational approximations 
can also
be constructed numerically using other representations.
By using Euclidean division and the method of undetermined coefficients, it
is possible to show that any
rational function $r(x)$ of type $(n+m,n)$ with distinct poles can be written in the form
\begin{align}
r(x)=\sum_{j=0}^m b_jx^j +\sum_{j=1}^n \frac{a_j}{x-z_j}.
\label{eq:rx}
\end{align}
In order to approximate functions with branch point singularities, 
lightning methods (\cite{gopal}, \cite{gopal2}) fix the poles in the
representation~(\ref{eq:rx}) a priori to cluster exponentially along rays
in the complex plane, terminating at
the singular points of the function being approximated.
The coefficients $\{a_j\}$ and $\{b_j\}$
are then determined by solving a least squares problem at oversampled points.
It was proved in 
\cite{gopal} that, 
for any sequence of $n$ complex points exhibiting exponential clustering,
with spacing scaling as
$O(n^{-1/2})$ on a logarithmic scale, there exists a sequence of
rational approximations $\{r_n\}$  of type $(n-1,n)$
with poles at these points, which
achieves a root-exponential rate of convergence $O(\exp(-C\sqrt{n}))$
for functions with branch point singularities.

Rather than fixing the poles of a rational function a priori, one can
instead fix the interpolation or support points of a rational
interpolant.
As shown in~\cite{aaa},
any rational function $r(x)$ of type $(n,n)$ which does not have poles at the
points $\{z_j\}$
can be written in the barycentric form 
\begin{align}
r(x)=\sum_{j=1}^{n+1} \frac{w_j f_j}{x-z_j} \Big/ \sum_{j=1}^{n+1} \frac{w_j}{x-z_j},
\end{align}
so that $r(z_j)=f_j$. The 
adaptive Antoulas-Anderson (AAA) algorithm \cite{aaa} is a 
rational function approximation method based on this form, which increases the order
$n$ at each iteration, selecting
the additional collocation point $z_{n+1}$ in a greedy fashion.
Like lightning methods, the weights are found by solving a least-squares problem. Unlike
lightning methods, however, the locations of the singular points of the 
function being approximated do not have to be
known in advance. The AAA method is also root-exponentially convergent, and achieves
a convergence rate close to the minimax rate.

While all of the aforementioned
methods can achieve root-exponential rates of convergence, 
Trefethen et al. \cite{cluster} made a key observation that the constant $C$  
in the rates of convergence $O(\exp(-C \sqrt{n}))$
can be improved for most rational approximation methods by  
employing poles with so-called tapered
exponential clustering around singularities, so that
the poles $\{z_j\}$ cluster like $O(\sqrt{n}-\sqrt{j})$ on
a logarithmic scale.
It was shown in~\cite{resolution} that
lightning
approximations~(\ref{eq:rx}) 
with $m=O(\sqrt{n})$ and with poles
$\{z_j\}$ with tapered exponential clustering attain the minimax
rate for the functions $f(x)=x^\alpha$ on $[0,1]$, where $\alpha > 0$.

Rational approximation can also be applied after a change of variables.  An
approach referred to as reciprocal-log approximation \cite{log} uses
approximations of the form $r(-\log{x})$, where $r(s)$ is an $n$-th order rational function
with poles determined a priori, either lying on a parabolic contour or confluent
at the same point in the complex plane.  Similarly to lightning methods, the coefficients are
determined through  a linear least-squares problem using collocation points
that cluster exponentially around $z=0$.  This method converges at a rate of
$O(\exp(-Cn))$ or $O(\exp(-C n/\log{n}))$ for functions
with branch-point singularities, depending on the form of the
approximation and the function's behaviour in the complex plane.

An alternative approach is to use a combination of a change of variables
and an approximation scheme that 
converges rapidly for smooth functions on the real line. 
By applying smooth transformations to functions
with singularities at the endpoints of
some finite intervals on the real line, they 
can be transformed into rapidly
decaying functions, with the
singularities mapped to the point at infinity. 
After this transformation, 
such functions can be approximated accurately using the Sinc approximation,
by an $n$-term truncated Sinc expansion.
Two primary approaches of this type have been developed: the SE-Sinc 
and DE-Sinc approximations (see, for example, \cite{stenger}, \cite{sinc} and
\cite{develop}). The SE-Sinc approximation combines 
the single-exponential transformation with the Sinc approximation, resulting
in a convergence rate of
$O(\exp(-C\sqrt{n}))$, while the DE-Sinc approximation
combines the double-exponential transformation with
the Sinc approximation, to further improve the convergence rate to 
$O(\exp(-C n/\log{n}))$.

While the aforementioned methods require no special knowledge of the 
singularities being approximated,
a class of methods known as expert-driven approximation can be 
used to leverage 
such information.
For example, one can leverage knowledge of the leading terms in the asymptotic expansion
of the singularity to achieve a smaller approximation error.
This information is often available for the 
solutions of boundary value problems for PDEs
on domain with corners---as 
revealed by Lehman (\cite{lehman}) and Wasow (\cite{wasow}), 
the solutions of the Dirichlet problem for linear second order 
elliptic PDEs in two dimensions have singular expansions of the form
\begin{align}
u(r,\theta) \sim \sum_{k,m,l \ge 0} a_{k,m,l} r^{k+l/\alpha}(r^p\log{r})^m \varphi_{k,m,l}(\theta),
\end{align}
where $\varphi_{k,m,l}$ are smooth functions, 
$r$ is the radial distance from the corner, 
$\pi \alpha$ is the interior angle at the corner (so that $1/\alpha \ge 1/2$), and $p \geq 1$ is an integer.
Many well-developed methods fall under the category of expert-driven 
approximation, such as the method of 
auxiliary mapping (see, for example, \cite{MAM}, \cite{MAM2}),
in which an analytic change of variables is used  
to lessen the singular behaviour of the function,
and enriched approximation methods
(see, for example, \cite{enriched}),  
in which singular basis functions are used to augment 
a conventional basis.
Some examples of enriched approximation
methods include extended/generalized finite element methods
(see, for example, \cite{olson}, \cite{fix}, \cite{extend}), 
enriched spectral and pseudo-spectral methods (see, for
example, \cite{spec}, \cite{gopal}, \cite{pseudo}), and 
integral equation methods using singular basis functions
(see, for example, \cite{corners} and \cite{corners2}).

A much different class of approaches is 
based on the idea
that the functions we are interested in approximating 
often belong to the range of certain integral operators. 
One such method proposed by Beylkin and Monz\'{o}n \cite{expo} 
involves representing a function by a linear combination of
exponential terms with complex-valued exponents and coefficients.
This method is motivated by the observation that many functions admit
representations by exponential integrals over contours in the complex plane, 
which can then be discretized by quadrature.
Instead of starting with a contour integral,
the existence of such representations is only
assumed implicitly, and the exponents (which they also call nodes) are
obtained by finding the roots of a c-eigenpolynomial corresponding to a
Hankel matrix, constructed from uniform samples of the function
over the interval, while the coefficients (or weights) are determined via a
Vandermonde system.
This method can be highly effective for representing functions, though 
we note that their method only minimizes the error at the sample
points, and, for singular functions, they only try to control the error on a
subinterval which excludes the singularities.

In this paper, we present a method for approximating functions with an
endpoint singularity over $[0,1] \subset \mathbb{R}$ or, more generally, a
curve $\Gamma \subset \mathbb{C}$,
where the functions have the form
$f(x)=\int_{a}^{b} x^{\mu} \sigma(\mu) \, d \mu$,
where $0<a<b<\infty$,
$x \in [0,1]$ or $x \in \Gamma$, and 
$\sigma(\mu)$ is some signed Radon measure over $[a,b]$
or some distribution supported on $[a,b]$.
Some examples of such functions are 
$x^c=\int_{a}^{b} x^{\mu} \delta(\mu-c) \, d \mu$
and 
$x^c(\log{x})^m=(-1)^m \int_{a}^{b} x^{\mu} \delta^{(m)}(\mu-c) \, d \mu$,
where $a \leq c \leq b$, $m \in \mathbb{Z}$ and $m \geq 0$.
Our method represents these functions as expansions of the form
$\hat{f}_N(x)=\sum_{j=1}^N \hat c_j x^{t_j}$,
so that $\norm{f-\hat{f}_N}_{L^{\infty}[0,1]} \approx \epsilon$,
where the singular powers 
$t_1$, $t_2$, \ldots,
$t_N$
are determined a priori
based on the 
desired approximation accuracy $\epsilon$ and the values of $a$ and $b$.
The coefficients of the expansion are determined by numerically solving a 
Vandermonde-like collocation
problem
\begin{align}
\begin{pmatrix}
x^{t_1}_1& x^{t_2}_1 & \   \ldots& \ x^{t_N}_1\\
x^{t_1}_2& x^{t_2}_2 & \   \ldots& \ x^{t_N}_2\\
\vdots & \vdots& \  \ddots &   \vdots\\
x^{t_1}_N& x^{t_2}_N & \   \ldots& \ x^{t_N}_N
\end{pmatrix}
\begin{pmatrix}
c_1\\
c_2\\
\vdots\\
c_N
\end{pmatrix}
=
\begin{pmatrix}
f(x_1)\\
f(x_2)\\
\vdots\\
f(x_N)
\end{pmatrix}
\end{align}
for $f(x)$ at the points 
$x_1$, $x_2$, \ldots,
$x_N$,
where the collocation points are likewise determined a priori by 
$\epsilon$, $a$ and $b$.
We both prove and show numerically that, in order to obtain a uniform approximation
error of $\epsilon$, the number of basis functions and collocation points
grows as $N=O(\log{\frac{1}{\epsilon}})$.

Note that our assumption on the form of the functions being approximated resembles the
approach by Stenger in \cite{stenger2}, in which he also assumed that the
functions being approximated belong to some predetermined regularity class.
Our assumption means that our method focuses
on functions $f(x)$ that are 
in the range of the truncated Laplace
transform 
after the change of variable $x=e^{-s}$, 
with
$f(e^{-s})=\int_a^b e^{-s\mu}\sigma{(\mu)} \, d\mu$.
The reciprocal-log approximation \cite{log} shares a similar idea.
This method specializes in approximating functions with branch-point
singularities, such as $f(x) = x^{\alpha}$ on $[0,1]$, where $\alpha>0$, which 
are transformed into decaying exponentials $e^{-s\alpha}$ by 
the same change of variable.
Their approach leverages the fact that certain rational approximations
can be obtained to approximate these decaying exponentials
with an exponential 
rate of convergence. Consequently, $x^{\alpha}$ can be approximated 
with an exponential rate of convergence using
a rational approximation $r(s)$ with the change of variable $s=-\log{x}$.
In contrast, our method relies on the discretization of the integral 
representation of $f(e^{-s})$ through the use of the singular value decomposition of
the truncated Laplace transform. Our procedure yields the quadrature nodes that 
enable us to approximate $f(x)$ using singular powers.
The methodology in 
\cite{expo} also bears certain similarities with our method, 
in that they assume implicit integral representations of the functions, 
with decaying exponential kernels. However, rather than directly
discretizing the integrals or the integral operators, 
they identify the exponential terms and coefficients in their approximations 
through an analysis of  
the singular value decomposition of some Hankel 
matrix constructed from the function values. 

Our method converges exponentially, 
in contrast to rational approximation which converges only at a
root-exponential rate.  When compared to
the DE-Sinc approximation method which requires a large number of
collocation points placed at the both endpoints after applying the smooth
transformation (even when singularities only occur at only one endpoint),
and reciprocal-log approximation which uses many collocation points together
with least squares, our method has a small number of both basis functions
and collocation points, such that the coefficients can be determined via a
square, low-dimensional Vandermonde-like system.
Unlike the method proposed by Beylkin and Monz\'{o}n \cite{expo}, which 
only ensures an accurate approximation at 
equidistant points, our method ensures
a small 
uniform error over the entire interval. 
Compared to expert-driven approximation,
our method does not require any prior knowledge of the singularity types,
besides the values of $a$ and $b$, and the resulting basis functions depend
only on these values, together with the precision $\epsilon$.

The structure of this paper is as follows. \Cref{sec:pre} reviews 
the truncated Laplace transform and the 
truncated singular value decomposition of a matrix. \Cref{sec:NA} demonstrates some numerical
findings about the singular value decomposition of 
the truncated Laplace transform. \Cref{sec:AA} develops the main analytical
tools of this paper.
\Cref{sec:prac} presents some numerical experiments which
provide practical conditions for the 
use of the theorems in \Cref{sec:AA}. \Cref{sec:interp}
shows that functions of the form $f(x) = \int_a^b x^\mu \sigma(\mu)\, d\mu$ 
can be approximated uniformly by expansions in singular powers. 
\Cref{sec:NAEA} shows that the coefficients of such expansions
can be obtained numerically by solving a Vandermonde-like system, 
and provides a bound for the uniform approximation error.
\Cref{sec:ext} illustrates that the previous results can be extended
to the case where the measure is replaced by a distribution. 
\Cref{sec:sum} describes the resulting numerical algorithm 
for approximating functions of the form $f(x) = \int_a^b x^\mu \sigma(\mu)\, d\mu$
by expansions in singular powers. 
Finally, \Cref{sec:NE} presents several numerical
experiments which demonstrate the performance of our algorithm.

\section{Mathematical Preliminaries}
\label{sec:pre}

In this section, we provide some mathematical preliminaries.

\subsection{The Truncated Laplace Transform}
\label{sec:pre1}
Throughout this paper, we utilize the analytical and numerical properties of
the truncated Laplace transform, 
which have been previously presented in 
\cite{laplace}. Here, 
we briefly review the key 
properties.

For a function $f(x) \in L^2[a,b]$, where $0<a<b<\infty$,
the truncated Laplace transform $\mathcal{L}_{a,b}$
is a linear mapping 
$L^2[a,b] \rightarrow L^2[0,\infty)$, defined by the formula
\begin{align}
(\mathcal{L}_{a,b}(f))(x)=\int_{a}^{b} e^{-xt} f(t) \, dt.
\end{align}
We introduce the operator $T_{\gamma}\colon L^2[0,1] \rightarrow L^2[0,\infty)$,
defined by the formula
\begin{align}
\label{eq:Top}
(T_{\gamma}(f))(x)=\int_{0}^{1} e^{-x(t+\frac{1}{\gamma-1})} f(t) \, dt,
\end{align}
so that $T_{\gamma}$ is the truncated Laplace transform shifted
from $L^2[a,b]$ to $L^2[0,1]$, 
where $\gamma=\frac{b}{a}$.
It is clear that $\mathcal{L}_{a,b}$ and $T_{\gamma}$ are compact operators (see, for example~\cite{compact}).

As pointed out in~\cite{laplace}, the singular value decomposition of the operator $T_{\gamma}$
consists of an orthonormal sequence of right singular 
functions $\{u_i\}_{i=0,1,\ldots,\infty} \in L^2[0,1]$, 
an orthonormal sequence of left singular 
functions $\{v_i\}_{i=0,1,\ldots,\infty} \in L^2[0,\infty)$, 
and a discrete sequence of singular values $\{\alpha_i\}_{i=0,1,\ldots,\infty}
\in \mathbb{R}$. 
The operator $T_{\gamma}$ can be rewritten as 
\begin{align}
\label{eq:svdT}
(T_{\gamma}(f))(x)=\sum_{i=0}^{\infty} \alpha_i \Bigl{(} \int_{0}^{1}
u_i(t)f(t) \, dt \Bigr{)} v_i(x),
\end{align}
for any function $f(x) \in L^2[0,1]$.
Note that
\begin{align}
\label{eq:Tu}
T_{\gamma}(u_i)&=\alpha_i v_i,
\end{align}
and
\begin{align}
\label{eq:Tv}
T^*_{\gamma}(v_i)&=\alpha_i u_i,
\end{align}
for all $i=0$, $1$, \ldots,
where $T^*_{\gamma}\colon L^2[0,\infty) \rightarrow L^2[0,1]$ is the adjoint of $T_{\gamma}$, defined by
\begin{align}
\label{eq:Tstar}
(T^*_{\gamma}(g))(t)=\int_{0}^{\infty} e^{-x(t+\frac{1}{\gamma-1})} g(x) \, dx.
\end{align}
Similarly, $T^*_{\gamma}$ can be rewritten as
\begin{align}
\label{eq:svdTstar}
(T^*_{\gamma}(g))(t)=\sum_{i=0}^{\infty} \alpha_i \Bigl{(} \int_{0}^{\infty}
v_i(x)g(x) \, dx \Bigr{)} u_i(t).
\end{align}
Furthermore,
for all $i=0$, $1$, \ldots,
\begin{align}
  \alpha_i > \alpha_{i+1} \geq 0,
\end{align}
and the sequence $\{\alpha_i\}_{i=0,1,\ldots,\infty}$ decays
exponentially fast in $i$, where the decay rate is described in \Cref{thm:alpha}.

Assume that the left singular functions of $\mathcal{L}_{a,b}$ are denoted by
$\tilde{v}_0$, $\tilde{v}_1$, \ldots, and that 
the right singular functions of $\mathcal{L}_{a,b}$ are denoted by
$\tilde{u}_0$, $\tilde{u}_1$, $\ldots$ .
Then, the relations between the singular functions of  
$\mathcal{L}_{a,b}$ and those of $T_{\gamma}$ are given by the formulas  
\begin{align}
u_i(t) &= \sqrt{b-a} \ \tilde{u}_i(a+(b-a)t),
\end{align}
and
\begin{align}
v_i(x) &= \frac{1}{\sqrt{b-a}} \ \tilde{v}_i\biggl{(}\frac{x}{b-a}\biggr{)},
\end{align}
for all $i=0$, $1$, $\ldots$ .
It is observed in~\cite{laplace} that 
$\tilde{v}_0$, $\tilde{v}_1$, $\ldots$ are the eigenfunctions
of the $4$th order differential operator $\hat{D}_{\omega}$,
defined by
\begin{align}
\hspace{-27mm} \Bigl{(}\hat{D}_{\omega}(f)\Bigr{)} (\omega)=
-\frac{\text{d}^2}{\text{d}{\omega}^2}\biggl{(} {\omega}^2 
\frac{\text{d}^2}{\text{d}{\omega}^2} f(\omega) \biggr{)}
+(a^2+b^2) \frac{\text{d}}{\text{d}{\omega}}\bigg{(} {\omega}^2  
\frac{\text{d}}{\text{d}{\omega}}f(\omega)\biggr{)}+(-a^2b^2{\omega}^2 +
2a^2)f(\omega),
\end{align}
where $f \in C^4[0,\infty) \cap L^2[0,\infty)$,
and that
$\tilde{u}_0$, $\tilde{u}_1$, $\ldots$ are the eigenfunctions of the $2$nd order
differential operator $\tilde{D}_t$,
defined by
\begin{align}
\Bigl{(}\tilde{D}_t(f)\Bigr{)} (t)=
\frac{\text{d}}{\text{d}{t}}\biggl{(} (t^2-a^2)(b^2-t^2) 
\frac{\text{d}}{\text{d}t} f(t) \biggr{)}
-2(t^2-a^2)f(t),
\end{align}
where $f\in C^2[a,b]$.
Thus, $\tilde{v}_i$, for all $i=0$, $1$, \ldots, 
can be evaluated by finding the solution to
the differential equation 
\begin{align}
\hspace{-25mm} 
-\frac{\text{d}^2}{\text{d}{\omega}^2}\biggl{(} {\omega}^2 
\frac{\text{d}^2}{\text{d}{\omega}^2} \tilde{v}_i(\omega) \biggr{)}
+(a^2+b^2) \frac{\text{d}}{\text{d}{\omega}}\bigg{(} {\omega}^2  
\frac{\text{d}}{\text{d}{\omega}}\tilde{v}_i(\omega)\biggr{)}+(-a^2b^2{\omega}^2 +
2a^2)\tilde{v}_i(\omega) =\hat{\chi}_i \tilde{v}_i(\omega),
\end{align}
where $\hat{\chi}_i$ is the $i$th eigenvalue of the differential operator
$\hat{D}_{\omega}$.
Similarly, 
$\tilde{u}_i$, for all $i=0$, $1$, \ldots, 
can be evaluated by finding the solution to
the differential equation 
\begin{align}
\frac{\text{d}}{\text{d}{t}}\biggl{(} (t^2-a^2)(b^2-t^2) 
\frac{\text{d}}{\text{d}t} \tilde{u}_i(t) \biggr{)}
-2(t^2-a^2)\tilde{u}_i(t) = \tilde{\chi}_i \tilde{u}_i(t),
\end{align}
where $\tilde{\chi}_i$ is the $i$th eigenvalue of the differential operator
$\tilde{D}_{t}$.
It is known that the singular functions  
$\tilde{v}_i$ and $\tilde{u}_i$ (and thus ${v}_i$ and ${u}_i)$ have exactly $i$ distinct roots,
for all $i=0$, $1$, $\ldots$ .

A procedure for the evaluation of the singular functions and singular values 
of the operator 
$T_{\gamma}$
is described comprehensively in~\cite{laplace} and~\cite{laplace2}.  

The following lemma states 
that, for any function which is analytic and bounded within 
a Bernstein ellipse $E_{\rho}$,
the coefficients in its Chebyshev expansion decay at an exponential rate 
(see Chapter 8 of \cite{nick}). We will use it to prove that the singular values
$\alpha_0$, $\alpha_1$, $\ldots$ of $T_{\gamma}$ decay exponentially.

\begin{lemma}
\label{lem:ak}
Let a function f analytic in $[-1,1]$ be analytically continuable to the open
Bernstein ellipse $E_{\rho}$, where it satisfies $|f(x)|\leq M$ for some $M$. Then its Chebyshev
coefficients satisfy $|a_0| \leq M$ and 
\begin{align}
|a_k|\leq 2M {\rho}^{-k}, \qquad k \geq 1. 
\end{align}
\end{lemma}

The following theorem demonstrates that the singular values of $T_{\gamma}$
decay at an exponential rate, which decreases as $\gamma$ 
increases.
\begin{theorem}
\label{thm:alpha}
Suppose that $\alpha_0$, $\alpha_1$, $\ldots$ are the singular values 
of $T_{\gamma}$, for some $\gamma >1$. Then, $\alpha_n$ 
decays exponentially with $n$, and the decay rate decreases with increasing $\gamma$.
Specifically, for each $c \in (0,1)$,
\begin{align}
\alpha_{n} \leq \sqrt{\frac{\gamma-1}{(1-c)(1-\rho)}}{\rho}^{-\frac{n}{2}},
\end{align}
where 
\begin{align}
\rho=1+\frac{4c}{\gamma-1}+\sqrt{\biggl{(}\frac{4c}{\gamma-1}\biggr{)}^2+\frac{8c}{\gamma-1}}.
\end{align}
\end{theorem}

\begin{proof}
We can view $T_{\gamma,[-1,1]}$ as 
the operator $T_{\gamma}$ shifted from 
$L^2[0,1]$ to $L^2[-1,1]$,
defined by 
\begin{align}
(T_{\gamma,[-1,1]}(f))(x)=\int_{-1}^{1} e^{-x(\frac{t+1}{2}+\frac{1}{\gamma-1})} f(t) \, dt,
\end{align}
while its adjoint $T^*_{\gamma,[-1,1]} \colon L^2[0,\infty)\rightarrow L^2[-1,1]$
is given by 
\begin{align}
(T^*_{\gamma,[-1,1]}(g))(t)=\int_{0}^{\infty} e^{-x(\frac{t+1}{2}+\frac{1}{\gamma-1})} g(x) \, dx.
\end{align}
It follows that the self-adjoint operator 
$S := T^*_{\gamma,[-1,1]}T_{\gamma,[-1,1]} \colon L^2[-1,1]\rightarrow L^2[-1,1]$ is given by
\begin{align}
(S(f))(x)&=\int_{0}^{\infty} e^{-y(\frac{x+1}{2}+\frac{1}{\gamma-1})}
\Bigl{[}\int_{-1}^{1} e^{-y(\frac{t+1}{2}+\frac{1}{\gamma-1})}f(t)\, dt\Bigr{]}  \, dy \notag \\
&=\int_{-1}^{1}\Bigl{[} \int_{0}^{\infty} e^{-y(\frac{x+1}{2}+\frac{1}{\gamma-1})}
e^{-y(\frac{t+1}{2}+\frac{1}{\gamma-1})}\, dy\Bigr{]} f(t) \, dt \notag \\
&=\int_{-1}^{1}\Bigl{[} \int_{0}^{\infty} e^{-y(\frac{x+t}{2}+\frac{\gamma+1}{\gamma-1})}
\, dy\Bigr{]} f(t) \, dt \notag \\
&=\int_{-1}^{1}\frac{2(\gamma-1)}{(x+t)(\gamma-1)+2(\gamma+1)}
f(t) \, dt.
\end{align}
We now let 
\begin{align}
K(x,t)=\frac{2(\gamma-1)}{(x+t)(\gamma-1)+2(\gamma+1)},
\end{align}
so that 
\begin{align}
(S(f))(x)&=\int_{-1}^{1} K(x,t) f(t) \, dt.
\end{align}
Notice that, for each fixed $x \in [-1,1]$, 
$K(x,t)$ has a pole
at $t=-x -2-\frac{4}{\gamma-1}<-1$, where $\gamma \in (1,\infty)$.
Thus, for each fixed $x\in [-1,1]$, $K(x,t)$ is analytic on $[-1,1]$, and admits an analytic continuation to a Berstein ellipse $E_{\rho}$  
with the semi-major axis $a$, where
\begin{align}
a=\frac{1}{2}(\rho+\frac{1}{\rho}).
\end{align}
We observe that
$a=0-(-1-c\frac{4}{\gamma-1})=1+c\frac{4}{\gamma-1}$, for some
$c \in (0,1)$. 
Thus,
\begin{align}
1+c\frac{4}{\gamma-1}=\frac{1}{2}(\rho+\frac{1}{\rho}),
\end{align}
which implies 
\begin{align}
\label{eq:rho}
\rho=1+\frac{4c}{\gamma-1}\pm\sqrt{\biggl{(}\frac{4c}{\gamma-1}\biggr{)}^2+\frac{8c}{\gamma-1}},
\end{align}
where $\gamma \in (1,\infty)$.
Since the definition of $E_{\rho}$ assumes $\rho>1$, we have  
\begin{align}
\rho=1+\frac{4c}{\gamma-1}+\sqrt{\biggl{(}\frac{4c}{\gamma-1}\biggr{)}^2+\frac{8c}{\gamma-1}},
\end{align}
where the parameter $\rho$ decreases with increasing $\gamma$.
It follows that $K(x,z)$ can be expanded as 
\begin{align}
K(x,z)=\sum_{j=0}^{\infty} a_j(x)T_{j}(z),
\end{align}
for $z \in E_{\rho}$.
Noticing that $|K(x,z)|$ attains its maximum value when $z+x=-2-\frac{4c}{\gamma-1}$,
we have
\begin{align}
|K(x,z)|&=\biggl{|}\frac{2(\gamma-1)}{(x+z)(\gamma-1)+2(\gamma+1)}\biggr{|} \notag \\
&\leq \biggl{|} \frac{2(\gamma-1)}{(-2-\frac{4c}{\gamma-1})(\gamma-1)+2(\gamma+1)}\biggr{|} \notag \\
&=\frac{\gamma-1}{ 2-2c},
\end{align}
where $c \in (0,1)$ and $\gamma \in (1,\infty)$,
so $|K(x,z)|$ is uniformly bounded by $\frac{\gamma-1}{ 2-2c}$ for $x\in [-1,1]$ and $z \in E_{\rho}$. 
Thus, \Cref{lem:ak} implies that
\begin{align}
\sup_{-1\leq x \leq 1} |a_j(x)| \leq \frac{\gamma-1}{ 1-c}{\rho}^{-j},
\end{align}
where $\rho$ is given by \Cref{eq:rho}.
For each $n$, we let 
\begin{align}
K_n(x,z)=\sum_{j=0}^{n-1} a_j(x)T_{j}(z),
\end{align}
and we define $S_n \colon L^2[-1,1]\rightarrow L^2[-1,1]$ by
\begin{align}
(S_n(f))(x)
&=\int_{-1}^{1}K_{n}(x,t)
f(t) \, dt.
\end{align}
Since the Chebyshev polynomials are bounded in uniform norm by $1$
and the size of coefficients decay exponentially fast, we have
\begin{align}
\hspace{-10mm} |K(x,z)-K_{n}(x,z)|=\biggl{|}\sum_{j=n}^{\infty}{a_j(x)T_j(z)}\biggr{|}
\leq \sum_{j=n}^{\infty} {|}a_j(x){|} \leq \frac{\gamma-1}{(1-c)(1-\rho)} {\rho}^{-n}.
\end{align}
It follows that
\begin{align}
\norm{S-S_{n}} \leq \frac{2(\gamma-1)}{(1-c)(1-\rho)} {\rho}^{-n},
\end{align}
and we have 
\begin{align}
\lambda_{n} \leq \norm{S-S_{n}} \leq \frac{2(\gamma-1)}{(1-c)(1-\rho)} {\rho}^{-n},
\end{align}
where $\lambda_1\geq \lambda_2\geq  \ldots$ are the singular values of $S$, because
the $n$-th singular value is the optimal error of the 
rank-$(n-1)$ approximation to $S$.
Recalling that
$\alpha_0 \geq \alpha_1 \geq \ldots$ are the singular values of $T_{\gamma}$,
which are the same as the singular values of
$T_{\gamma, [-1,1]}$, and they satisfy
\begin{align}
2\alpha^2_n = \lambda_n,
\end{align}
we have 
\begin{align}
\alpha_{n} \leq \sqrt{\frac{\gamma-1}{(1-c)(1-\rho)}}{\rho}^{-\frac{n}{2}}.
\end{align}
\end{proof}


\subsection{The Truncated Singular Decomposition (TSVD)}

The singular value decomposition (SVD) of a matrix $A\in
\mathbb{R}^{m \times n}$ is defined by 
\begin{align}
A=U\Sigma V^T,
\end{align}
where the left and right matrices $U\in \mathbb{R}^{m \times m}$ and $V \in
\mathbb{R}^{n \times n}$ are orthogonal, and the matrix $\Sigma \in
\mathbb{R}^{m \times n}$ is a diagonal matrix with the singular values of
$A$ on the diagonal, in descending order, so that
\begin{align}
\Sigma=\text{diag} (\sigma_1, \sigma_2, \ldots, \sigma_{\min\{m,n\}}).
\end{align}
Let $r \leq \min\{m,n\}$ denote the rank of $A$, which is equal to the number of
nonzero entries on the diagonal, and suppose that $k \le r$. The $k$-truncated
singular value decomposition ($k$-TSVD) of $A$ is defined as 
\begin{align}
\label{eq:Ak}
A_k=U\Sigma_k V^{T},
\end{align}
where 
\begin{align}
\label{eq:Sigmak}
\Sigma_k=\text{diag} (\sigma_1, \ldots, \sigma_k, 0, 
\ldots, 0) \in \mathbb{R}^{m \times n}.
\end{align}
The pseudo-inverse of $A_k$ is defined by
\begin{align}
A_k^{\dagger}=V\Sigma_k^{\dagger} U^{T}\in \mathbb{R}^{n \times m},
\end{align}
where 
\begin{align}
\Sigma_k^{\dagger}=\text{diag} ({\sigma_1}^{-1}, \ldots, {\sigma_k}^{-1},0,\ldots, 0) 
\in \mathbb{R}^{n \times m}.
\end{align}

The following theorem bounds the sizes of the solution and residual, when a
perturbed linear system is solved using the TSVD.  It follows the same
reasoning as the proof of Theorem~3.4 in~\cite{svd},
and can be viewed as a more explicit version of Lemma 3.3 in  
\cite{coppe}.


\begin{theorem}
\label{thm:tsvd}
Suppose that $A \in \mathbb{R}^{m \times n}$, where $m\ge n$, and let
$\sigma_1 \ge \sigma_2 \ge \cdots \ge \sigma_n$ be the singular values 
of $A$. Let $b \in \mathbb{R}^m$, and suppose that $x\in \mathbb{R}^n$ satisfies
  \begin{align}
Ax = b.
  \end{align}
Let $\epsilon > 0$, and suppose further that 
  \begin{align}
\hat x_k = (A+E)^\dagger_k (b+e),
  \end{align}
where $(A+E)^\dagger_k$ is the pseudo-inverse of the $k$-TSVD of $A+E$, so
that 
  \begin{align}
\hat \sigma_{k} \ge \epsilon \ge \hat \sigma_{k+1},
  \end{align}
where $\hat \sigma_k$ and $\hat \sigma_{k+1}$ are the $k$th and $(k+1)$th
largest singular values of $A+E$, defining $\hat \sigma_{n+1} :=0$, 
and where $E\in \mathbb{R}^{m\times n}$ and $e\in
\mathbb{R}^m$, with $\norm{E}_2 < \epsilon/2$. Then
  \begin{align}
\norm{\hat x_k}_2 \le \frac{1}{\hat \sigma_k} (2\epsilon\norm{x}_2 +
\norm{e}_2) + \norm{x}_2
  \end{align}
and
  \begin{align}
\norm{A \hat x_k - b}_2 \le 5\epsilon \norm{x}_2 + \frac{3}{2}
\norm{e}_2.
  \end{align}

\end{theorem}

\begin{proof}
Let $\sigma_1 \ge \sigma_2 \ge \cdots \ge \sigma_n$ denote the singular values 
of $A$, and let $A_k$ be the $k$-TSVD of $A$. We observe that $A_k x =
b-(A-A_k)x$.  Letting $r_k = (A-A_k)x$ denote the residual, we see that
$\norm{r_k}_2 \le \sigma_{k+1}\norm{x}_2$ and that $b - A_k x = r_k$, defining
$\sigma_{n+1} :=0$.  Let
$x_k = A_k^\dagger b$.  Clearly, $b - Ax_k = r_k$ and $\norm{x_k}_2 \le
\norm{x}_2$.

Let $\hat A := A+E$. We see that
  \begin{align}
\hat x_k &= \hat A_k^\dagger (b+e) \notag \\
&= \hat A_k^\dagger (A x_k + r_k + e) \notag \\
&= \hat A_k^\dagger (\hat A x_k - E x_k + r_k + e) \notag \\
&= \hat A_k^\dagger (- E x_k + r_k + e) 
  + \hat A_k^\dagger \hat A x_k \notag \\
&= \hat A_k^\dagger (- E x_k + r_k + e) 
  + \hat A_k^\dagger \hat A_k x_k.
    \label{eq:xkform}
  \end{align}
Taking norms on both sides and observing that $\hat A_k^\dagger \hat A_k$ is 
an orthogonal projection,
  \begin{align}
\norm{\hat x_k}_2 &\le  \norm{\hat{A}_k^\dagger}_2 ( \norm{E}_2\norm{x_k}_2 +
\norm{e}_2 +\norm{r_k}_2) + \norm{x_k}_2 \notag \\
  &\le  \norm{\hat{A}_k^\dagger}_2 ( \norm{E}_2\norm{x_k}_2 +
\norm{e}_2 +\sigma_{k+1}\norm{x}_2) + \norm{x_k}_2.
  \end{align}
Letting $\hat\sigma_1 \ge \hat\sigma_2 \ge \cdots \ge \hat\sigma_n$ denote
the singular values of $A+E$, we have by the Bauer-Fike Theorem
(see~\cite{nm}) that $\abs{\hat \sigma_j - \sigma_j} \le \norm{E}_2$
for $j=1$, $2$, \ldots, $n$.  Since $\hat \sigma_k \ge \epsilon \ge \hat
\sigma_{k+1}$ and $\norm{E}_2 < \epsilon/2$, we see that
$\sigma_{k+1} < 3\epsilon/2$. Therefore,
  \begin{align}
\norm{\hat x_k}_2 &\le  \frac{1}{\hat \sigma_k} 
  \bigl(\frac{\epsilon}{2}\norm{x}_2 + \norm{e}_2
  + \frac{3\epsilon}{2}\norm{x}_2\bigr) + \norm{x}_2 \notag \\
&= \frac{1}{\hat \sigma_k} (2\epsilon\norm{x}_2 + \norm{e}_2) + \norm{x}_2.
  \end{align}

To bound the residual, we observe that
  \begin{align}
A\hat x_k - b &= A\hat x_k - Ax_k - r_k \notag \\
  &= A(\hat x_k - x_k) - r_k \notag \\
  &= \hat A(\hat x_k - x_k) - E(\hat x_k - x_k) - r_k.
  \end{align}
From~\Cref{eq:xkform}, we have that
  \begin{align}
\hat x_k - x_k = \hat A_k^\dagger(-Ex_k + r_k + e) - (I-\hat{A}_k^\dagger
\hat{A}_k) x_k.
  \end{align}
Combining these two formulas,
  \begin{align}
&\hspace{-21mm} A\hat x_k - b = 
    \hat A\hat A_k^\dagger(-Ex_k + r_k + e) - \hat A(I-\hat{A}_k^\dagger
    \hat{A}_k) x_k- E(\hat x_k - x_k)- r_k \notag \\
&\hspace{-8mm} = \hat A_k \hat A_k^\dagger(-Ex_k + r_k + e) - \hat A(I-\hat{A}_k^\dagger
    \hat{A}_k) x_k- E(\hat x_k - x_k)- r_k \notag \\
  &\hspace{-8mm} = \hat A_k \hat A_k^\dagger(-Ex_k + e) - \hat A(I-\hat{A}_k^\dagger
    \hat{A}_k) x_k- E(\hat x_k - x_k)- (I - \hat A_k \hat A_k^\dagger) r_k.
  \end{align}
Since $\hat A(I-\hat{A}_k^\dagger \hat{A}_k) = (\hat A-\hat
A_k)(I-\hat{A}_k^\dagger \hat{A}_k)$, we see that
  \begin{align}
\hspace{-25mm} A\hat x_k - b = 
    \hat A_k \hat A_k^\dagger(-Ex_k + e) - (\hat A - \hat A_k)(I-\hat{A}_k^\dagger
    \hat{A}_k) x_k- E(\hat x_k - x_k)- (I - \hat A_k \hat A_k^\dagger) r_k.
  \end{align}
Taking norms on both sides and observing that $\hat A_k \hat A_k^\dagger$
and $(I-\hat{A}_k^\dagger \hat{A}_k)$ are orthogonal projections,
  \begin{align}
\hspace*{-5em} \norm{A\hat x_k - b}_2  &\le
  2\norm{E}_2 \norm{x_k}_2 + \norm{e}_2 + \hat\sigma_{k+1}\norm{x_k}_2 +
  \norm{E}_2 \norm{\hat x_k}_2 + \norm{r_k}_2  \notag \\
&\le
  \frac{7}{2}\epsilon\norm{x}_2 + \norm{e}_2
  + \frac{1}{2}\epsilon \norm{\hat x_k}_2  \notag \\
&\le
  5\epsilon\norm{x}_2 + \frac{3}{2}\norm{e}_2.
  \end{align}
\end{proof}

\section{Numerical Tools}
\label{sec:NA}
In this section, we present several numerical experiments
examining some numerical properties of 
the singular value decomposition of the shifted truncated Laplace
transform, $T_{\gamma}$.
We make the following observations:

\begin{enumerate}
\item Note that the singular values of $T_{\gamma}$ decay exponentially,
with the decay rate
decreasing as $\gamma$ increases,
as suggested by \Cref{thm:alpha}.  
The numerical experiments illustrated in \Cref{fig:alpha} show that
the singular values decay at a much greater rate than the bound
provided in \Cref{thm:alpha}. 
    \item \Cref{fig:vnorm1,fig:vnorm2,fig:unorm} show that 
    the $L^{\infty}$-norms of both the left and right singular functions
    and the $L^{1}$-norm of the left singular functions
    are small, for $\gamma \in [2,1250]$.
    \item Suppose that 
    $x_1$, $x_2$, \ldots , $x_n$ are the roots of $v_n(x)$, and that
    $t_1$, $t_2$, \ldots , $t_n$ are the roots of $u_n(t)$. Let the weights
    ${w}_1$, ${w}_2$, \ldots , ${w}_n$ 
    and $\tilde{w}_1$, $\tilde{w}_2$, \ldots , $\tilde{w}_n$ satisfy
    \begin{align}
    \label{eq:weight1defcopy}
    \int_{0}^{\infty} v_i(x) \, dx= 
    \sum_{k=1}^{n} {w}_k v_i(x_k),
    \end{align}
    and
    \begin{align}
    \label{eq:weight2defcopy}
    \int_{0}^{1} u_i(t) \, dt= 
    \sum_{k=1}^{n} \tilde{w}_k u_i(t_k),
    \end{align}
    for all $i=0$, $1$, \ldots , $n-1$. Then the weights are all positive. 
    Moreover, \Cref{fig:vw1,fig:uw} show that 
    $\max_{1\leq k \leq n} \sqrt{w_k}$ and $\norm{\tilde{w}}_1$ 
    are small, for $\gamma \in [2,1250]$.
\item Let 
\begin{align}
A^{\infty}_n &:= \sum_{i=n}^{\infty} \alpha_i \norm{v_i}_{L^{\infty}[0,\infty)}
\norm{u_i}_{L^{\infty}[0,1]}, \label{eq:Ainfty} \\
A^{1}_n &:= \sum_{i=n}^{\infty} \alpha_i \norm{v_i}_{L^{1}[0,\infty)}
\norm{u_i}_{L^{1}[0,1]}, \label{eq:A1}\\
A^{1,\infty}_n &:= \sum_{i=n}^{\infty} \alpha_i \norm{v_i}_{L^{1}[0,\infty)}
\norm{u_i}_{L^{\infty}[0,1]}, \label{eq:A1infty}\\
A^{\infty,1}_n &:= \sum_{i=n}^{\infty} \alpha_i \norm{v_i}_{L^{\infty}[0,\infty)}
\norm{u_i}_{L^{1}[0,1]}, \label{eq:Ainfty1}\\
U_n &:= 
\sum_{i=0}^{n-1} \norm{u_i}_{L^{\infty}[0,1]}, \label{eq:Un}\\
V_n &:= \sum_{i=0}^{n-1} \norm{v_i}_{L^{\infty}[0,\infty)}\label{eq:Vn}.
\end{align}
Numerical experiments illustrated in \Cref{fig:alpha,fig:vnorms,fig:vws,fig:unorms}
imply that 
$A^{\infty}_n$, 
$A^{1}_n$, 
$A^{1,\infty}_n$, 
and
$A^{\infty,1}_n$ are approximately equal to $\alpha_n$ and that $U_n$ and $V_n$ are
small.
Moreover, we observe that
$A^{\infty,1}_n \norm{w}_1 \approx \alpha_n 
\norm{v_n}_{L^{\infty}[0,\infty)}\norm{w}_1$, and the size of
$\norm{v_n}_{L^{\infty}[0,\infty)}\norm{w}_1$ is illustrated in
\Cref{fig:vw2}.
\end{enumerate}
\begin{figure}
  \centering
  \includegraphics[scale=0.47]{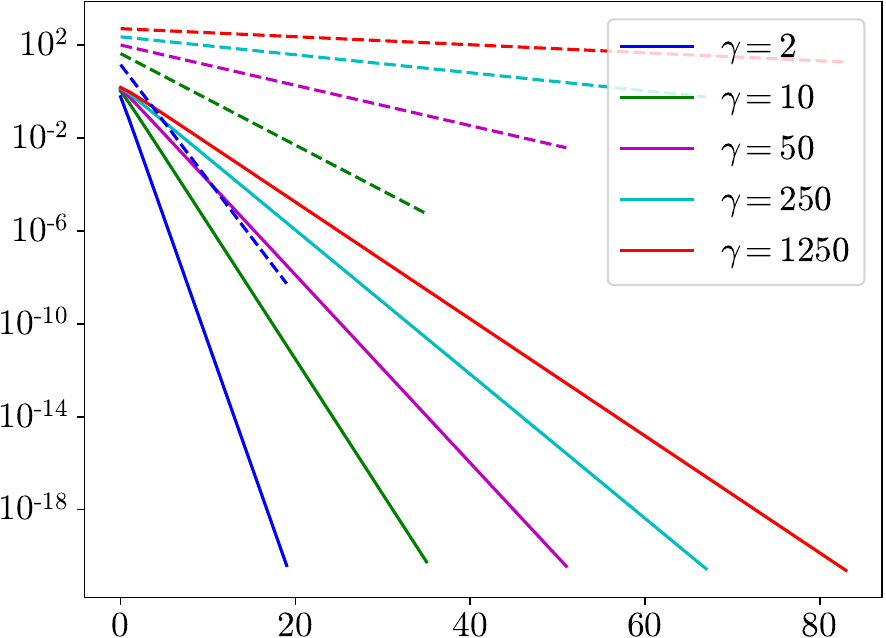}
  \caption{The singular values $\alpha_n$ of $T_{\gamma}$,
          as a function of $n$. The dashed lines indicate the
          bound defined in \Cref{thm:alpha} with $c=0.99$, for the
          corresponding values of $\gamma$.}
  \label{fig:alpha}
\end{figure}

\begin{figure}[!ht]
\centering
\subfloat[\label{fig:vnorm1} $\norm{v_n}_{L^{\infty}[0,\infty)}$,
          as a function of $n$.]{%
\includegraphics[scale=0.47]{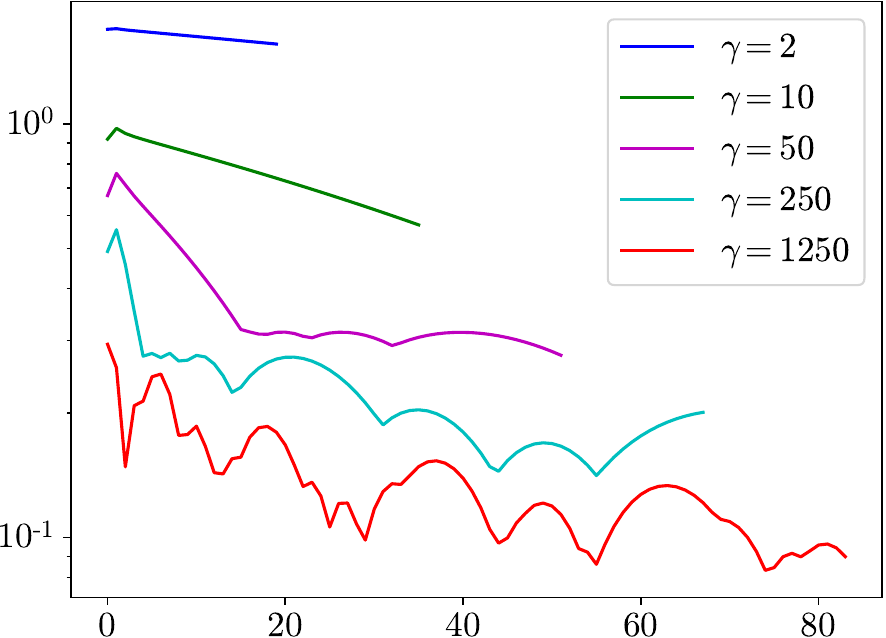}
}    
\subfloat[\label{fig:vnorm2} $\norm{v_n}_{L^{1}[0,\infty)}$,
          as a function of $n$.]{%
\includegraphics[scale=0.47]{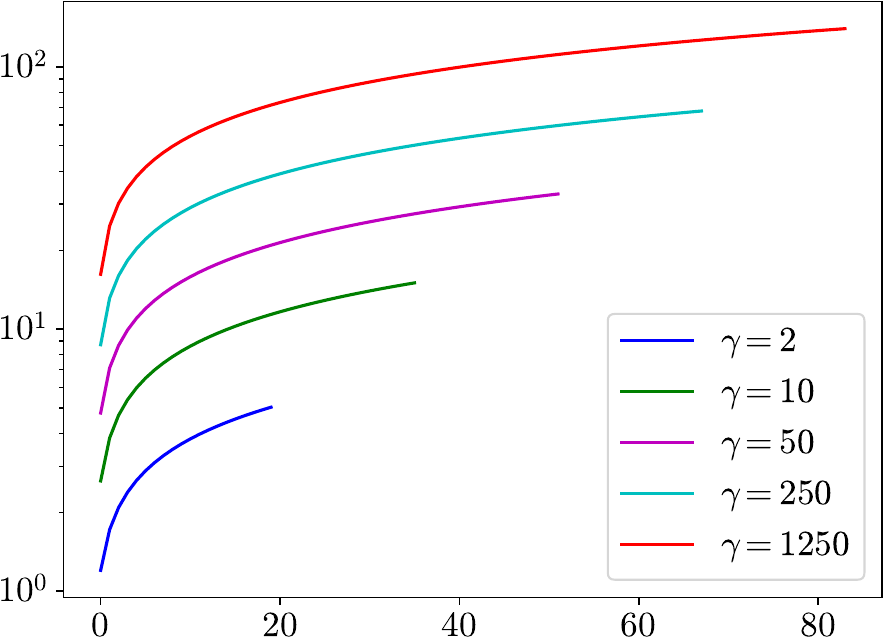}
}
\caption{}
  \label{fig:vnorms}
\end{figure}
\begin{figure}[!ht]
\centering
\subfloat[\label{fig:vw1} $\max_{1\leq k\leq n}{\sqrt{w_k}}$
          corresponding to $n$.]{%
\includegraphics[scale=0.47]{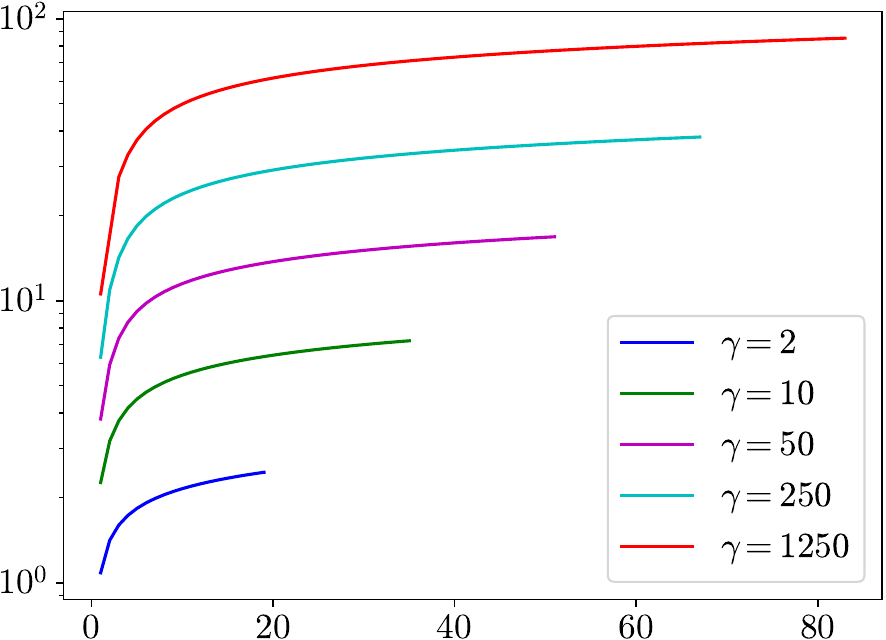}
}    
\subfloat[\label{fig:vw2} $\norm{v_n}_{L^{\infty}[0,\infty)}\cdot \norm{{w}}_{1}$, as
          a function of $n$.]{%
\includegraphics[scale=0.47]{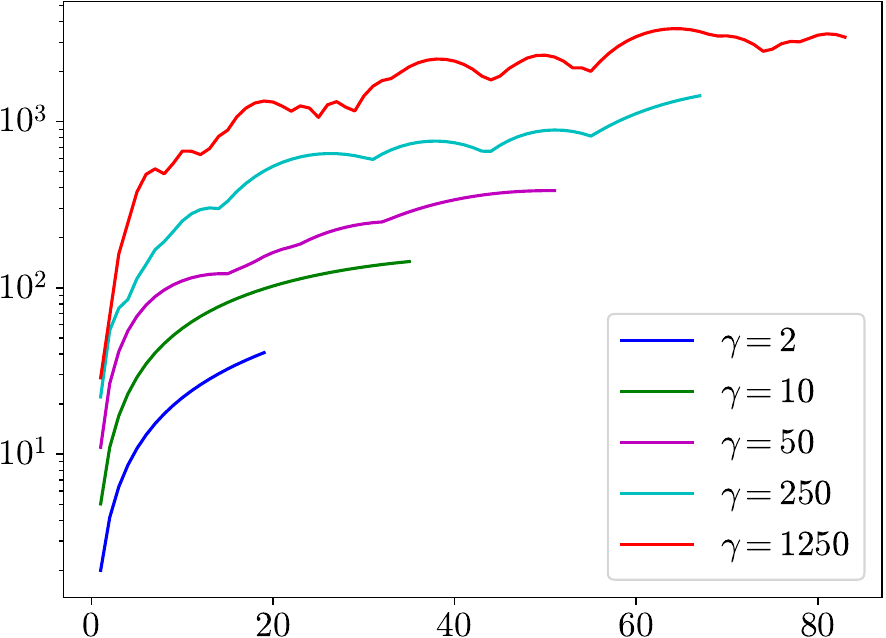}
}
\caption{}
  \label{fig:vws}
\end{figure}
\begin{figure}[!ht]
\centering
\subfloat[\label{fig:unorm} $\norm{u_n}_{L^{\infty}[0,1]}$,
          as a function of $n$.]{%
\includegraphics[scale=0.47]{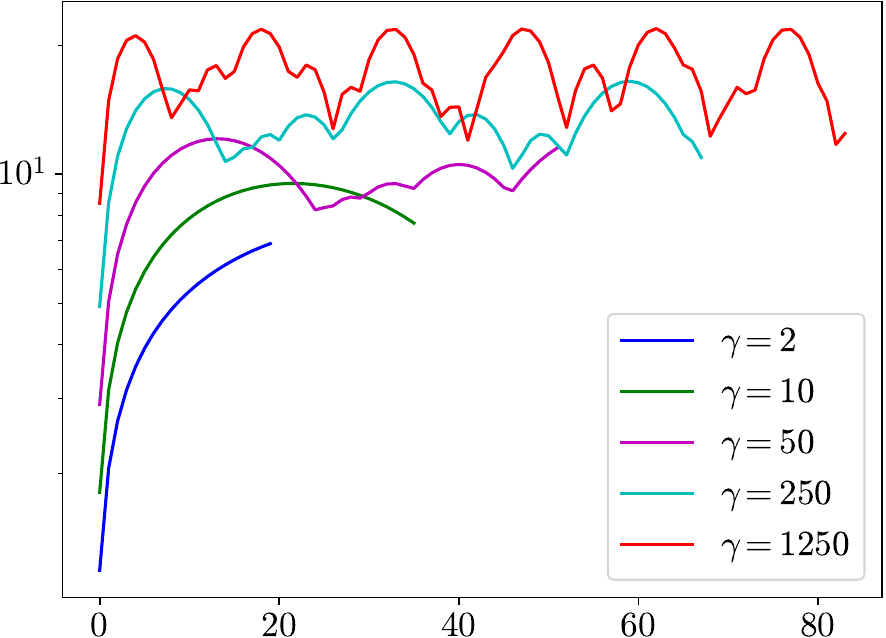}
}    
\subfloat[\label{fig:uw} $\norm{\tilde{w}}_{1}$
          corresponding to $n$.]{%
\includegraphics[scale=0.47]{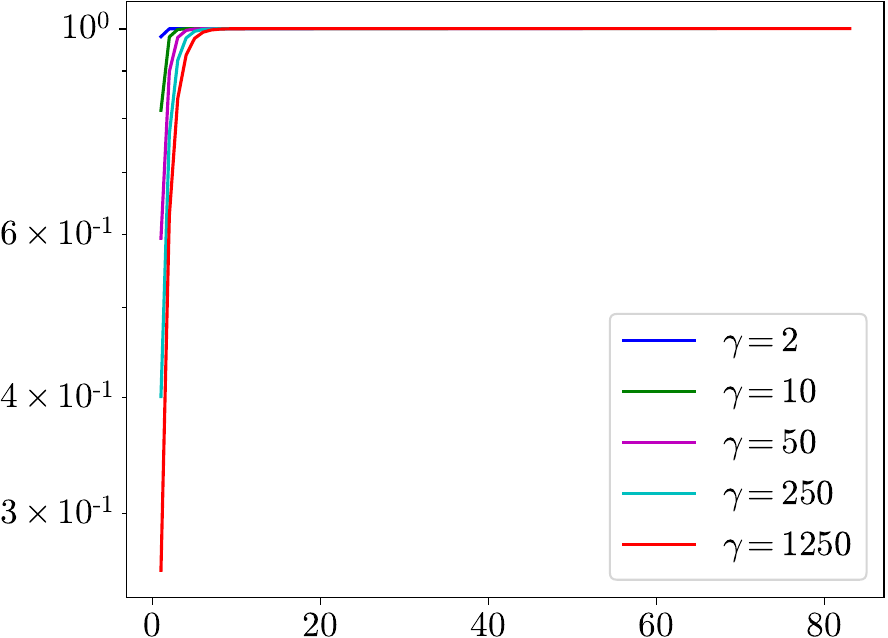}
}
\caption{}
  \label{fig:unorms}
\end{figure}

\section{Analytical Tools}
\label{sec:AA}

In this section, we present the principal analytical tools of this paper.

The following theorem states that the product of any two functions
in the range of the operator $T_{\gamma}$, introduced in \Cref{sec:pre1}, 
can be expressed as $T_{\gamma}$ applied to some $L^{\infty}[0,1]$ function,
after a change of variable.
This result directly follows from the definition of the truncated 
Laplace transform.

\begin{theorem}
Suppose that the functions $p$,
$q$ $\in L^2[0,\infty)$ are
defined by 
\begin{align}
\label{eq:p}
p(x)=\int_{0}^{1} e^{-x(t+\frac{1}{\gamma-1})} \eta(t) \, dt,
\end{align}
and 
\begin{align}
\label{eq:q}
q(x)=\int_{0}^{1} e^{-x(t+\frac{1}{\gamma-1})} \varphi(t) \, dt,
\end{align}
respectively,
for some $\eta$, $\varphi$ $\in L^2[0,1]$, and some $\gamma >1$.
Then, there exists a $\sigma \in L^{\infty}[0,1]$, 
such that
\begin{align}
p(x)\cdot q(x)=\int_{0}^{1} e^{-x(2t+\frac{2}{\gamma-1})} \sigma(t) \, dt.
\end{align}
\end{theorem}

\begin{proof}
For any $p$ and $q$ defined by \Cref{eq:p} and \Cref{eq:q}, we have
\begin{align}
\label{eq:int1}
p(x)\cdot q(x)&=\int_{0}^{1} e^{-x(t+\frac{1}{\gamma-1})} \eta(t) \, dt
    \int_{0}^{1} e^{-x(s+\frac{1}{\gamma-1})} \varphi(s) \, ds \notag \\
&=\int_{0}^{1}\int_{0}^{1}e^{-x(t+s+\frac{2}{\gamma-1})}
  \eta(t) \varphi(s) \, ds \, dt.
\end{align}
Defining a new variable $u=t+s$ and changing the range of integration,  
\Cref{eq:int1} becomes
\begin{align}
\label{eq:pq3}
&\hspace{0mm}p(x)\cdot q(x)=
  \int_{0}^{1} e^{-x(u+\frac{2}{\gamma-1})} \int_{0}^{u}
  \eta(u-s) \varphi(s) \, ds \, du \notag \\
  &\hspace{35mm}+ \int_{1}^{2} e^{-x(u+\frac{2}{\gamma-1})} \int_{u-1}^{1}
  \eta(u-s) \varphi(s) \, ds \, du.
\end{align}
Letting $v=\frac{u}{2}$, we have
\begin{align}
\label{eq:pq2}
&p(x)\cdot q(x)=
  \int_{0}^{\frac{1}{2}} e^{-x(2v+\frac{2}{\gamma-1})} \int_{0}^{2v}
  \eta(2v-s) \varphi(s) \, ds \, 2dv \notag \\
  &\hspace{35mm}+ \int_{\frac{1}{2}}^{1} e^{-x(2v+\frac{2}{\gamma-1})} \int_{2v-1}^{1}
  \eta(2v-s) \varphi(s) \, ds \, 2dv \notag \\
  &\hspace{17mm}=
  \int_{0}^{1} e^{-x(2v+\frac{2}{\gamma-1})}\sigma(v)\, dv, 
\end{align}
where
\begin{align}
 \sigma(v)&=
  2\int_{0}^{2v}
  \eta(2v-s) \varphi(s) \, ds, 
\end{align}
for $v \in [0,\frac{1}{2}]$, and 
\begin{align}
 \sigma(v)&=
  2\int_{2v-1}^{1}
  \eta(2v-s) \varphi(s) \, ds,
\end{align}
for $v \in [\frac{1}{2},1]$.
\end{proof}

\begin{observation}
\label{obs:quadvivj}
Suppose we have nodes $x_1$, $x_2$, \ldots , $x_n$ and
weights $w_1$, $w_2$, \ldots , $w_n$, such that 
\begin{align}
\hspace{-15mm} \Bigl{|}\int_0^{\infty} \int_{0}^{1} e^{-x
(t+\frac{1}{\gamma-1})} \eta(t) \, dt \, dx-
\sum_{j=1}^{n}w_j \int_{0}^{1} e^{-x_j
(t+\frac{1}{\gamma-1})} \eta(t) \, dt
\Bigr{|} \leq \epsilon\norm{\eta}_{L^{\infty}[0,1]},
\end{align}
for any $\eta$ $\in L^{\infty}[0,1]$.
Notice that 
\begin{align}
&\hspace{15mm}
\int_0^{\infty}
p(x)\cdot q(x) \, dx \notag \\
&\hspace{10mm} =
\int_0^{\infty}
p(\frac{u}{2})\cdot q(\frac{{u}}{2})\cdot \frac{1}{2}\, du \notag \\
&\hspace{10mm}=\int_{0}^{\infty}\frac{1}{2}\int_{0}^{1} e^{-u(t+\frac{1}
{\gamma-1})}\sigma(t)\, dt\, du.
\end{align}
Thus,
\begin{align}
\hspace{-23mm}
\Bigl{|}\int_0^{\infty} p(x)\cdot q(x)\, dx
-\sum_{j=1}^{n}\frac{w_j}{2}\cdot p(\frac{x_j}{2})\cdot 
q(\frac{x_j}{2})
\Bigr{|} \leq \frac{1}{2}\epsilon\norm{\sigma}_{L^{\infty}[0,1]} \leq
\epsilon\norm{\eta}_{L^2[0,1]} \norm{\varphi}_{L^2[0,1]}.
\end{align}
\end{observation}

The following theorem shows that the product of
any two functions in the range of $T^*_{\gamma}$
can be expressed as 
$T^*_{\gamma}$ applied to some $L^{\infty}[0,\infty)$ function.

\begin{theorem}
\label{thm:produiuj}
Suppose that the functions $p,q \in L^2[0,1]$ are defined by 
\begin{align}
\label{eq:p2}
p(t)=\int_{0}^{\infty} e^{-x(t+\frac{1}{\gamma-1})} \eta(x) \, dx,
\end{align}
and
\begin{align}
\label{eq:q2}
q(t)=\int_{0}^{\infty} e^{-x(t+\frac{1}{\gamma-1})} \varphi(x) \, dx,
\end{align}
respectively,
for some $\eta$, $\varphi \in L^2[0,\infty)$, and some $\gamma >1$. 
Then, there exists a $\sigma \in L^{\infty}[0,\infty)$,
such that
\begin{align}
\label{eq:f2}
p(t)\cdot q(t)=\int_{0}^{\infty} e^{-x(t+\frac{1}{\gamma-1})} \sigma(x) \, dx.
\end{align}
\end{theorem}

\begin{proof}
For any $p$ and $q$ defined by \Cref{eq:p2} and \Cref{eq:q2}, we have
\begin{align}
\label{eq:int2}
p(t)\cdot q(t)&=
\int_{0}^{\infty} e^{-\omega (t+\frac{1}{\gamma-1})} \eta(\omega) \, d\omega
\int_{0}^{\infty} e^{-x(t+\frac{1}{\gamma-1})} \varphi(x) \, dx \notag  \\
&=\int_{0}^{\infty} \int_{0}^{\infty} e^{-(\omega +x)(t+\frac{1}{\gamma-1})}
\eta(\omega)\varphi(x) \, dx \, d\omega.
\end{align}
Defining $u=\omega+x$ and changing the range of integration, \Cref{eq:int2} becomes  
\begin{align}
\label{eq:pq4}
p(t)\cdot q(t)
&=\int_{0}^{\infty}e^{-u(t+\frac{1}{\gamma-1})} \int_{0}^{u}
\eta(\omega)\varphi(u-\omega) \, d\omega \, du \notag \\
&=\int_{0}^{\infty}e^{-u(t+\frac{1}{\gamma-1})}\sigma(u)\, du,
\end{align}
where 
\begin{align}
\sigma(u)
=\int_{0}^{u}
\eta(\omega)\varphi(u-\omega) \, d\omega,
\end{align}
for $u \in [0,\infty)$.
\end{proof}

\begin{observation}
\label{obs:quaduiuj}
Suppose we have nodes $t_1$, $t_2$, \ldots , $t_n$ and
weights $\tilde{w}_1$, $\tilde{w}_2$, \ldots , $\tilde{w}_n$, such that 
\begin{align}
\label{eq:obs2}
\hspace{-20mm}
\Bigl{|}\int_0^{1} \int_{0}^{\infty} e^{-x
(t+\frac{1}{\gamma-1})} \eta(x) \, dx\, dt-
\sum_{j=1}^{n}\tilde{w}_j \int_{0}^{\infty} e^{-x
(t_j+\frac{1}{\gamma-1})} \eta(x) \, dx
\Bigr{|} \leq \epsilon\norm{\eta}_{L^{\infty}[0,\infty)},
\end{align}
for any $\eta$ $\in L^{\infty}[0,\infty)$.
Since 
\begin{align}
p(t)\cdot q(t) = \int_{0}^{\infty}e^{-x(t+\frac{1}{\gamma-1})}\sigma(x)\, dx,
\end{align}
we have
\begin{align}
\hspace{-18mm} \Bigl{|}\int_0^{1} p(t)\cdot q(t)\, dt
-\sum_{j=1}^{n}\tilde{w}_j\cdot p(t_j)\cdot 
q({t_j})
\Bigr{|} \leq \epsilon\norm{\sigma}_{L^{\infty}[0,\infty)}
\leq \epsilon\norm{\eta}_{L^2[0,\infty)}\norm{\varphi}_{L^2[0,\infty)}.
\end{align}
\end{observation}

Leveraging the multiplication rule established 
earlier, we demonstrate that the following quadrature rule 
accurately integrates the products of 
the kernel of $T_{\gamma}$ and the right singular functions of
$T_{\gamma}$.
\begin{corollary}
\label{col:expu}
Suppose that 
we have a quadrature rule to integrate $\alpha_i u_i \cdot \alpha_j u_j$ 
to within an error of $\epsilon$,
for all $i, j=0$, $1$, \ldots , $n-1$. Suppose further that
$t_1$, $t_2$, \ldots , $t_m$ are the quadrature nodes, and
$\tilde{w}_1$, $\tilde{w}_2$, \ldots , $\tilde{w}_m$ are the quadrature weights.
Then, the error of the quadrature rule applied to functions 
of the form $f(t)=e^{-x (t+\frac{1}{\gamma-1})}u_i(t)$, with $x \in
[0,\infty)$, is bounded by 
\begin{align}
\label{eq:expuerr}
\hspace{0mm}
\frac{\epsilon}{\alpha_i}V_n
+A^{\infty}_n\norm{u_i}_{L^{\infty}[0,1]} \norm{\tilde{w}}_1,
\end{align}
where $V_n$ and $A^{\infty}_n$ are defined in \Cref{eq:Vn} and \Cref{eq:Ainfty},
respectively.
\end{corollary}
\begin{proof}
Since $e^{-x (t+\frac{1}{\gamma-1})}$ can be written as 
\begin{align}
\label{eq:exp}
e^{-x (t+\frac{1}{\gamma-1})}=\sum_{j=0}^{\infty}v_j(x) \alpha_j u_j(t),
\end{align}
we have 
\begin{align}
\label{eq:expu1}
&\hspace{-12mm}  \biggl{|} \int_{0}^{1} e^{-x (t+\frac{1}{\gamma-1})}u_i(t) \, dt -
\sum_{k=1}^{m} \tilde{w}_k e^{-x (t_k+\frac{1}{\gamma-1})} u_i(t_k) \biggr{|} \notag \\
&\hspace{-17mm}=\biggl{|} \int_{0}^{1} \Bigl{(} \sum_{j=0}^{\infty}v_j(x) \alpha_j u_j(t)\Bigr{)}u_i(t) \, dt -
\sum_{k=1}^{m} \tilde{w}_k \Bigl{(}\sum_{j=0}^{\infty}v_j(x) \alpha_j u_j(t_k)\Bigr{)} u_i(t_k) \biggr{|} \notag \\
&\hspace{-17mm}= \biggl{|} \sum_{j=0}^{n-1} \int_{0}^{1} v_j(x) \alpha_j u_j(t) u_i(t) \, dt
+\sum_{j=n}^{\infty} \int_{0}^{1} v_j(x) \alpha_j u_j(t) u_i(t) \, dt \notag \\
&\hspace{-2mm} - \sum_{j=0}^{n-1}
\Bigl{(} \sum_{k=1}^{m} \tilde{w}_k v_j(x) \alpha_j u_j(t_k)\Bigr{)} u_i(t_k)
- \sum_{j=n}^{\infty}
\Bigl{(} \sum_{k=1}^{m} \tilde{w}_k v_j(x) \alpha_j u_j(t_k)\Bigr{)} u_i(t_k) \biggr{|} \notag \\
&\hspace{-17mm}= \biggl{|} \sum_{j=0}^{n-1} \int_{0}^{1} v_j(x) \alpha_j u_j(t) u_i(t) \, dt
- \sum_{j=0}^{n-1}
\Bigl{(} \sum_{k=1}^{m} \tilde{w}_k v_j(x) \alpha_j u_j(t_k)\Bigr{)} u_i(t_k) \notag \\
&\hspace{-2mm}  -\sum_{j=n}^{\infty}
\Bigl{(} \sum_{k=1}^{m} \tilde{w}_k v_j(x) \alpha_j u_j(t_k)\Bigr{)} u_i(t_k) \biggr{|} \notag \\
&\hspace{-17mm} \leq 
\sum_{j=0}^{n-1}{|} v_j(x)| \cdot \frac{\epsilon}{\alpha_i }
+\biggl{|}
\sum_{k=1}^{m} \tilde{w}_k u_i(t_k)\Bigl{(} \sum_{j=n}^{\infty}v_j(x) \alpha_j u_j(t_k)\Bigr{)}
 \biggr{|} \notag \\
&\hspace{-17mm} \leq \frac{\epsilon}{\alpha_i }V_n
+A^{\infty}_n\norm{u_i}_{L^{\infty}[0,1]} \norm{\tilde{w}}_1.
\end{align}
\end{proof}
Suppose that $x_1$, $x_2$, \ldots, $x_m$ and
$w_1$, $w_2$, \ldots, $w_m$ are the nodes and weights of a 
quadrature rule which integrates 
$\alpha_i v_i \cdot \alpha_j v_j$, to within an error of $\epsilon$, 
for all $i, j=0$, $1$, \ldots , $n-1$.
The following theorem shows that, if the left singular
functions $\{v_i\}_{i=0,1,\ldots,n-1}$ of
the operator $T_{\gamma}$, 
are used as interpolation basis,
then, the interpolation matrix
for the nodes $x_1$, $x_2$, \ldots, $x_m$ 
is well conditioned, 
provided that the 
maximum error $\epsilon$ of integrating $\alpha_iv_i\cdot \alpha_jv_j$, 
for $i, j=0$, $1$, \ldots , $n-1$,
is sufficiently small.

\begin{theorem}
\label{thm:interp}
Suppose that we have an $m$-point quadrature rule which integrates
$\alpha_iv_i \cdot \alpha_j v_j$, to within an error of $\epsilon$,
for all $i,j$ $=0$, $1$, \ldots , $n-1$. Suppose further that 
$x_1$, $x_2$, \ldots , $x_m$ are the quadrature nodes, and 
$w_1$, $w_2$, \ldots , $w_m$ are the quadrature weights.
Let the matrix $A \in \mathbb{R}^{m \times n}$ be given by the formula
\begin{align}
\label{eq:A}
A=
\begin{pmatrix}
v_0(x_1)& v_1(x_1)& \   \ldots& \  v_{n-1}(x_1)\\
v_0(x_2)& v_1(x_2)& \   \ldots& \  v_{n-1}(x_2)\\
\vdots & \vdots& \  \ddots &   \vdots\\
v_0(x_m)& v_1(x_m)& \   \ldots& \  v_{n-1}(x_m)
\end{pmatrix}
,
\end{align}
and let the matrix $W$ be the diagonal matrix with diagonal entries 
$w_1$, $w_2$, \ldots , $w_m$. We define the matrix $E = [e_{jk}]$ such that
\begin{align}
\label{eq:defe}
E=I-A^{T} W A.
\end{align}
Then,
\begin{align}
|e_{jk}| \leq  \frac{\epsilon}{\alpha_{j-1} \alpha_{k-1}}.
\end{align}
\end{theorem}
\begin{proof}
From \Cref{eq:defe}, we have  
\begin{align}
e_{jk}=\delta_{jk} - \sum_{l=1}^{m} w_l v_{j-1}(x_l) v_{k-1}(x_l),
\end{align}
where $\delta_{jk}=1$ if $j=k$, and $\delta_{jk}=0$ otherwise.
Then,
\begin{align}
&\hspace{-5mm} |e_{jk}|=\Bigl{|}\delta_{jk} - \sum_{l=1}^{m} w_l \alpha_{j-1}v_{j-1}(x_l) \cdot
\alpha_{k-1} v_{k-1}(x_l)\frac{1}{\alpha_{j-1}\alpha_{k-1}}\Bigr{|} \notag \\
&\hspace{2mm} \leq \Bigl{|}\delta_{jk} - \frac{1}{\alpha_{j-1}\alpha_{k-1}}\int_{0}^{\infty} \alpha_{j-1}v_{j-1}(x)\cdot
\alpha_{k-1} v_{k-1}(x) \, dx \Bigr{|}
+\frac{\epsilon}{\alpha_{j-1}\alpha_{k-1}}
\notag \\
&\hspace{2mm} =
\frac{\epsilon}{\alpha_{j-1}\alpha_{k-1}}.
\end{align}
\end{proof}

The following corollary establishes an upper bound on the norm 
of the pseudo-inverse $A^{\dagger}$ of the matrix $A$ defined in \Cref{eq:A}. 

\begin{corollary}
\label{col:Anorm}
Suppose that we have a collection of quadrature nodes $x_1$, $x_2$, \ldots , $x_m$
and positive 
quadrature weights
$w_1$, $w_2$, \ldots , $w_m$,
which 
integrates $\alpha_iv_i \cdot \alpha_j v_j$ to within an error of $\epsilon
\leq \frac{\alpha_{n}^2}{2{n}}$,
for all $i,j$ $=0$, $1$, \ldots , $n-1$.
Let $A\in \mathbb{R}^{m \times n}$ be the matrix defined in \Cref{eq:A}. Then, 
\begin{align}
\label{eq:Anorm}
\norm{A^{\dagger}}_2
&< \sqrt{2}
\max_{1\leq i \leq m} \sqrt{w_i}, 
\end{align}
where $A^{\dagger}\in \mathbb{R}^{n \times m}$ is the pseudo-inverse of $A$.
\end{corollary}

\begin{proof}
Recalling that $w_1$, $w_2$, \ldots , $w_m$ are positive, we let 
$W^{\frac{1}{2}}$ denote a diagonal matrix with entries $\sqrt{w_1}$,
$\sqrt{w_2}$, \ldots , $\sqrt{w_m}$. 
We define $B$ such that $B=W^{\frac{1}{2}}A$. It follows from \Cref{eq:defe} that
$B^{T}B=I-E$. Since $e_{jk} < \frac{\epsilon}{\alpha_{n}^2}$, for all $j,k=1$,
$2$, \ldots , $n$, we have $\norm{E}_2 < 
\frac{\epsilon}{\alpha_{n}^2}{n}$. 
Let $\tilde{\sigma}_1$, $\tilde{\sigma}_2$, \ldots , $\tilde{\sigma}_n$ denote the 
singular values of $B^{T} B$.
Then, it can be shown that (see Theorem IIIa in~\cite{nm}) 
\begin{align}
|\tilde{\sigma}_j - 1 | \leq \norm{E}_2 < 
\frac{\epsilon}{\alpha_{n}^2}{n},   
\end{align}
for all $j=1$, $2$, \ldots , $n$.
This means that
\begin{align}
1-\frac{\epsilon}{\alpha_{n}^2}{n} < \tilde{\sigma}_j <   
1+\frac{\epsilon}{\alpha_{n}^2}{n}.
\end{align}
Letting $k=\min\{n,m\}$ and $\sigma_1$, $\sigma_2$, \ldots , $\sigma_k$ be the 
singular values of $B$, we have 
\begin{align}
\sqrt{1-\frac{\epsilon}{\alpha_{n}^2}{n}} < {\sigma}_j <   
\sqrt{1+\frac{\epsilon}{\alpha_{n}^2}{n}}.
\end{align}
Letting $B^{\dagger}$ be the pseudo-inverse of $B$, such that $B^{\dagger}B=I$,
since $B=W^{\frac{1}{2}}A$, we have that $A^{\dagger}=B^{\dagger}W^{\frac{1}{2}}$. Thus,
\begin{align}
\label{eq:normA1}
\norm{A^{\dagger}}_2 &\leq \norm{B^{\dagger}}_2\norm{W^{\frac{1}{2}}}_2 \notag \\
&< \frac{1}{\sqrt{1-\frac{\epsilon}{\alpha_{n}^2}{n}}}
\cdot \max_{1\leq i \leq m} \sqrt{w_i}. 
\end{align}
If we have 
$\epsilon \leq \frac{\alpha_{n}^2}{2{n}}$, then 
\Cref{eq:normA1} implies that
\begin{align}
\norm{A^{\dagger}}_2
&<
\sqrt{2} \max_{1\leq i \leq m} \sqrt{w_i}. 
\end{align}
\end{proof}
\section{Selecting the Quadrature Nodes and Weights}
\label{sec:prac}
In this section, we discuss how to construct the quadrature rules 
described in the conditions of the theorems presented in \Cref{sec:AA}.

Suppose that the nodes $t_1$, $t_2$, \ldots , $t_m$
are the roots of $u_m(t)$,
and that the weights $\tilde{w}_1$, $\tilde{w}_2$, \ldots , $\tilde{w}_m$ satisfy
\begin{align}
\label{eq:weight2def}
\int_{0}^{1} u_i(t) \, dt= 
\sum_{k=1}^{m} \tilde{w}_k u_i(t_k),
\end{align}
for all $i=0$, $1$, \ldots , $m-1$.
%
Then, \Cref{eq:svdTstar} and \Cref{eq:weight2def} imply that 
\begin{align}
&\hspace{-15mm}\Bigl{|}\int_0^{1} \int_{0}^{\infty} e^{-x
(t+\frac{1}{\gamma-1})} \eta(x) \, dx\, dt-
\sum_{k=1}^{m}\tilde{w}_k \int_{0}^{\infty} e^{-x
(t_k+\frac{1}{\gamma-1})} \eta(x) \, dx
\Bigr{|} \notag \\
&\hspace{-20mm} = \Bigl{|}\int_0^{1}\sum_{i=0}^{\infty}\alpha_i
\Bigl{(}\int_{0}^{\infty}v_i(x)\eta(x) \, dx \Bigr{)}u_i(t) \,
dt-
\sum_{k=1}^{m}\tilde{w}_k \sum_{i=0}^{\infty}\alpha_i 
\Bigl{(}\int_{0}^{\infty}v_i(x)\eta(x) \, dx \Bigr{)}u_i(t_k)
\Bigr{|} \notag \\ 
&\hspace{-20mm} = \Bigl{|}\int_0^{1}\sum_{i=m}^{\infty}\alpha_i 
\Bigl{(}\int_{0}^{\infty}v_i(x)\eta(x) \, dx \Bigr{)}u_i(t) \,
dt-
\sum_{k=1}^{m}\tilde{w}_k \sum_{i=m}^{\infty}\alpha_i 
\Bigl{(}\int_{0}^{\infty}v_i(x)\eta(x) \, dx \Bigr{)}u_i(t_k)
\Bigr{|} \notag \\ 
&\hspace{-20mm} \leq \Bigl{|}\int_0^{1}\sum_{i=m}^{\infty}\alpha_i 
\Bigl{(}\int_{0}^{\infty}v_i(x)\eta(x) \, dx \Bigr{)}u_i(t) \,
dt \Bigr{|} +
\Bigl{|}\sum_{k=1}^{m}\tilde{w}_k \sum_{i=m}^{\infty}\alpha_i 
\Bigl{(}\int_{0}^{\infty}v_i(x)\eta(x) \, dx \Bigr{)}u_i(t_k)
\Bigr{|} \notag \\ 
&\hspace{-20mm} \leq \norm{\eta}_{L^{\infty}[0,\infty)}{(}A^{1}_m+
A^{1,\infty}_m \norm{\tilde{w}}_1{)},
\end{align}
where $A^1_m$ and $A^{1,\infty}_m$ are defined in \Cref{eq:A1}
and \Cref{eq:A1infty}, respectively.
It follows from \Cref{obs:quaduiuj} that 
\begin{align}
\label{eq:E1}
\hspace{-25mm} E_1:=\max_{0\leq i,j \leq n-1}{\Bigl{|}\int_{0}^{1} \alpha_i u_i(t)\cdot \alpha_j u_j(t) \, dt- 
\sum_{k=1}^{m} \tilde{w}_k \alpha_i u_i(t_k)\cdot \alpha_j
u_j(t_k)\Bigr{|}}
\leq A^1_m+
A^{1,\infty}_m \norm{\tilde{w}}_1.
\end{align}
Since $A^1_m \approx \alpha_m$,
 $A^{1,\infty}_m \approx \alpha_m$,  and $\norm{\tilde{w}}_1$
is small, we have $E_1 \lesssim \alpha_m$.
If $E_1 \leq \alpha^2_n$,
then
\Cref{col:expu}
guarantees that
such a quadrature rule integrates the functions 
$f(t)=e^{-x(t+\frac{1}{\gamma-1})}
u_i(t)$, 
for $i=0$, $1$, \ldots , $n-1$, to an error of approximately
the same size as $\alpha_n$.
Since the singular values $\alpha_i$ decay exponentially,
we see that $E_1 \lesssim \alpha_m \leq \alpha^2_n$ when $m \approx 2n$.
In practice, however, it is unnecessary to take $m$ to be so large.
Numerical experiments for $\gamma=10$, $50$ and
$250$ demonstrate that,
by choosing $m=n$, 
the error of the quadrature rule applied to 
$\alpha_i u_i\cdot \alpha_j u_j$,
for all $i,j=0$, $1$, \ldots , $n-1$,
turns out to be smaller
than $\alpha^2_n$, as shown in
\Cref{fig:uiuj1,fig:uiuj2,fig:uiuj3}.
Thus, it follows from \Cref{col:expu} that  
the error of the quadrature rule applied to 
$f(t)=e^{-x(t+\frac{1}{\gamma-1})}
u_i(t)$, 
for $i=0$, $1$, \ldots , $n-1$, is approximately $\alpha_n$.   
%
%
%
\begin{figure}[!ht]
\centering
\subfloat[\label{fig:uiuj1} A comparison between $E_1$ and $\alpha^2_n$, as a function
  of $n$, where $m=n$.]{%
\includegraphics[scale=0.47]{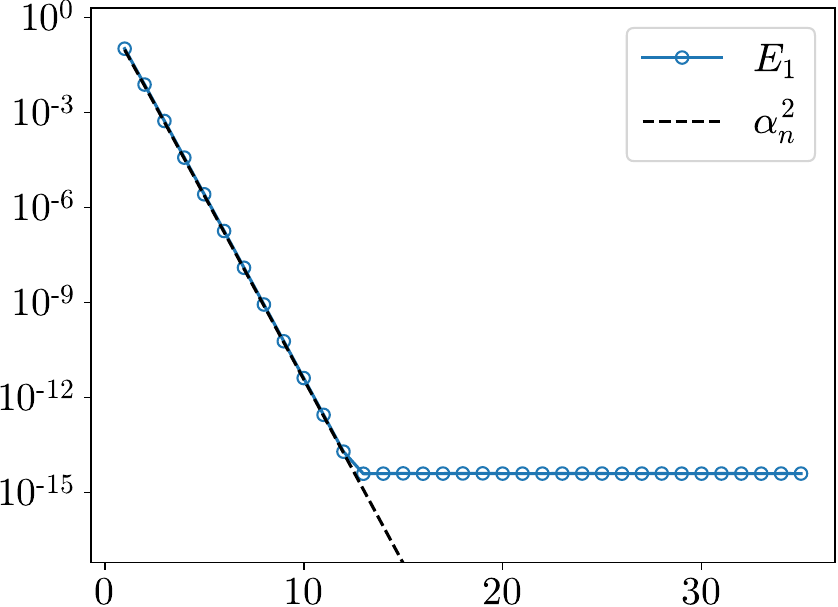}
}    
\hfill
\subfloat[\label{fig:Anorm1} A comparison between $\norm{A^{\dagger}}_2$ and
  $\sqrt{2}\max_{k} \sqrt{w_k}$, as a function
  of $n$, where $m=n$.]{%
\includegraphics[scale=0.47]{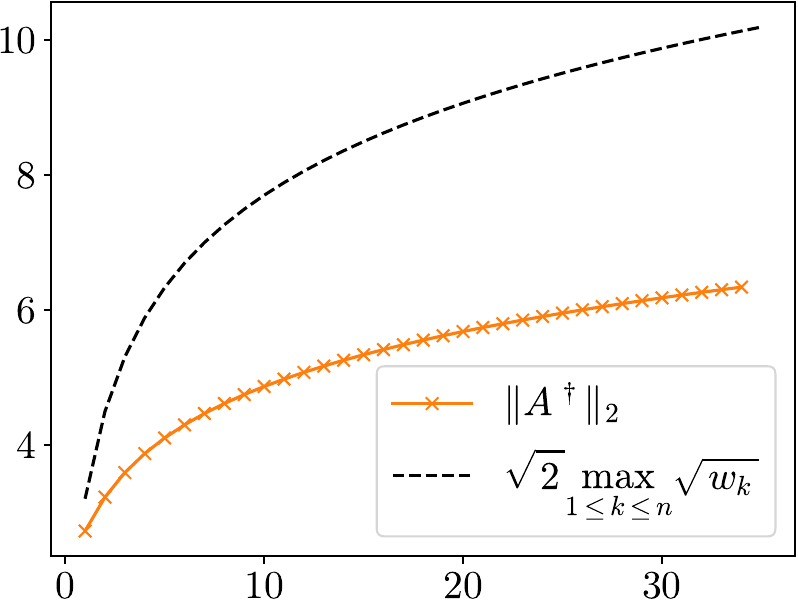}
}
\caption{$\gamma=10$}
  \label{fig:prac1}
\end{figure}
\begin{figure}[!ht]
\centering
\subfloat[\label{fig:uiuj2} A comparison between $E_1$ and $\alpha^2_n$, as a function
  of $n$, where $m=n$.]{%
\includegraphics[scale=0.47]{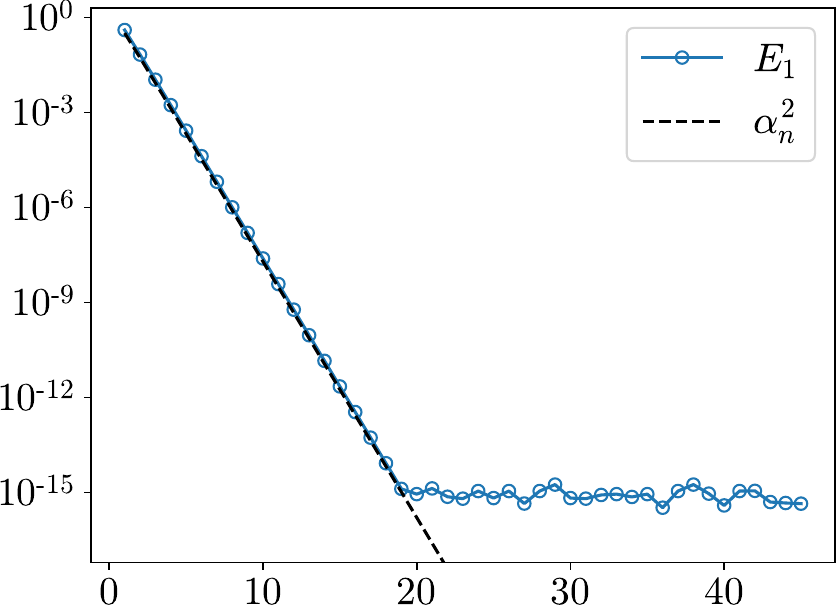}
}    
\hfill
\subfloat[\label{fig:Anorm2} A comparison between $\norm{A^{\dagger}}_2$ and
  $\sqrt{2}\max_{k} \sqrt{w_k}$, as a function
  of $n$, where $m=n$.]{%
\includegraphics[scale=0.47]{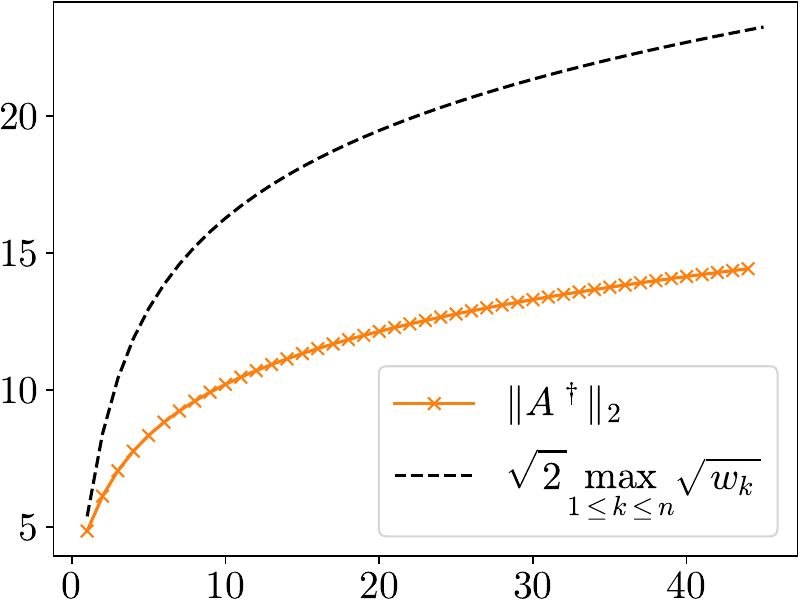}
}
\caption{$\gamma=50$}
  \label{fig:prac2}
\end{figure}
\begin{figure}[!ht]
\centering
\subfloat[\label{fig:uiuj3} A comparison between $E_1$ and $\alpha^2_n$, as a function
  of $n$, where $m=n$.]{%
\includegraphics[scale=0.47]{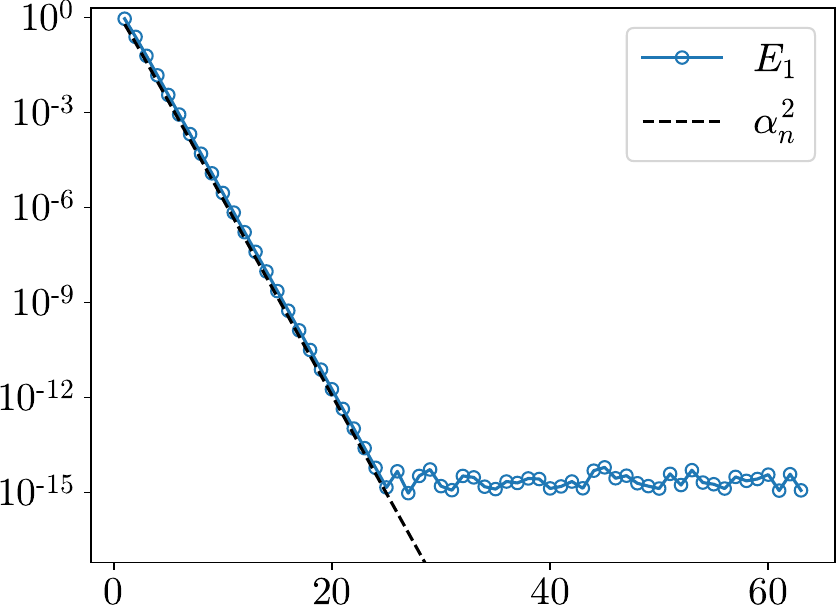}
}    
\hfill
\subfloat[\label{fig:Anorm3} A comparison between $\norm{A^{\dagger}}_2$ and
  $\sqrt{2}\max_{k} \sqrt{w_k}$, as a function
  of $n$, where $m=n$.]{%
\includegraphics[scale=0.47]{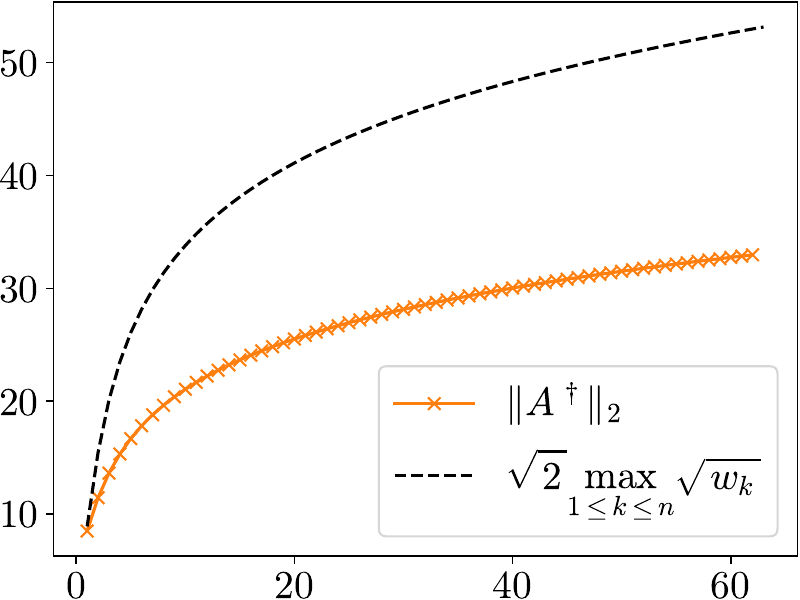}
}
\caption{$\gamma=250$}
  \label{fig:prac3}
\end{figure}

Suppose now that the nodes $x_1$, $x_2$, 
\ldots , $x_m$ are the
roots of $v_m(x)$, and the weights $w_1$, $w_2$, \ldots , $w_m$
satisfy 
\begin{align}
\label{eq:weight1def}
\int_{0}^{\infty} v_i(x) \, dx= 
\sum_{k=1}^{m} {w}_k v_i(x_k),
\end{align}
for all $i=0$, $1$, \ldots , $m-1$.
Then, \Cref{eq:svdT} and \Cref{eq:weight1def} imply that 
\begin{align}
&\hspace{-15mm} \Bigl{|}\int_0^{\infty} \int_{0}^{1} e^{-x
(+\frac{1}{\gamma-1})} \eta(t) \, dt \, dx-
\sum_{k=1}^{m}w_k \int_{0}^{1} e^{-x_k
(t+\frac{1}{\gamma-1})} \eta(t) \, dt
\Bigr{|} \notag \\
&\hspace{-20mm} = \Bigl{|}\int_0^{\infty} \sum_{i=0}^{\infty} 
\alpha_i\Bigl{(} \int_{0}^1 u_i(t)\eta(t) \, dt \Bigr{)}v_i(x)\, dx-
\sum_{k=1}^{m}w_k \sum_{i=0}^{\infty} 
\alpha_i\Bigl{(} \int_{0}^1 u_i(t)\eta(t) \, dt\Bigr{)} v_i(x_k)
\Bigr{|} \notag \\
&\hspace{-20mm} = \Bigl{|}\int_0^{\infty} \sum_{i=m}^{\infty} 
\alpha_i\Bigl{(} \int_{0}^1 u_i(t)\eta(t) \, dt \Bigr{)}v_i(x)\, dx-
\sum_{k=1}^{m}w_k \sum_{i=m}^{\infty} 
\alpha_i\Bigl{(} \int_{0}^1 u_i(t)\eta(t) \, dt\Bigr{)} v_i(x_k)
\Bigr{|} \notag \\
&\hspace{-20mm} \leq \Bigl{|}\int_0^{\infty} \sum_{i=m}^{\infty} 
\alpha_i\Bigl{(} \int_{0}^1 u_i(t)\eta(t) \, dt \Bigr{)}v_i(x)\, dx \Bigr{|}+
\Bigl{|} \sum_{k=1}^{m}w_k \sum_{i=m}^{\infty} 
\alpha_i\Bigl{(} \int_{0}^1 u_i(t)\eta(t) \, dt\Bigr{)} v_i(x_k)
\Bigr{|} \notag \\
&\hspace{-20mm} \leq
\norm{\eta}_{L^{\infty}[0,1]}{(}A^1_m +
 A^{\infty,1}_m \norm{w}_1{)}, 
\end{align}
where $A^1_m$ and $A^{\infty,1}_m$ are defined in \Cref{eq:A1}
and \Cref{eq:Ainfty1}, respectively.
It follows from \Cref{obs:quadvivj} that 
\begin{align}
\label{eq:E2}
\hspace{-27mm} E_2:=\max_{0\leq i,j \leq n-1}\Bigl{|}\int_{0}^{\infty} \alpha_i v_i(x)\cdot \alpha_j v_j(x) \, dx- 
\sum_{k=1}^{m} \frac{{w}_k}{2} \alpha_i v_i(\frac{x_k}{2})\cdot \alpha_j
v_j(\frac{x_k}{2})\Bigr{|}
\leq
A^1_m +
 A^{\infty,1}_m \norm{w}_1.
\end{align}
Since $A^1_m \approx \alpha_m$,
 $A^{\infty,1}_m \approx \alpha_m$,
 and $A^{\infty,1}_m\norm{{w}}_1 \approx \alpha_m \norm{v_m}_{L^{\infty}{[0,\infty)}}\norm{w}_1$,
we have $E_2 \lesssim \alpha_m(1+\norm{v_m}_{L^{\infty}{[0,\infty)}}\norm{w}_1)$,
with the size of $\norm{v_m}_{L^{\infty}{[0,\infty)}}\norm{w}_1$ illustrated 
in \Cref{fig:vw2}.
If $E_2 \leq \frac{\alpha^2_n}{2{n}}$,
then \Cref{col:Anorm} guarantees that the norm 
of $A^\dagger \in \mathbb{R}^{n \times m}$
achieves the bound specified in \Cref{eq:Anorm}.
We see that $E_2 \lesssim \alpha_m(1+\norm{v_m}_{L^{\infty}{[0,\infty)}}\norm{w}_1) \leq \frac{\alpha^2_n}{2{n}}$ 
when $m \approx 2n$.
However, instead of choosing $m$ to be so large,
we can once again take $m=n$,
and use the nodes $x_k$ and the weights $w_k$ rather than
$x_k/2$ and 
$w_k/2$.
Unlike the error of the quadrature rule in \Cref{eq:weight2def} applied to
$\alpha_iu_i\cdot \alpha_j u_j$,
for $i,j=0$, $1$, \ldots , $n-1$,
which is, in practice,
less than $\alpha^2_n$, the error of the quadrature rule in \Cref{eq:weight1def} applied to 
$\alpha_iv_i\cdot \alpha_j v_j$, 
for $i,j=0$, $1$, \ldots , $n-1$,
lies somewhere between $\alpha^2_n$ and $\alpha_n$. 
However, we observe that the special structure of $A \in \mathbb{R}^{n \times n}$
allows the norm of $A^\dagger$ to
still
attain the bound specified in 
\Cref{eq:Anorm}. The results for $\gamma=10$, $50$ and $250$ are
shown in \Cref{fig:Anorm1,fig:Anorm2,fig:Anorm3}, respectively. 

\begin{remark}
It is worth emphasizing that the choice of quadrature nodes is not unique. Any set of
quadrature nodes with corresponding weights that satisfy 
\Cref{eq:weight2def} or \Cref{eq:weight1def} can be employed for our purposes.
In this paper, we choose the roots of $u_m$ and $v_m$ to be the quadrature nodes,
since the associated weights are positive and reasonably small, which we have shown in 
\Cref{sec:NA}.
\end{remark}

\section{Approximation by Singular Powers}
\label{sec:interp}
In this section, we present a method for approximating a function 
of the form 
\begin{align}
\label{eq:repf}
f(x)=\int_{a}^{b} x^{\mu} \sigma(\mu) \, d \mu, \qquad x\in [0,1],
\end{align}
for some signed Radon measure $\sigma(\mu)$,
using a basis of  
$\{x^{t_j}\}_{j=1}^{N}$ for some specially chosen points 
$t_1$, $t_2$, \ldots , $t_N \in [a,b]$.
Our approach involves approximating $f$ by the left
singular functions of $T_{\gamma}$,
and then discretizing the integral representation of   
these left singular 
functions in the form of 
$\{x^{t_j}\}_{j=1}^{N}$.

In the following theorem,
we establish the existence of such an approximation,
and quantify its approximation error.

\begin{theorem}
\label{thm:err}
Let $f$ be a function of the form \Cref{eq:repf}. 
Suppose that
$\tilde{t}_1$, $\tilde{t}_2$, \ldots , $\tilde{t}_N$
and $\tilde{w}_1$, $\tilde{w}_2$, \ldots , $\tilde{w}_N$
are the quadrature nodes and weights of a quadrature rule 
such that $E_1 \leq \alpha^2_n$, where 
$E_1$ is defined in \Cref{eq:E1}. Let $t_j = a+(b-a)\tilde{t}_j$,
for all $j=1$, $2$, \ldots , $N$.
Then, there exists a vector $c \in \mathbb{R}^{N}$ such that the function 
\begin{align}
{f}_N(x)=\sum_{j=1}^{N} c_j x^{t_j},
\end{align}
satisfies
\begin{align}
\label{eq:ferror}
&\hspace{-3mm} \norm{f-f_N}_{L^{\infty}[0,1]} \leq 
|\sigma|\alpha_n
\Bigl{(} \frac{A^{\infty}_n}{\alpha_n}+U_n V_n+
\frac{A^{\infty}_n}{\alpha_n}{U_n}^2\norm{\tilde{w}}_1\Bigr{)},
\end{align}
where 
$A^{\infty}_n$, $U_n$ and $V_n$ are defined in
\Cref{eq:Ainfty}, \Cref{eq:Un} and \Cref{eq:Vn}, respectively,
and the norm of the coefficient vector $c$ is bounded by 
\begin{align}
\label{eq:normchat}
\hspace{-5mm} \norm{{{c}}}_2 \leq \norm{\tilde{w}}_1 |\sigma| \cdot {U_n}^2.
\end{align}
\end{theorem}
\begin{proof}
Substituting $\omega=-\log{x}$ into \Cref{eq:repf}, we have 
\begin{align}
f(e^{-\omega})&=\int_{a}^{b} e^{-\omega\mu} \sigma(\mu) \, d\mu,\notag \\
&=\int_{0}^{1} e^{-\tilde{\omega}(\bar{\mu}+\frac{1}{\gamma-1})}
{(b-a)} \sigma((b-a)\bar{\mu}+a) \, d\bar{\mu},
  \qquad \tilde{\omega} \in [0,\infty),
\end{align}
where $\bar{\mu}=\frac{\mu-a}{b-a}$, and
$\tilde{\omega}=(b-a)\omega$.
Since $\{\alpha_i\}_{i=0,1,\ldots,\infty}$ decays exponentially, we 
truncate the SVD of the operator $T_{\gamma}$ after $n$ terms and obtain 
\begin{align}
e^{-\omega(t+\frac{1}{\gamma-1})}&=\sum_{i=0}^{\infty} v_i(\omega)\alpha_i u_i(t)
\approx \sum_{i=0}^{n-1} v_i(\omega)\alpha_i u_i(t).
\end{align}
Then,
we construct
the approximation $\tilde{f}$ to $f$, defined by
\begin{align}
\label{eq:exact}
\tilde{f}(e^{-{{\omega}}})=\sum_{i=0}^{n-1}
\alpha_i \biggl{(} \int_{0}^{1}
u_i(\bar{\mu}){(b-a)} \sigma((b-a)\bar{\mu}+a) \, d\bar{\mu}
\biggr{)} v_i(\tilde{\omega}).
\end{align}
Thus,
\begin{align}
\label{eq:approx1}
\hspace{5mm} \tilde{f}(x)&=\sum_{i=0}^{n-1}\alpha_i \biggl{(} \int_{0}^{1}
u_i(\bar{\mu}){(b-a)} \sigma((b-a)\bar{\mu}+a) \, d\bar{\mu}
\biggr{)} v_i(-(b-a)\log x)\notag \\
&=\sum_{i=0}^{n-1}\tilde{c}_i\alpha_i v_i(-(b-a)\log x), 
\end{align}
for $x \in [0,1]$,
where
\begin{align}
\label{eq:deftc}
\tilde{c}_i=\int_{0}^{1} u_i(\bar{\mu})
(b-a) \sigma((b-a)\bar{\mu}+a)\, d\bar{\mu}. 
\end{align}
We observe that
\begin{align}
\label{eq:normct2}
|\tilde{c_i}| \leq |{\sigma}|\cdot
\norm{u_i}_{L^{\infty}[0,1]}.
\end{align}
Thus,
\begin{align}
\label{eq:diff1}
\norm{f-\tilde{f}}_{L^{\infty}[0,1]}= 
\bbnorm{\sum_{i=n}^{\infty} \tilde{c}_i{\alpha}_i
v_i(-(b-a)\log{x})}_{L^{\infty}[0,1]}
\leq |\sigma|A^{\infty}_n,
\end{align}
where $A^{\infty}_n$ is defined in \Cref{eq:Ainfty} and $A^{\infty}_n \approx \alpha_n$.
According to \Cref{eq:Tu}, we have
\begin{align}
\alpha_iv_i(\omega)=\int_{0}^{1}e^{-\omega (t+\frac{1}{\gamma-1})} u_i(t) \, dt.
\end{align}
Since
\begin{align}
 E_1=\max_{0\leq i,j \leq n-1} \Bigl{|} \int_{0}^{1}
 \alpha_iu_i(t) \alpha_j u_j(t)\,dt- 
\sum_{l=1}^{N} \tilde{w}_l\alpha_i 
u_i(\tilde{t}_l)\alpha_j u_j(\tilde{t}_l)\Bigr{|} 
\leq {\alpha_n}^2,
\end{align}
for all $i,j$ $=0$, $1$, \ldots , $n-1$,
it follows from \Cref{col:expu} that  
\begin{align}
&\hspace{0mm} \tilde{E}_i:= \Bigl{|}\int_{0}^{1}e^{-\omega (t+\frac{1}{\gamma-1})} u_i(t) \, dt-
\sum_{l=1}^{N} \tilde{w}_l e^{-\omega (\tilde{t}_l+\frac{1}
{\gamma-1})} u_i(\tilde{t}_l)\Bigr{|} \notag \\
& \hspace{5mm} \leq \alpha_{n}V_n
+A^{\infty}_n\norm{u_i}_{L^{\infty}[0,1]} \norm{\tilde {w}}_1,
\end{align}
where $A^{\infty}_n$ and $V_n$ are defined in \Cref{eq:Ainfty}
and \Cref{eq:Vn}, respectively.

Recalling \Cref{eq:approx1}, we have
\begin{align}
\tilde{f}(x)
=\sum_{i=0}^{n-1}\tilde{c}_i\alpha_i v_i(-(b-a)\log x),
\end{align}
so
\begin{align}
\Bigl{|} \tilde{f}(e^{-\frac{\tilde{\omega}}{b-a}}) - \sum_{i=0}^{n-1}\tilde{c}_i\sum_{j=1}^{N}\tilde{w}_j
e^{-\tilde{\omega} (\tilde{t}_j+\frac{1}{\gamma-1})} u_i(\tilde{t}_j) \Bigr{|}
\leq \sum_{i=0}^{n-1}\abs{\tilde{c}_i} \tilde{E}_i,
\end{align}
which means that
\begin{align}
\label{eq:rep1}
\Bigl{|} \tilde{f} (e^{-\frac{\tilde{\omega}}{b-a}})-
\sum_{i=0}^{n-1}\tilde{c}_i\sum_{j=1}^{N}\tilde{w}_j
e^{-\tilde{\omega} \bar{t}_j} u_i(\bar{t}_j-\frac{1}{\gamma-1}) \Bigr{|}
\leq \sum_{i=0}^{n-1}\abs{\tilde{c}_i} \tilde{E}_i,
\end{align}
where $\tilde{\omega}=-(b-a)\log x$ and
$\bar{t}_j=\tilde{t}_j+\frac{1}{\gamma-1}$, for all $j=1$, $2$, \ldots , $N$.
Substituting $e^{-\frac{\tilde{\omega}}{b-a}}=x$ into \Cref{eq:rep1}, 
we define the approximation $f_N$
to $\tilde{f}$, 
\begin{align}
\label{eq:rep2}
&\hspace{-13mm}f_N(x)=
\sum_{i=0}^{n-1} \tilde{c}_i\sum_{j=1}^{N}\tilde{w}_j
u_i(\bar{t}_j-\frac{1}{\gamma-1})x^{(b-a)\bar{t}_j}, \notag \\
&\hspace{-4mm} :=\sum_{j=1}^{N} {{c}}_j x^{(b-a)\bar{t}_j}, \qquad x\in [0,1],
\end{align}
where ${{c}}_j =\tilde{w}_j\sum_{i=0}^{n-1} \tilde{c}_i
u_i(\bar{t}_j-\frac{1}{\gamma-1})$, for $j=1$, $2$, \ldots , $N$.
Letting $t_j=(b-a)\bar{t}_j$, we have $t_j=(b-a)\tilde{t}_j+a$,
for $j=1$, $2$, \ldots , $N$.
\Cref{eq:rep2} and \Cref{eq:normct2} imply that
\begin{align}
\hspace{-20mm}\norm{{{c}}}_2
&\leq \norm{\tilde{w}}_1\cdot \max_{1\leq j \leq N} \Bigl{|}
\sum_{i=0}^{n-1} \tilde{c}_i u_i(\bar{t}_j -\frac{1}{\gamma-1})\Bigr{|} \notag \\
&\leq \norm{\tilde{w}}_1\cdot 
\sum_{i=0}^{n-1} {|}\tilde{c}_i {|} \norm{u_i}_{L^{\infty}[0,1]} \notag \\
&\leq \norm{\tilde{w}}_1 {|}\sigma{|} \cdot
\sum_{i=0}^{n-1} \norm{u_i}^2_{L^{\infty}[0,1]} \notag \\
&\leq \norm{\tilde{w}}_1 |\sigma|
\cdot {U_n}^2,
\end{align}
where $U_n$ is defined in \Cref{eq:Un}.
The approximation error of $f_N$ to $\tilde{f}$ can be bounded by 
\begin{align}
\label{eq:diff2}
&\hspace{-10mm} \norm{\tilde{f}-{f}_N}_{L^{\infty}[0,1]}\leq 
{\sum_{i=0}^{n-1}\abs{\tilde{c}_i}\tilde{E}_i} \notag \\
&\hspace{17mm} \leq \sum_{i=0}^{n-1} | {\sigma}| \norm{u_i}_{L^{\infty}[0,1]} 
(\alpha_n V_n +A^{\infty}_n \norm{u_i}_{L^{\infty}[0,1]} \norm{\tilde{w}}_1) \notag \\
&\hspace{17mm} \leq |\sigma|\cdot (\alpha_n  U_n V_n+
A^{\infty}_n \sum^{n-1}_{i=0}\norm{u_i}^2_{L^{\infty}[0,1]} \norm{\tilde{w}}_1) \notag \\
&\hspace{17mm} \leq |\sigma|\cdot (\alpha_n  U_n V_n+
A^{\infty}_n {U_n}^2 \norm{\tilde{w}}_1).
\end{align}
Thus, we obtain the bound on the approximation error of $f_N$ to $f$ as
\begin{align}
\label{eq:fNf}
&\hspace{-3mm} \norm{f-f_N}_{L^{\infty}[0,1]} \leq 
\norm{f-\tilde{f}}_{L^{\infty}[0,1]}+  
\norm{\tilde{f}-f_N}_{L^{\infty}[0,1]}  \notag \\
&\hspace{23mm}\leq |\sigma|\alpha_n
\Bigl{(} \frac{A^{\infty}_n}{\alpha_n}+U_n V_n+
\frac{A^{\infty}_n}{\alpha_n}{U_n}^2\norm{\tilde{w}}_1\Bigr{)},
\end{align}
which is approximately equal to $|\sigma|\alpha_n$, since $A^{\infty}_n \approx \alpha_n$,
and $U_n$, $V_n$ and $\norm{\tilde{w}}_1$ are small.
\end{proof}

\section{Numerical Approximation and Error Analysis}
\label{sec:NAEA}
In the previous section, we have shown that, given any function $f$ of
the form \Cref{eq:repf} and any quadrature rule such that 
$E_1 \leq {\alpha_n}^2$, where $E_1$ is defined in \Cref{eq:E1},
there exists  
a coefficient vector $c \in \mathbb{R}^{N}$ 
such that,
letting $t_1$, $t_2$, \ldots , $t_N$ denote the quadrature nodes  
shifted to the interval $[a,b]$, 
\begin{align}
\label{eq:hatf2}
{f}_N(x)=\sum_{j=1}^{N} c_j x^{t_j}
\end{align}
is uniformly close to $f$, to within an error given by \Cref{eq:ferror}.

In this section, we show that, 
by choosing a quadrature rule with quadrature nodes $s_1$, $s_2$,
\ldots , $s_N$ such that 
$E_2 \leq \frac{{\alpha_n}^2}{2{n}}$,
where $E_2$ is defined in \Cref{eq:E2},
and letting 
\begin{align}
x_j=e^{-\frac{s_j}{b-a}},
\end{align}
for
$j=1$, $2$, \ldots , $N$,
we can construct
an approximation
\begin{align}
\label{eq:hatf3}
\hat{f}_N(x)=\sum_{j=1}^{N} \hat{c}_j x^{t_j}
\end{align}
which is also uniformly close to $f$,
by numerically solving a linear system
\begin{align}
\label{eq:VC=F}
Vc=F,
\end{align}
for the coefficient vector $\hat{c} \in \mathbb{R}^N$,
where
\begin{align}
\label{eq:V}
V=
\begin{pmatrix}
x^{t_1}_1& x^{t_2}_1 & \   \ldots& \ x^{t_N}_1\\
x^{t_1}_2& x^{t_2}_2 & \   \ldots& \ x^{t_N}_2\\
\vdots & \vdots& \  \ddots &   \vdots\\
x^{t_1}_N& x^{t_2}_N & \   \ldots& \ x^{t_N}_N
\end{pmatrix} \in \mathbb{R}^{N \times N},
\end{align}
and
\begin{align}
\label{eq:F}
F=({f}(x_1), {f}(x_2), \ldots, {f}(x_N))^{T}\in \mathbb{R}^{N}. 
\end{align}
The uniform approximation error of $\hat{f}_N$ to $f$ over $[0,1]$ is bounded in 
\Cref{thm:globalerr}. Recall that 
$E_1 \leq \alpha^2_n$ and 
$E_2 \leq \frac{{\alpha_n}^2}{2{n}}$ when $N \approx 2n$
and the quadrature nodes are chosen to be the roots of $u_N(t)$
and $v_N(x)$
(see \Cref{sec:prac}).

In the following lemma,
we establish upper bounds on the norm
and the residual
of the perturbed TSVD solution $\hat{c}$ to the linear system in \Cref{eq:VC=F},
in terms of the norm of the 
coefficient vector $c$ in \Cref{eq:hatf2}.

\begin{lemma}
\label{lem:hatck}
Let $V \in \mathbb{R}^{N \times N}$, 
$F \in \mathbb{R}^{N}$, and
$\epsilon >0$.
Suppose that
\begin{align}
\label{eq:original2}
\hat{c}_k = (V+\delta V)_k^\dagger (F+\delta F),
\end{align}
where
$(V+\delta V)_k^\dagger$ is the pseudo-inverse of the $k$-TSVD of $V+\delta V$,
so that 
\begin{align}
\hat{\alpha}_k \geq \epsilon \geq \hat{\alpha}_{k+1},
\end{align}
where $\hat{\alpha}_k$ and $\hat{\alpha}_{k+1}$ denote the $k$th
and $(k+1)$th largest singular values of $V+\delta V$, defining 
$\hat{\alpha}_{N+1} := 0$,
where $\delta V \in \mathbb{R}^{N \times N}$ and
$\delta F \in \mathbb{R}^N$, 
with
\begin{align}
\norm{\delta V}_2\leq\epsilon_0\cdot\mu_1 < \frac{\epsilon}{2},
\end{align}
and
\begin{align}
\norm{\delta F}_2\leq\epsilon_0\cdot\mu_2,
\end{align}
for some $\epsilon_0$, $\mu_1$, $\mu_2 >0$.
Suppose further that
\begin{align}
\label{eq:original}
V c = F+\Delta F,
\end{align}
for some
$\Delta F \in \mathbb{R}^{N}$ and $c \in \mathbb{R}^N$.
Then,
\begin{align}
\label{eq:normhatck}
\norm{\hat{c}_k}_2 \leq \frac{1}{\hat{\alpha}_k}(2\epsilon+\hat{\alpha}_k)
\norm{c}_2+\frac{1}
{\hat{\alpha}_k}
(\norm{\Delta F}_2+\epsilon_0 \cdot \mu_2),
\end{align}
and 
\begin{align}
\norm{V\hat{c}_k -(F+\Delta F)}_2 &\leq 5
\epsilon \norm{c}_2+\frac{3}{2}\norm{\Delta F}_2+
\frac{3}{2}\epsilon_0\cdot\mu_2.
\end{align}
\end{lemma}

\begin{proof}
By \Cref{eq:original2}, we have
\begin{align}
(V+\delta V)_k \hat{c}_k = F+\delta F &= F+\Delta F-\Delta F+\delta F
= F+\Delta F +e,
\end{align}
where
$e:=-\Delta F+\delta F$. 
Thus, \Cref{thm:tsvd} implies that 
\begin{align}
\label{eq:normchatk}
\norm{\hat{c}_k}_2 &\leq
\frac{1}{\hat{\alpha}_k}(2\epsilon+\hat{\alpha}_k)
\norm{c}_2+\frac{1}
{\hat{\alpha}_k}
\norm{-\Delta F+\delta F}_2 \notag \\
&\leq
\frac{1}{\hat{\alpha}_k}(2\epsilon+\hat{\alpha}_k)
\norm{c}_2+\frac{1}
{\hat{\alpha}_k}
(\norm{\Delta F}_2+\epsilon_0 \cdot \mu_2),
\end{align}
and that
\begin{align}
\norm{V\hat{c}_k -(F+\Delta F)}_2 &\leq 5 
\epsilon \norm{c}_2+\frac{3}{2}\norm{\Delta F}_2+
\frac{3}{2}\epsilon_0\cdot\mu_2.
\end{align}
\end{proof}
The following observation 
bounds the backward error,
$\norm{V\hat{c}_k-F}_2$,
where $\hat{c}_k$ is the TSVD solution to the perturbed 
linear system, defined in \Cref{eq:original2}.

\begin{observation}
\label{obs:hatck} 
According to \Cref{lem:hatck}, 
the TSVD solution $\hat{c}_k$ to the perturbed linear system
is bounded by the norm of $c$, as described in \Cref{eq:normhatck},
where $c$ is the exact solution to the linear system $V{c} =F+\Delta F$,
and satisfies \Cref{eq:normchat}. Thus,
the resulting backward error is bounded by 
\begin{align}
\label{eq:back}
&\hspace{-10mm}\norm{V\hat{c}_k-F}_2 
=\bnorm{V\hat{c}_k -
\bigl{(}F+\Delta F\bigr{)}+\Delta F}_2 \notag \\
&\hspace{10mm}
\leq \bnorm{V\hat{c}_k -
\bigl{(}F+\Delta F\bigr{)}}_2+\norm{\Delta F}_2 \notag \\
&\hspace{10mm}
\leq 
5\epsilon \norm{c}_2+\frac{5}{2}\norm{\Delta F}_2+
\frac{3}{2}\epsilon_0\cdot\mu_2.
\end{align}
\end{observation}

Although the interpolation matrix $V$ in the basis of 
$\{x^{{t}_j}\}_{j=1}^{N}$ tends to be ill-conditioned,
resulting in a loss of stability in the solution to the linear system 
in \Cref{eq:VC=F}, 
we have shown in \Cref{lem:hatck} and \Cref{obs:hatck} that,
when
the TSVD is used to solve the linear system
in \Cref{eq:VC=F},
the backward error, $\norm{V\hat{c}_k-F}_2$,
which measures the discrepancy between $f$ and $\hat{f}_N$
at the collocation points,
is nonetheless small.

The following lemma bounds the $L^{\infty}$-norm of a
function of the form \Cref{eq:repf},
in terms of its values at the collocation points $\{x_j\}_{j=1}^{N}$. The constant
appearing in this bound serves the same role as the Lebesgue constant for polynomial 
interpolation.

\begin{lemma}
\label{lem:normp}
Suppose that 
$s_1$, $s_2$, \ldots , $s_N$ and
$w_1$, $w_2$, \ldots , $w_N$ are the
nodes and weights of a quadrature rule 
such that 
$E_2 
\leq \frac{{\alpha_n}^2}{2{n}}$, 
where $E_2$ is defined in \Cref{eq:E2},
and that the collocation points 
$X:=(x_j)_{j=1}^{N}$ are defined by the 
formula
$x_j=e^{-\frac{s_j}{b-a}}$,
for $j=1$, $2$, \ldots , $N$.
Let $f(x)$ be a function of the form \Cref{eq:repf},
for some signed Radon measure $\sigma$, and 
let $f(X) \in \mathbb{R}^{N}$ be the vector of values of $f(x)$ sampled
at $X$. Then, 
\begin{align}
\label{eq:truefl}
\hspace{-25mm}
\norm{f}_{L^{\infty}[0,1]} \leq
\sqrt{2}\cdot\max_{1\leq j\leq N}\sqrt{w_j}\cdot V_n \cdot
\norm{f(X)}_2+
|\sigma| A^{\infty}_n\cdot(1+\sqrt{N}\cdot\sqrt{2}\max_{1\leq j\leq N}\sqrt{w_j}\cdot V_n),
\end{align}
where $A^{\infty}_n$ and $V_n$ are defined in \Cref{eq:Ainfty} and \Cref{eq:Vn},
respectively.
\end{lemma}

\begin{proof}
Recall from the proof of \Cref{thm:err} that 
$f(x)$ can be approximated by
\begin{align}
\label{eq:tildec3}
\hspace{-10mm} \tilde{f}(x)&=
\sum_{i=0}^{n-1}\tilde{c}_i\alpha_i v_i(-(b-a)\log x), 
\end{align}
such that
\begin{align}
\label{eq:ferrnorm}
\hspace{-10mm} \norm{f-\tilde{f}}_{L^{\infty}[0,1]}
\leq |{\sigma}|A^{\infty}_n,
\end{align}
where $A^{\infty}_n$ is defined in \Cref{eq:Ainfty}, and $\tilde{c}_i$, 
for $i=0$, $1$, \ldots , $n-1$,
is given by
\Cref{eq:deftc}.
Let
$\tilde{f}(X)=\bigl{(}\tilde{f}(x_1)$,
$\tilde{f}(x_2)$, \ldots , $\tilde{f}(x_N)\bigr{)}^T$.
Since 
\Cref{col:Anorm} implies that
$\norm{A^{\dagger}}_2 <
\sqrt{2} \max_{1 \leq j \leq N} \sqrt{w_j}$,
where the matrix $A^{\dagger} \in \mathbb{R}^{n \times N}$ is the
pseudo-inverse of $A$,
the coefficient vector $\tilde{c}$ in \Cref{eq:tildec3}
can be found stably by 
the formula
\begin{align}
(\tilde{c}_i \alpha_i) =A^{\dagger}\tilde{f}(X).
\end{align}
From \Cref{eq:tildec3}, we have
\begin{align}
&\hspace{0mm} \norm{\tilde{f}}_{L^{\infty}[0,1]}= 
\bbnorm{\sum_{i=0}^{n-1} \tilde{c}_i{\alpha}_i
v_i(-(b-a)\log{x})}_{L^{\infty}[0,1]} \notag \\
&\hspace{17mm} \leq \norm{(\tilde{c}_i \alpha_i)}_2
\sqrt{\sum_{i=0}^{n-1}\norm{v_i}^2_{L^{\infty}[0,\infty)}} \notag \\
&\hspace{17mm}\leq \norm{A^{\dagger}}_2\cdot \norm{\tilde{f}(X)}_2 \cdot V_n \notag \\
&\hspace{17mm}\leq \sqrt{2}\cdot \max_{1\leq j\leq N}\sqrt{w_j}\cdot \norm{\tilde{f}(X)}_2 \cdot V_n,
\end{align}
where $V_n$ is defined in \Cref{eq:Vn}.
Since
\begin{align}
\bigl{|}f(x_j)-\tilde{f}(x_j)\bigr{|}\leq \norm{f-\tilde{f}}_{L^{\infty}[0,1]},
\end{align}
for all $j=1$, $2$, \ldots , $N-1$, we have 
\begin{align}
\norm{\tilde{f}(X)}_2 \leq \sqrt{N} \norm{f-\tilde{f}}_{L^{\infty}[0,1]}+
\norm{{f}(X)}_2.
\end{align}
It follows that 
\begin{align}
&\hspace{-25mm}\norm{f}_{L^{\infty}[0,1]} \leq \norm{f-\tilde{f}}_{L^{\infty}[0,1]} + 
\norm{\tilde{f}}_{L^{\infty}[0,1]} \notag \\
&\hspace{-8mm}\leq \norm{f-\tilde{f}}_{L^{\infty}[0,1]} + 
\sqrt{2}\cdot \max_{1\leq j\leq N}\sqrt{w_j}\cdot \norm{\tilde{f}(X)}_2 \cdot V_n \notag \\
&\hspace{-8mm}\leq \norm{f-\tilde{f}}_{L^{\infty}[0,1]} + 
\sqrt{2}\cdot \max_{1\leq j\leq N}\sqrt{w_j}\cdot
(\sqrt{N} \norm{f-\tilde{f}}_{L^{\infty}[0,1]}+
\norm{{f}(X)}_2) \cdot V_n \notag \\
&\hspace{-8mm}\leq
\sqrt{2}\cdot\max_{1\leq j\leq N}\sqrt{w_j}\cdot V_n \cdot
\norm{f(X)}_2+
|\sigma| A^{\infty}_n\cdot(1+\sqrt{N}\cdot\sqrt{2}\max_{1\leq j\leq N}\sqrt{w_j}\cdot V_n).
\end{align}
\end{proof}

The following theorem provides an upper bound on the global approximation error
of $\hat{f}_N$ to $f$,
when the coefficient vector in the approximation is computed by
solving $Vc=F$
using the TSVD.

\begin{theorem}
\label{thm:globalerr}
Let $f(x)$ be a function of the form \Cref{eq:repf},
for some signed Radon measure $\sigma(\mu)$. 
Suppose that 
$\tilde{t}_1$, $\tilde{t}_2$, \ldots , $\tilde{t}_N$ 
are the nodes of a quadrature rule such that $E_1 \leq \alpha^2_n$,
and that $s_1$, $s_2$, \ldots , $s_N$ are the nodes of 
a quadrature rule such that  
$E_2 \leq \frac{{\alpha_n}^2}{2{n}}$, 
where $E_1$ and $E_2$ are defined in \Cref{eq:E1} and \Cref{eq:E2},
respectively.
Let $t_j=(b-a)\tilde{t}_j +a$, 
for $j=1$, $2$, \ldots , $N$, 
and let 
$x_1$, $x_2$, \ldots , $x_N$ 
be the collocation points defined by the
formula
$x_j=e^{-\frac{s_j}{b-a}}$. 
Suppose that $V \in \mathbb{R}^{N \times N}$ is defined in \Cref{eq:V}
and $F \in \mathbb{R}^N$ is defined in \Cref{eq:F},  
and let $\epsilon >0$.
Suppose further that
\begin{align}
\label{eq:VCF1}
\hat{c}_k = (V+\delta V)_k^\dagger (F+\delta F),
\end{align}
where
$(V+\delta V)_k^\dagger$ is the pseudo-inverse of the $k$-TSVD of $V+\delta V$,
so that 
\begin{align}
\hat{\alpha}_k \geq \epsilon \geq \hat{\alpha}_{k+1},
\end{align}
where $\hat{\alpha}_k$ and $\hat{\alpha}_{k+1}$ denote the $k$th
and $(k+1)$th largest singular values of $V+\delta V$, defining $\hat{\alpha}_{N+1}:=0$,
where $\delta V \in \mathbb{R}^{N \times N}$ and
$\delta F \in \mathbb{R}^N$, 
with
\begin{align}
\norm{\delta V}_2\leq\epsilon_0\cdot\mu_1 < \frac{\epsilon}{2},
\end{align}
and
\begin{align}
\norm{\delta F}_2\leq\epsilon_0\cdot\mu_2,
\end{align}
for some $\epsilon_0$, $\mu_1$, $\mu_2>0$.
Let
\begin{align}
\hat{f}_N(x) =\sum_{j=1}^{N} \hat{c}_{k,j} x^{t_j},
\end{align}
with $\hat{c}_k$ defined in \Cref{eq:VCF1}.
Then,
\begin{align}
\label{eq:interperr}
&\hspace{-20mm}\norm{f-\hat{f}_N}_{L^{\infty}[0,1]} \notag \\
&\hspace{-25mm} \leq
\sqrt{2}\cdot \max_{1\leq j\leq N}\sqrt{w_j}\cdot V_n \cdot
\biggl{(}5 \epsilon\cdot \norm{\tilde{w}}_1 |\sigma|\cdot {U_n}^2
+\frac{5}{2}
\sqrt{N}\cdot |{\sigma}|\alpha_n 
 \Bigl{(}\frac{A^{\infty}_n}{\alpha_n}+U_n V_n +\frac{A^{\infty}_n}{\alpha_n}{U_n}^2 \norm{\tilde{w}}_1\Bigr{)} \notag \\
&\hspace{-4mm} +\frac{3}{2}\epsilon_0\cdot \mu_2 \biggr{)} +
(|\sigma|+\sqrt{N}\norm{\hat{c}_k}_2)\cdot A^{\infty}_n\cdot(1+\sqrt{N}\cdot\sqrt{2}\max_{1\leq j\leq N}\sqrt{w_j}\cdot V_n),
\end{align}
where 
\begin{align}
\label{eq:interphatck}
&\hspace{-25mm}\norm{\hat{c}_k}_2 
\leq \frac{1}{\hat{\alpha}_k}(2\epsilon+\hat{\alpha}_k)
\norm{\tilde{w}}_1 |\sigma| \cdot {U_n}^2+\frac{1}
{\hat{\alpha}_k}\biggl{(}\sqrt{N}\cdot |{\sigma}|\alpha_n 
 \Bigl{(}\frac{A^{\infty}_n}{\alpha_n}+U_n V_n +\frac{A^{\infty}_n}{\alpha_n}{U_n}^2 \norm{\tilde{w}}_1\Bigr{)} \notag \\
&\hspace{-4mm} +\epsilon_0\cdot\mu_2\biggr{)}. 
\end{align}
\end{theorem}
\begin{proof}
We observe that
\begin{align}
\hat{f}_N(x)=\int_{a}^{b} x^{\mu} \, \hat{\sigma}_N(\mu) \, d\mu,
\end{align}
for the signed Radon measure 
\begin{align}
\hat{\sigma}_N(t)=\sum_{j=1}^{N} \hat{c}_{k,j}\delta(t-t_j),
\end{align}
where $\delta(t)$ is the Dirac delta function.
Therefore, $f(x)-\hat{f}_N(x)$
can be rewritten as 
\begin{align}
f(x)-\hat{f}_N(x)=\int_{a}^{b} x^{\mu}
\bigl{(}\sigma(\mu)-\hat{\sigma}_N(\mu)\bigr{)}\, d\mu,
\end{align}
where $\sigma(\mu)$ is defined in \Cref{eq:repf}.
By \Cref{thm:err}, there exists a vector $c \in \mathbb{R}^N$, such that
\begin{align}
{f}_N(x)=\sum_{j=1}^N c_jx^{t_j}
\end{align}
is uniformly close to $f$, with an error bounded by \Cref{eq:ferror}.
Let $X:=(x_j)_{j=1}^{N}$ and
$\Delta F:= f(X)-f_N(X)$. Then, 
\begin{align}
\label{eq:normdf}
&\hspace{10mm} \norm{\Delta F}_2  \leq \sqrt{N}\cdot\norm{f-f_N}_{L^{\infty}[0,1]} \notag \\ 
&\hspace{22mm} \leq \sqrt{N}\cdot |{\sigma}|\alpha_n 
 \Bigl{(}\frac{A^{\infty}_n}{\alpha_n}+U_n V_n +\frac{A^{\infty}_n}{\alpha_n}{U_n}^2 \norm{\tilde{w}}_1\Bigr{)},
\end{align}
where $A^{\infty}_n$, $U_n$ and $V_n$ are defined in \Cref{eq:Ainfty}, \Cref{eq:Un}
and \Cref{eq:Vn},
respectively.
By \Cref{eq:back} and \Cref{eq:normchat}, we have
\begin{align}
\label{eq:trueback}
&\hspace{-12mm}\norm{f(X)-\hat{f}_N(X)}_2  \notag \\
&\hspace{-17mm} = \norm{V\hat{c}_k-F}_2 \notag \\
&\hspace{-17mm} \leq 
5\epsilon \norm{c}_2+\frac{5}{2}\norm{\Delta F}_2+
\frac{3}{2}\epsilon_0\cdot\mu_2 \notag \\
&\hspace{-17mm}\leq 5 \epsilon\cdot \norm{\tilde{w}}_1 |\sigma|\cdot {U_n}^2
+\frac{5}{2}
\sqrt{N}\cdot |{\sigma}|\alpha_n 
 \Bigl{(}\frac{A^{\infty}_n}{\alpha_n}+U_n V_n +\frac{A^{\infty}_n}{\alpha_n}{U_n}^2 \norm{\tilde{w}}_1\Bigr{)}
+\frac{3}{2}\epsilon_0\cdot \mu_2,
\end{align}
where $\tilde{w}$ is the vector of 
the quadrature weights such that $E_1 \leq \alpha^2_n$.
It follows from \Cref{lem:normp} that the uniform error of the approximation
of 
$\hat{f}_N$ to $f$ is bounded as 
\begin{align}
&\hspace{-20mm}\norm{f-\hat{f}_N}_{L^{\infty}[0,1]} \notag \\
&\hspace{-25mm}
\leq \sqrt{2}\cdot \max_{1\leq j\leq N}\sqrt{w_j}\cdot V_n \cdot
\biggl{(}5 \epsilon\cdot \norm{\tilde{w}}_1 |\sigma|\cdot {U_n}^2
+\frac{5}{2}
\sqrt{N}\cdot |{\sigma}|\alpha_n 
 \Bigl{(}\frac{A^{\infty}_n}{\alpha_n}+U_n V_n +\frac{A^{\infty}_n}{\alpha_n}{U_n}^2 \norm{\tilde{w}}_1\Bigr{)} \notag \\
&\hspace{-4mm} +\frac{3}{2}\epsilon_0\cdot \mu_2 \biggr{)} 
+|\sigma-\hat{\sigma}_N| A^{\infty}_n\cdot(1+\sqrt{N}\cdot\sqrt{2}\max_{1\leq j\leq N}\sqrt{w_j}\cdot V_n).
\end{align}
Since $\abs{\sigma-\hat{\sigma}_N} \leq \abs{\sigma}+\abs{\hat{\sigma}_N}$ and 
$\abs{\hat{\sigma}_N} \leq \norm{\hat{c}_k}_1$, we have
\begin{align}
\label{eq:resinterp}
&\hspace{-20mm}\norm{f-\hat{f}_N}_{L^{\infty}[0,1]} \notag \\
&\hspace{-24mm} \leq
\sqrt{2}\cdot \max_{1\leq j\leq N}\sqrt{w_j}\cdot V_n \cdot
\biggl{(}5 \epsilon\cdot \norm{\tilde{w}}_1 |\sigma|\cdot {U_n}^2
+\frac{5}{2}
\sqrt{N}\cdot |{\sigma}|\alpha_n 
 \Bigl{(}\frac{A^{\infty}_n}{\alpha_n}+U_n V_n +\frac{A^{\infty}_n}{\alpha_n}{U_n}^2 \norm{\tilde{w}}_1\Bigr{)} \notag \\
&\hspace{-4mm} +\frac{3}{2}\epsilon_0\cdot \mu_2 \biggr{)} 
+(|\sigma|+\norm{\hat{c}_k}_1)\cdot A^{\infty}_n\cdot(1+\sqrt{N}\cdot\sqrt{2}\max_{1\leq j\leq N}\sqrt{w_j}\cdot V_n) \notag \\
&\hspace{-24mm} \leq
\sqrt{2}\cdot \max_{1\leq j\leq N}\sqrt{w_j}\cdot V_n \cdot
\biggl{(}5 \epsilon\cdot \norm{\tilde{w}}_1 |\sigma|\cdot {U_n}^2
+\frac{5}{2}
\sqrt{N}\cdot |{\sigma}|\alpha_n 
 \Bigl{(}\frac{A^{\infty}_n}{\alpha_n}+U_n V_n +\frac{A^{\infty}_n}{\alpha_n}{U_n}^2 \norm{\tilde{w}}_1\Bigr{)} \notag \\
&\hspace{-4mm}+\frac{3}{2}\epsilon_0\cdot \mu_2 \biggr{)}
+(|\sigma|+\sqrt{N}\norm{\hat{c}_k}_2)\cdot A^{\infty}_n\cdot(1+\sqrt{N}\cdot\sqrt{2}\max_{1\leq j\leq N}\sqrt{w_j}\cdot V_n),
\end{align}
where $\norm{\hat{c}_k}_2$ is bounded by  
substituting 
\Cref{eq:normchat} and \Cref{eq:normdf} into 
\Cref{eq:normhatck}, 
\begin{align}
\label{eq:normhatck2}
&\hspace{-22mm}\norm{\hat{c}_k}_2 \leq \frac{1}{\hat{\alpha}_k}(2\epsilon+\hat{\alpha}_k)
\norm{c}_2+\frac{1}
{\hat{\alpha}_k}(\norm{\Delta F}_2+\epsilon_0\cdot\mu_2) \notag \\
&\hspace{-13mm} \leq \frac{1}{\hat{\alpha}_k}(2\epsilon+\hat{\alpha}_k)
\norm{\tilde{w}}_1 |\sigma| \cdot {U_n}^2+\frac{1}
{\hat{\alpha}_k}\biggl{(}\sqrt{N}\cdot |{\sigma}|\alpha_n 
 \Bigl{(}\frac{A^{\infty}_n}{\alpha_n}+U_n V_n +\frac{A^{\infty}_n}{\alpha_n}{U_n}^2 \norm{\tilde{w}}_1\Bigr{)} \notag \\
&\hspace{26mm}  +\epsilon_0\cdot\mu_2\biggr{)}. 
\end{align}
\end{proof}

\begin{remark}
\label{rmk:error1}
By ignoring all the small terms in \Cref{eq:resinterp} and \Cref{eq:normhatck2},
and recalling that $A^{\infty}_n \approx \alpha_n$, 
we have
\begin{align}
\label{eq:resinterp2}
\norm{f-\hat{f}_N}_{L^{\infty}[0,1]}
\lesssim (\epsilon +\alpha_n) |\sigma| + \epsilon_0 \cdot  \mu_2+\alpha_n \norm{\hat{c}_k}_2,
\end{align}
where
\begin{align}
\label{eq:normhatck4}
\norm{\hat{c}_k}_2 \lesssim (\frac{\epsilon}{\hat{\alpha}_k}+1+\frac{\alpha_n}{\hat{\alpha}_k})|\sigma|+
\frac{\epsilon_0}{\hat{\alpha}_k}\mu_2.
\end{align}
Thus, 
\begin{align}
\label{eq:resinterp3}
\norm{f-\hat{f}_N}_{L^{\infty}[0,1]}
\lesssim (\epsilon +\alpha_n+\frac{\alpha_n\epsilon}{\hat{\alpha}_k}
+\frac{\alpha^2_n}{\hat{\alpha}_k}) |\sigma| + \epsilon_0 \cdot  \mu_2+\frac{\alpha_n}{\hat{\alpha}_k}\epsilon_0 \cdot \mu_2.
\end{align}
\end{remark}
Neglecting
all the insignificant terms,
the accuracy of the approximation depends
on $\alpha_n$, $\epsilon$, $\hat{\alpha}_k$ and $|{\sigma}|$, 
as well as the machine precision $\epsilon_0$.

Recalling that $\hat{\alpha}_k \geq \epsilon$. 
If we choose $\epsilon \approx \alpha_n$ in \Cref{eq:resinterp3}, then
$\frac{\alpha_n}{\hat{\alpha}_k}\leq \frac{\alpha_n}{\epsilon}\approx 1$ and, accordingly,
\begin{align}
\label{eq:emperr}
\norm{f-\hat{f}_N}_{L^{\infty}[0,1]} &\lesssim 
 (\epsilon+\alpha_n)\abs{\sigma}+
\epsilon_0 \mu_2 \notag \\
&\approx
\alpha_n\abs{\sigma}+
\epsilon_0 \mu_2.
\end{align}
Thus, the approximation error can achieve a bound that is roughly proportional to $\alpha_n \abs{\sigma}$.
Otherwise, if $\epsilon$ is significantly smaller than $\alpha_n$, then the error
will exceed $\alpha_n|\sigma|$ because of the term $\frac{{\alpha_n}^2}{\hat{\alpha}_k}$.

\section{Extension from Measures to Distributions}
\label{sec:ext}
In \Cref{sec:NAEA}, we presented an algorithm for approximating functions
of the form
\begin{align}
f(x)=\int_{a}^{b} x^{\mu} \sigma(\mu) \, d \mu, \qquad x\in [0,1],
\end{align}
where $\sigma$ is a signed Radon measure, and derived an estimate for the
uniform error of the approximation in \Cref{thm:globalerr}. In this section, we
observe that this same algorithm can be applied more generally to functions
of the form
\begin{align}
\label{eq:frepnew}
f(x) = \inner{\sigma(\mu), x^\mu},
\end{align}
where $\sigma \in \mathcal{D}'(\R)$ is a distribution supported on 
the interval $[a,b]$.
Since every distribution with compact support has a finite
order, it follows that $\sigma \in C^m([a,b])^*$ for some order $m\ge 0$.

Recall that  
\begin{align}
\abs{\inner{\sigma, \varphi}} \leq \norm{\sigma}_{C^m([a,b])^*}
\cdot \norm{\varphi}_{C^m([a,b])},
\end{align}
where
\begin{align}
\norm{\varphi}_{C^m([a,b])} = \sum_{n=0}^{m} \, \, \sup_{x \in [a,b]}
\abs{\varphi^{(n)}},
\end{align}
and
\begin{align}
\norm{\sigma}_{C^m([a,b])^*} = \sup_{\substack{\varphi \in C^m([a,b]) \\
\norm{\varphi}_{C^m([a,b])}=1}} \abs{\sigma(\varphi)}.
\end{align}
We can use the 
algorithm of \Cref{sec:NAEA} to approximate a function
of the form \Cref{eq:frepnew}, where the approximation error is 
bounded by the following
theorem, which generalizes \Cref{thm:globalerr}.

\begin{theorem}
\label{thm:globalerrdist}
Let $f(x)$ be a function of the form \Cref{eq:frepnew},
Suppose that 
$\tilde{t}_1$, $\tilde{t}_2$, \ldots , $\tilde{t}_N$ 
are the nodes of a quadrature rule such that $E_1 \leq \alpha^2_n$,
and that $s_1$, $s_2$, \ldots , $s_N$ are the nodes of 
a quadrature rule such that  
$E_2 \leq \frac{{\alpha_n}^2}{2{n}}$, 
where $E_1$ and $E_2$ are defined in \Cref{eq:E1} and \Cref{eq:E2},
respectively.
Let $t_j=(b-a)\tilde{t}_j +a$, 
for $j=1$, $2$, \ldots , $N$, 
and let 
$x_1$, $x_2$, \ldots , $x_N$ 
be the collocation points defined by the
formula
$x_j=e^{-\frac{s_j}{b-a}}$. 
Suppose that $V \in \mathbb{R}^{N \times N}$ is defined in \Cref{eq:V}
and $F \in \mathbb{R}^N$ is defined in \Cref{eq:F},  
and let $\epsilon >0$.
Suppose further that
\begin{align}
\label{eq:VCF2}
\hat{c}_k = (V+\delta V)_k^\dagger (F+\delta F),
\end{align}
where
$(V+\delta V)_k^\dagger$ is the pseudo-inverse of the $k$-TSVD of $V+\delta V$,
so that 
\begin{align}
\hat{\alpha}_k \geq \epsilon \geq \hat{\alpha}_{k+1},
\end{align}
where $\hat{\alpha}_k$ and $\hat{\alpha}_{k+1}$ denote the $k$th
and $(k+1)$th largest singular values of $V+\delta V$, defining $\hat{\alpha}_{N+1} :=0$,
where $\delta V \in \mathbb{R}^{N \times N}$ and
$\delta F \in \mathbb{R}^N$, 
with
\begin{align}
\norm{\delta V}_2\leq\epsilon_0\cdot\mu_1 < \frac{\epsilon}{2},
\end{align}
and
\begin{align}
\norm{\delta F}_2\leq\epsilon_0\cdot\mu_2,
\end{align}
for some $\epsilon_0$, $\mu_1$, $\mu_2 >0$.
Let
\begin{align}
\hat{f}_N(x) =\sum_{j=1}^{N} \hat{c}_{k,j} x^{t_j},
\end{align}
with $\hat{c}_k$ defined in \Cref{eq:VCF2}.
Then,
\begin{align}
\label{eq:interperr2}
&\hspace{-22mm}\norm{f-\hat{f}_N}_{L^{\infty}[0,1]} \notag \\
&\hspace{-27mm} \leq
\sqrt{2}\cdot \max_{1\leq j\leq N}\sqrt{w_j}\cdot V_n \cdot
\biggl{(}5 \epsilon\cdot \norm{\tilde{w}}_1 \cdot\norm{\sigma}_{C^m([a,b])^*} \cdot
\max_{0\leq i \leq n-1} \bnorm{u_i\bigl{(}\tfrac{t-a}{b-a}\bigr{)}}_{C^m([a,b])}
\cdot {U_n}^2 \notag \\
&\hspace{-19mm}+\frac{5}{2}
\sqrt{N}\cdot
\norm{\sigma}_{C^m([a,b])^*} \cdot
\max_{0\leq i \leq n-1} \bnorm{u_i\bigl{(}\tfrac{t-a}{b-a}\bigr{)}}_{C^m([a,b])}
\cdot \alpha_n 
 \Bigl{(}\frac{A^{\infty}_n}{\alpha_n}+U_n V_n +\frac{A^{\infty}_n}{\alpha_n}{U_n}^2 \norm{\tilde{w}}_1\Bigr{)} \notag \\
&\hspace{-19mm}+\frac{3}{2}\epsilon_0\cdot \mu_2 \biggr{)}
+(\norm{\sigma}_{C^m([a,b])^*} \cdot 
\max_{0\leq i \leq n-1} \bnorm{u_i\bigl{(}\tfrac{t-a}{b-a}\bigr{)}}_{C^m([a,b])}
+\sqrt{N}\norm{\hat{c}_k}_2)\cdot A^{\infty}_n \notag \\
&\hspace{-19mm}\cdot (1+\sqrt{N}\cdot\sqrt{2}\max_{1\leq j\leq N}\sqrt{w_j}\cdot V_n),
\end{align}
where 
\begin{align}
\label{eq:interphatck2}
&\hspace{-25mm}\norm{\hat{c}_k}_2 \leq 
\frac{1}{\hat{\alpha}_k}(2\epsilon+\hat{\alpha}_k)
\norm{\tilde{w}}_1 \cdot \norm{\sigma}_{C^m([a,b])^*} \cdot
\max_{0\leq i \leq n-1} \bnorm{u_i\bigl{(}\tfrac{t-a}{b-a}\bigr{)}}_{C^m([a,b])}
 \cdot {U_n}^2 \notag \\
 &\hspace{-5mm}+\frac{1}
{\hat{\alpha}_k}\biggl{(}\sqrt{N}\cdot \norm{\sigma}_{C^m([a,b])^*} \cdot
\max_{0\leq i \leq n-1} \bnorm{u_i\bigl{(}\tfrac{t-a}{b-a}\bigr{)}}_{C^m([a,b])}
\notag \\
&\hspace{-4mm} \cdot \alpha_n 
\Bigl{(}\frac{A^{\infty}_n}{\alpha_n}+U_n V_n +\frac{A^{\infty}_n}{\alpha_n}{U_n}^2 \norm{\tilde{w}}_1\Bigr{)}
 +\epsilon_0\cdot\mu_2\biggr{)}.
\end{align}
\end{theorem}
\begin{proof}
Since the proof closely resembles the one of \Cref{thm:globalerr}, we 
omit it here. The only difference is that $\abs{\sigma}$ in \Cref{eq:interperr}
is replaced by
$\norm{\sigma}_{C^m([a,b])^*} \cdot 
\max_{0\leq i \leq n-1} \bnorm{u_i\bigl{(}\frac{t-a}{b-a}\bigr{)}}_{C^m([a,b])}$,
due to the fact that
\begin{align}
\abs{\inner{\sigma, u_i\bigl{(}\tfrac{t-a}{b-a}\bigr{)}}} \leq \norm{\sigma}_{C^m([a,b])^*} \cdot
\bnorm{u_i\bigl{(}\tfrac{t-a}{b-a}\bigr{)}}_{C^m([a,b])},
\end{align}
where the term 
$\bnorm{u_i\bigl{(}\frac{t-a}{b-a}\bigr{)}}_{C^m([a,b])}$ is not negligible,
for $m \geq 1$. 
\end{proof}
Note that, when $\sigma$ is a signed Radon measure, the corresponding
distribution has order zero, and 
$\norm{\sigma}_{C([a,b])^*}=\abs{\sigma}$.
Thus, in this case, the above bound on 
$\norm{f-\hat{f}_N}_{L^{\infty}[0,1]}$ is exactly the same as the bound 
described in 
\Cref{eq:interperr}.
\begin{remark}
Similar to \Cref{rmk:error1},
the approximation error is approximately bounded as 
\begin{align}
\label{eq:resinterp4}
&\hspace{-25mm} \norm{f-\hat{f}_N}_{L^{\infty}[0,1]}
\lesssim (\epsilon +\alpha_n+\frac{\alpha_n\epsilon}{\hat{\alpha}_k}
+\frac{\alpha^2_n}{\hat{\alpha}_k})\norm{\sigma}_{C^m([a,b])^*} \cdot
\max_{0\leq i \leq n-1} \bnorm{u_i\bigl{(}\tfrac{t-a}{b-a}\bigr{)}}_{C^m([a,b])} \notag \\
&\hspace{7mm} + \epsilon_0 \cdot  \mu_2+\frac{\alpha_n}{\hat{\alpha}_k}\epsilon_0 \cdot \mu_2.
\end{align}
If we choose $\epsilon \approx \alpha_n$, \Cref{eq:resinterp4} becomes 
\begin{align}
\label{eq:resinterp5}
\hspace{-20mm} \norm{f-\hat{f}_N}_{L^{\infty}[0,1]}
\lesssim (\epsilon +\alpha_n
)\norm{\sigma}_{C^m([a,b])^*} \cdot
\max_{0\leq i \leq n-1} \bnorm{u_i\bigl{(}\tfrac{t-a}{b-a}\bigr{)}}_{C^m([a,b])} + \epsilon_0 \cdot  \mu_2.
\end{align}
\end{remark}

\section{Numerical Algorithm}
\label{sec:sum}
The steps of the numerical algorithm for constructing the approximations 
described in  
\Cref{thm:globalerr} and \Cref{thm:globalerrdist} can be summarized
as follows:
\begin{enumerate}
\item Given $f(x)$ of the form \Cref{eq:repf} or \Cref{eq:frepnew}, compute $\gamma=\frac{b}{a}$.
\item Compute the right singular functions of $T_{\gamma}$, $u_0$, $u_1$, \ldots, using the algorithm 
described in Section $4.1$ of \cite{laplace}. 
\item Compute the left singular functions of $T_{\gamma}$, $v_0$, $v_1$, \ldots, using the algorithm 
described in Section $4.1$ of \cite{laplace2}. 
\item Compute the singular values of $T_{\gamma}$, $\alpha_0$, $\alpha_1$, \ldots, using the algorithm 
described in Section $4.2$ of \cite{laplace2}. 
\item Find $n$ such that 
$\alpha_n |\sigma|$ is the desired approximation error, where $\sigma$ is defined in 
\Cref{eq:repf} or 
\Cref{eq:frepnew}. 
\item Set $N=2n$. 
\item Use the algorithm for comrade matrices, described in \cite{roots}, to 
compute the roots of $u_N(t)$ as $\tilde{t}_1$, $\tilde{t}_2$,
\ldots, $\tilde{t}_N$, and the roots of
$v_N(x)$ as $s_1$, $s_2$, \ldots, $s_N$.
\item  
Rescale $s_j$ by $\frac{s_j}{2}$, 
for $j=1$, $2$, \ldots, $N$.
\item Obtain the non-integer powers 
$t_j=(b-a)\tilde{t}_j+a$, and the collocation points 
$x_j=e^{-\frac{s_j}{b-a}}$, for $j=1$, $2$, \ldots, $N$.
\item Use the non-integer powers
and the collocation points to
construct $V \in \mathbb{R}^{N \times N}$ as defined in \Cref{eq:V}
and $F \in \mathbb{R}^N$ as defined in \Cref{eq:F}.  
\item Solve the linear system $Vc=F$ for the coefficient vector $\hat{c}$,
using a TSVD solver with truncation point $\epsilon = \alpha_n$.
\item Construct the approximation $\hat{f}_N(x)=\sum_{j=1}^N \hat{c}_jx^{t_j}$. 
\end{enumerate}

In the previous section, we prove that, given 
any two $N$-point quadrature rules for which
$E_1 \leq \alpha^2_n$ and $E_2 \leq \frac{\alpha^2_n}{2{n}}$,
where $E_1$ and $E_2$ are defined in \Cref{eq:E1} and
\Cref{eq:E2}, respectively, one can numerically approximate 
$f$ by $\hat{f}_N$ uniformly
to precision $\alpha_n |\sigma|$. We can construct such quadrature rules by
selecting the roots of of $u_N(t)$ and $v_N(x)$ for $N \approx 2n$,
as shown in \Cref{sec:prac}.
However,
experiments in
\Cref{sec:prac} show that,
by taking $N=n$ and using the nodes $x_k$ instead of $x_k/2$, we can 
obtain the same result in practice.
Since the function $f(t)=e^{-x (t+\frac{1}{\gamma-1})}u_i(t)$ can be integrated 
to precision $\alpha^2_n$ using only $N=n$ points, and 
the interpolation matrix $A \in \mathbb{R}^{N \times n}$
defined in \Cref{eq:A} is well conditioned also for $N=n$,
we can achieve the same uniform approximation error of $\hat{f}_N$
to $f$
as described in \Cref{eq:interperr} in \Cref{thm:globalerr},
for $N=n$.

Previously, we assumed $\epsilon \approx \alpha_n$. When $N=n$, we instead
choose $\epsilon$ as follows.
First, we observe $\norm{V^{-1}}_2 \leq \frac{1}{\alpha_n}$,
as shown in \Cref{fig:Vnorm}.
Letting 
$\tilde{\alpha}_n$ denote the $n$-th singular value of $V$
and assuming that $\delta V$ satisfies 
$\norm{\delta V}_2 \leq \frac{\tilde{\alpha}_n}{2}$, we have
\begin{align}
\norm{(V+\delta V)^{-1}}_2 \leq \frac{1}{\tilde{\alpha}_n-\norm{\delta V}_2}
\leq \frac{2}{\tilde{\alpha}_n}=2\norm{V^{-1}}_2.
\end{align}
Thus, $\norm{(V+\delta V)^{-1}}_2 \lesssim \frac{1}{\alpha_n}$,
which is equivalent to
$\frac{1}{\hat{\alpha}_n} \lesssim \frac{1}{\alpha_n}$.
We have then that
$\frac{\alpha_n}{\hat{\alpha}_k} \lesssim 1$, and therefore, as long as $\epsilon$ is not larger 
than $\alpha_n$, the resulting approximation error is bounded by 
\begin{align}
\label{eq:emperr2}
\norm{f-\hat{f}_N}_{L^{\infty}[0,1]} &\lesssim (\epsilon+\alpha_n)
\cdot \abs{\sigma}+
\epsilon_0 \mu_2 \notag \\
&\lesssim \alpha_n\abs{\sigma}+
\epsilon_0 \mu_2.
\end{align}
In practice, we take $\epsilon=\epsilon_0$.

Therefore, we implement a practical version of the numerical algorithm, 
which closely follows the one
outlined at the beginning of this section,
with the following adjustments:
\begin{itemize}
\item
We replace $N=2n$ with $N=n$ in Step $6$.
\item
When taking $N=n$, we use the 
roots of $v_N(x)$, $s_1$, $s_2$, \ldots, $s_N$, without scaling them by $2$. Thus, we
delete Step $8$. 
\item
We replace $\epsilon=\alpha_n$ with $\epsilon = \epsilon_0$
in Step $11$.
\end{itemize}
The rest of the steps
remain the same.

\begin{figure}[!ht]
\centering
\subfloat[$\gamma=10$: $a=1$, $b=10$]{%
\includegraphics[scale=0.47]{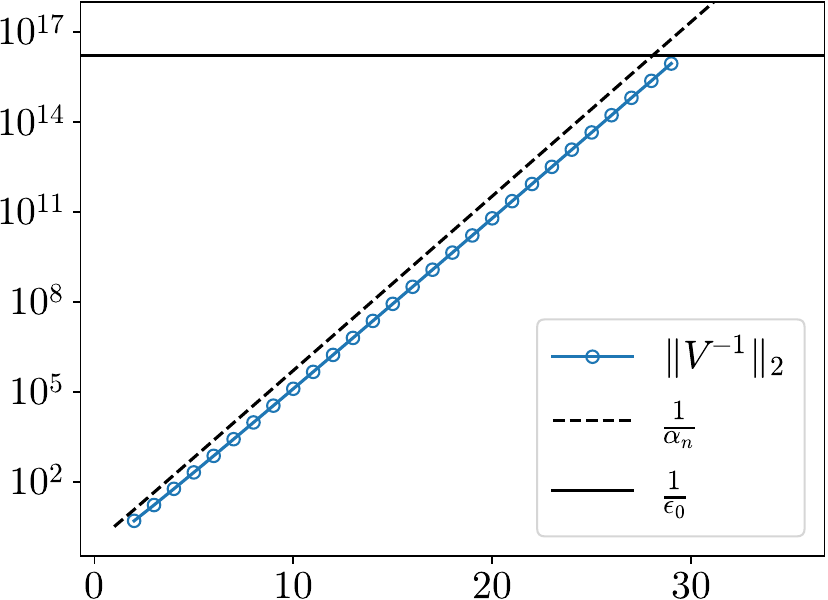}
}    

\subfloat[$\gamma=50$: $a=1$, $b=50$]{%
\includegraphics[scale=0.47]{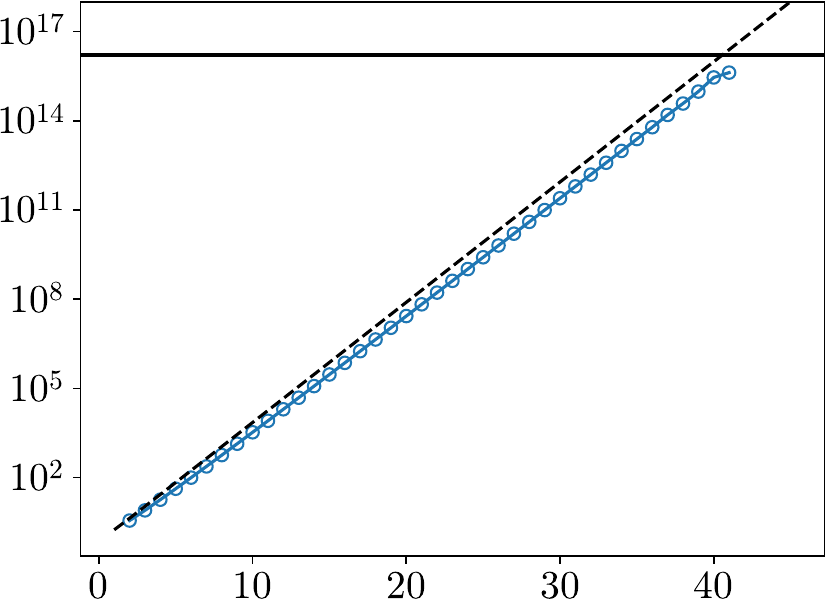}
}

\subfloat[$\gamma=250$: $a=1$, $b=250$]{%
\includegraphics[scale=0.47]{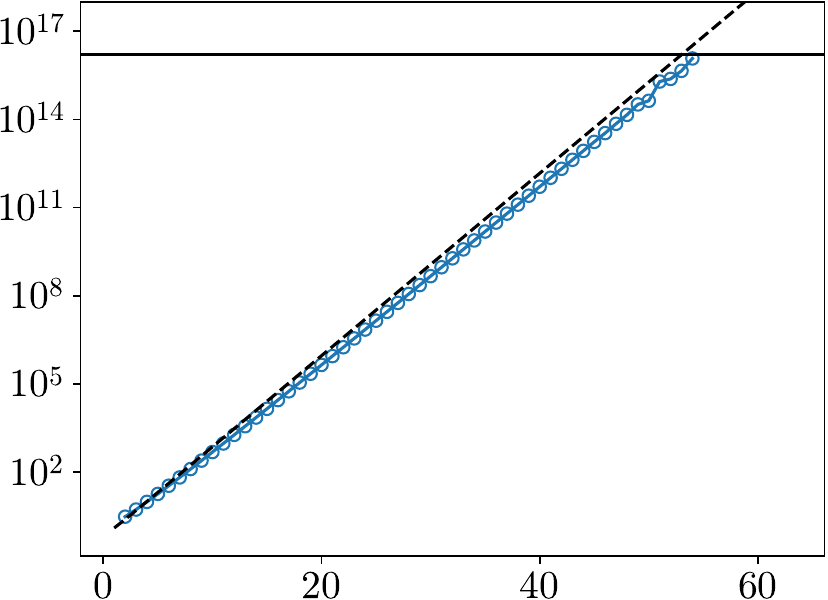}
}
\caption{A comparison between $\norm{V^{-1}}_2$ 
and $\frac{1}{\alpha_n}$, as a function of $n$,
for $\gamma=10$, $50$, $250$.}
  \label{fig:Vnorm}
\end{figure}

\section{Numerical Experiments}
\label{sec:NE}

In this section, we demonstrate the performance of our algorithm with several
numerical experiments. 
Our algorithm was implemented in Fortran 77, and compiled using the GFortran 
Compiler, version 12.2.0, with -O3 flag. All experiments were conducted on a 
laptop with 32 GB of RAM and an Intel $12$nd Gen Core i7-1270P CPU.
A demo of our approximation scheme is provided in  
\url{https://doi.org/10.5281/zenodo.8323315}.

All the experiments presented in this section are conducted using 
the practical version of the numerical algorithm described in 
\Cref{sec:sum}.

\subsection{Approximation Over Varying Values of $n$}
In this subsection, we approximate functions of the form
$f(x)=\int_a^b x^{\mu} \sigma(\mu) \, d\mu$, $x \in [0,1]$,
for the following cases of $\sigma(\mu)$:
\begin{align}
\sigma_1(\mu)&=\frac{1}{\mu}, \label{eq:sigma1}\\
\sigma_2(\mu)&=\sin(12\mu), \label{eq:sigma2}\\
\sigma_3(\mu)&=e^{-10 \mu}, \label{eq:sigma3}\\
\sigma_4(\mu)&=\mu \sin(\mu)\label{eq:sigma4}.
\end{align}

We estimate 
$\|f-\hat{f}_N\|_{L^{\infty}[0,1]}$ by 
evaluating $f$ and $\hat{f}_N$ at $2000$ uniformly distributed points over $[0,1]$,
and finding
the maximum error between $f$ and  
$\hat{f}_N$ at those points.
We repeat the experiments for
$\gamma=10$, $50$, $250$, and plot
${\|f-\hat{f}_N\|_{L^{\infty}[0,1]}}/{|\sigma|}$.
The results are displayed
in \Cref{fig:exp1,fig:exp2,fig:exp3}.


It is evident that
${\|f-\hat{f}_N\|_{L^{\infty}[0,1]}}/{|\sigma|}$ remains 
bounded by $\alpha_n$, as shown in \Cref{sec:sum},
until it reaches a stabilized level that is close to machine 
precision multiplied by some small constant.
Since $\{\alpha_i\}_{i=0, 1, \ldots , \infty}$ decays exponentially, the approximation 
exhibits an exponential rate of convergence in $N$.

\begin{figure}[!ht]
\centering
\subfloat[$f(x)=\int_1^{10} x^{\mu}\sigma_1(\mu) \,d\mu$]{%
\includegraphics[scale=0.47]{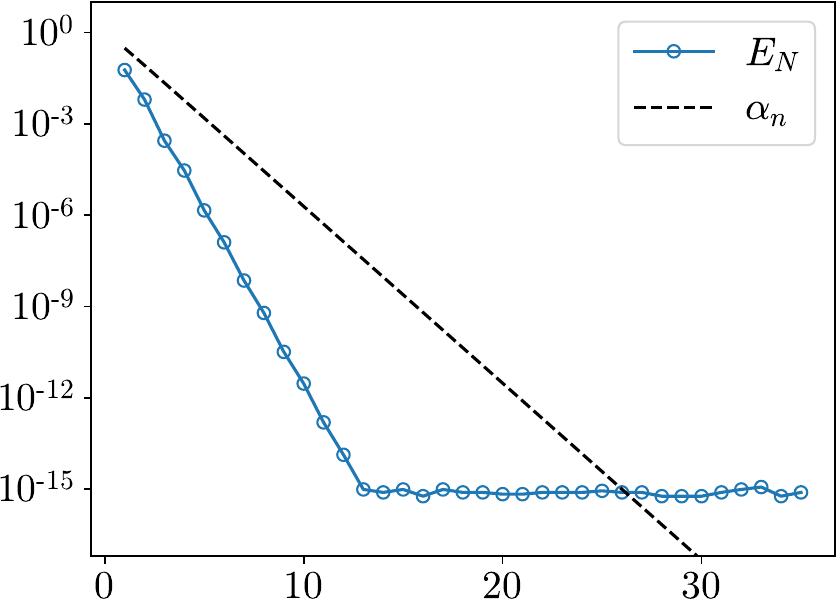}
}    
\subfloat[$f(x)=\int_1^{10} x^{\mu}\sigma_2(\mu) \,d\mu$]{%
\includegraphics[scale=0.47]{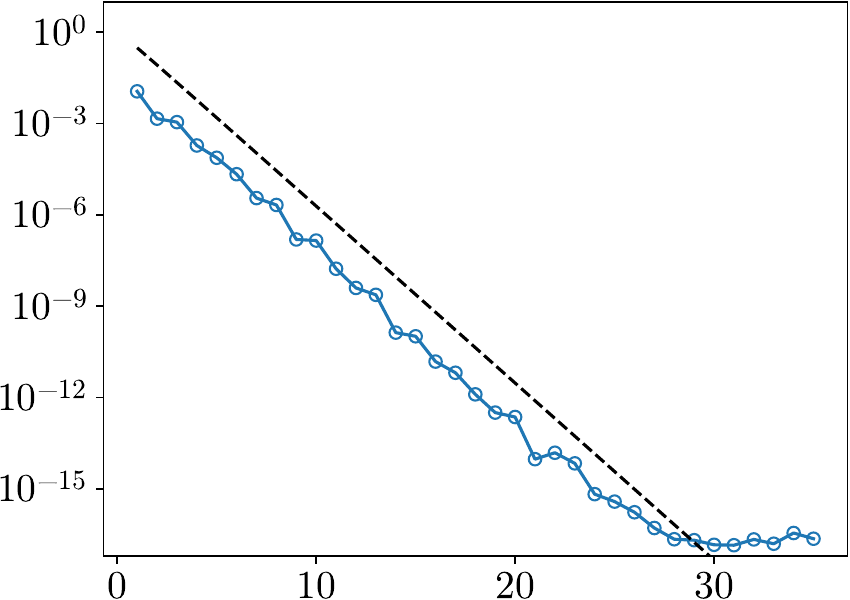}
}
\hfill
\subfloat[$f(x)=\int_1^{10} x^{\mu}\sigma_3(\mu) \,d\mu$]{%
\includegraphics[scale=0.47]{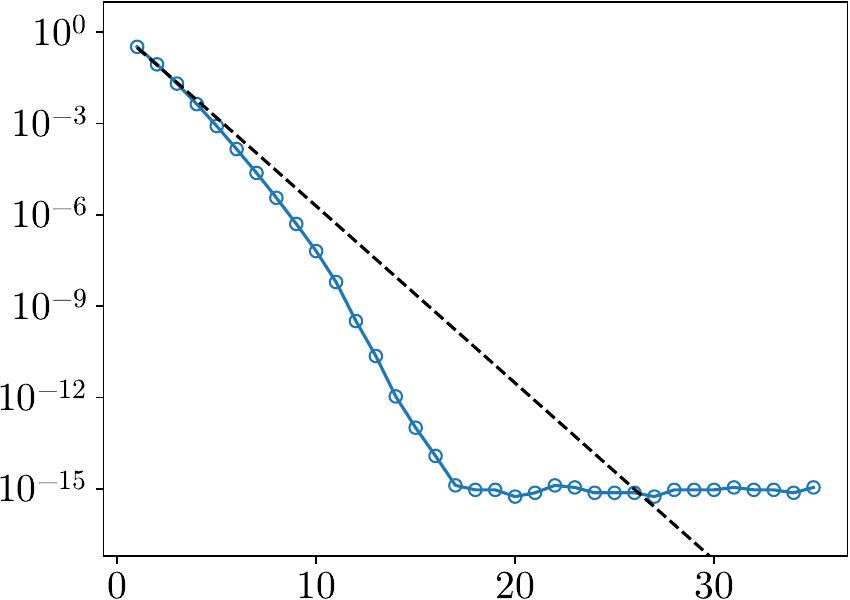}
}    
\subfloat[$f(x)=\int_1^{10} x^{\mu}\sigma_4(\mu) \,d\mu$]{%
\includegraphics[scale=0.47]{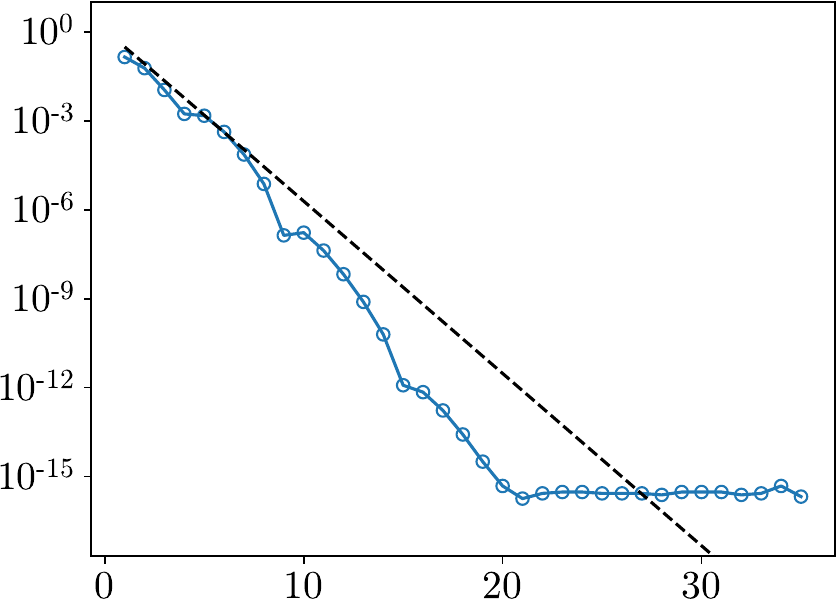}
}    
\caption{The $L^{\infty}$ approximation error over $[0,1]$,
$E_N:=\frac{\|f-\hat{f}_N\|_{L^{\infty}[0,1]}}{|\sigma|}$,
as a function of $n$, for $\gamma=10$.
}
  \label{fig:exp1}
\end{figure}
\begin{figure}[!ht]
\centering
\subfloat[$f(x)=\int_1^{50} x^{\mu}\sigma_1(\mu) \,d\mu$]{%
\includegraphics[scale=0.47]{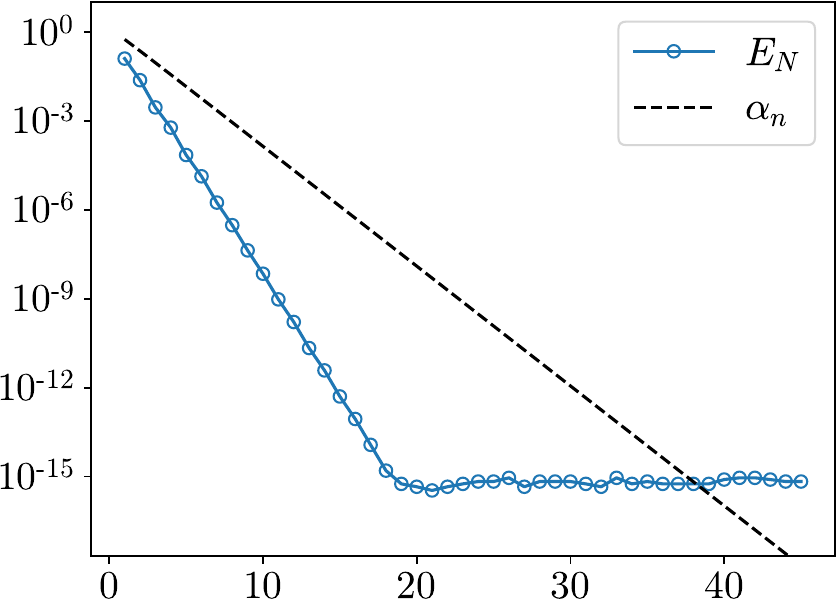}
}    
\subfloat[$f(x)=\int_1^{50} x^{\mu}\sigma_2(\mu) \,d\mu$]{%
\includegraphics[scale=0.47]{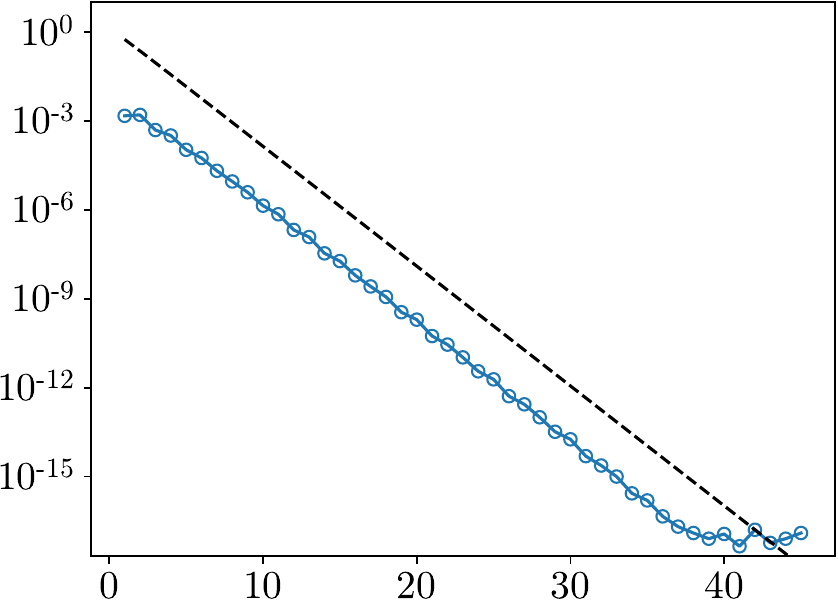}
}
\hfill
\subfloat[$f(x)=\int_1^{50} x^{\mu}\sigma_3(\mu) \,d\mu$]{%
\includegraphics[scale=0.47]{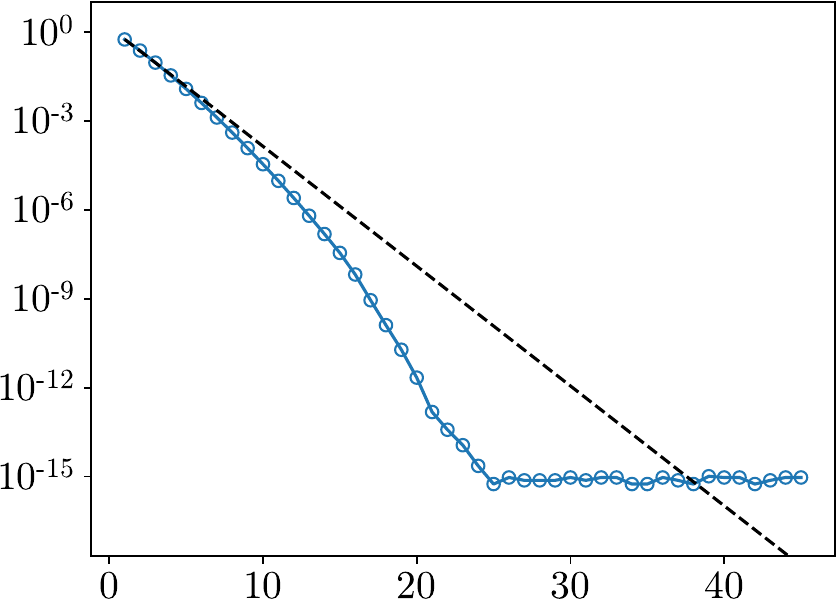}
}    
\subfloat[$f(x)=\int_1^{50} x^{\mu}\sigma_4(\mu) \,d\mu$]{%
\includegraphics[scale=0.47]{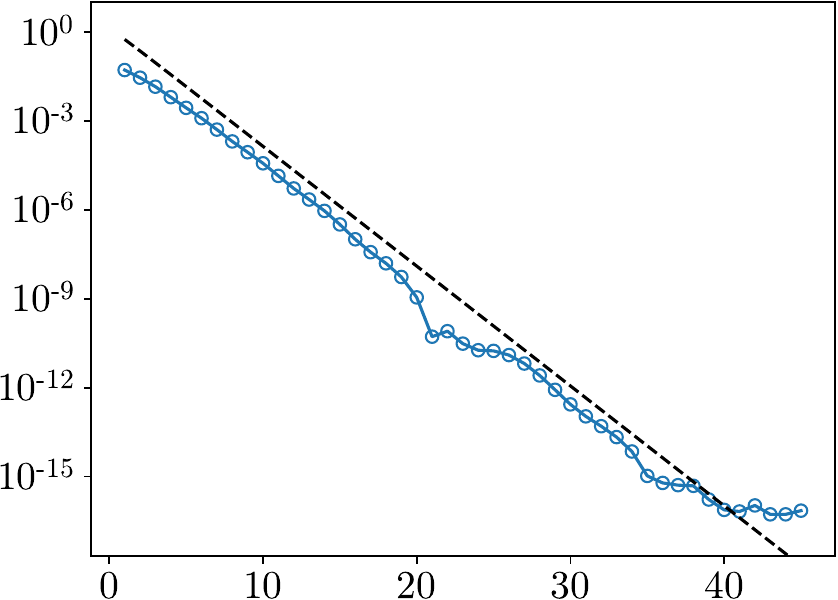}
}    
\caption{The $L^{\infty}$ approximation error over $[0,1]$,
$E_N:=\frac{\|f-\hat{f}_N\|_{L^{\infty}[0,1]}}{|\sigma|}$,
as a function of $n$, for $\gamma=50$.
}
  \label{fig:exp2}
\end{figure}
\begin{figure}[!ht]
\centering
\subfloat[$f(x)=\int_1^{250} x^{\mu}\sigma_1(\mu) \,d\mu$]{%
\includegraphics[scale=0.47]{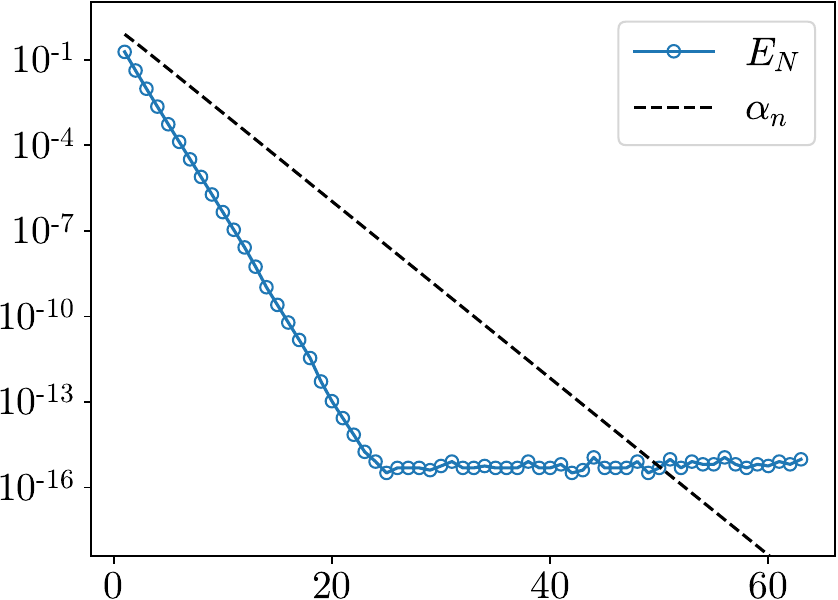}
}    
\subfloat[$f(x)=\int_1^{250} x^{\mu}\sigma_2(\mu) \,d\mu$]{%
\includegraphics[scale=0.47]{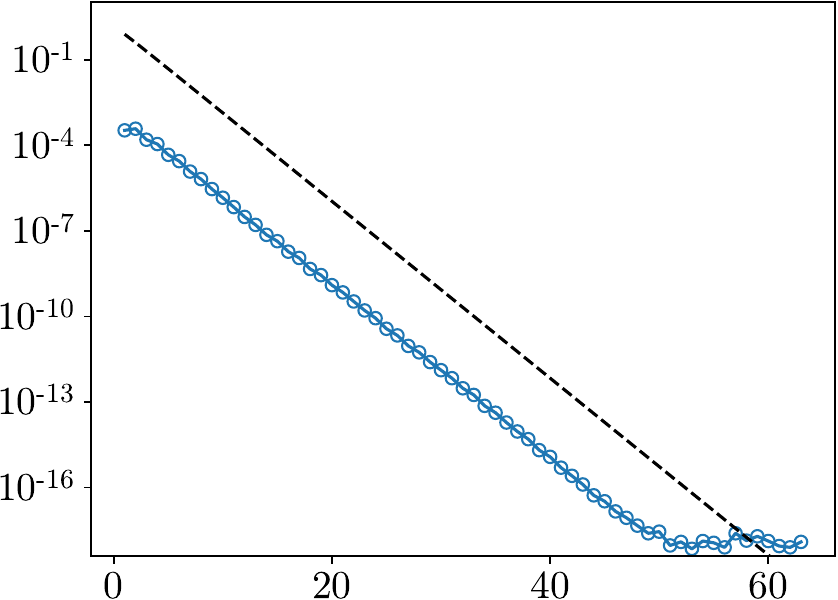}
}
\hfill
\subfloat[$f(x)=\int_1^{250} x^{\mu}\sigma_3(\mu) \,d\mu$]{%
\includegraphics[scale=0.47]{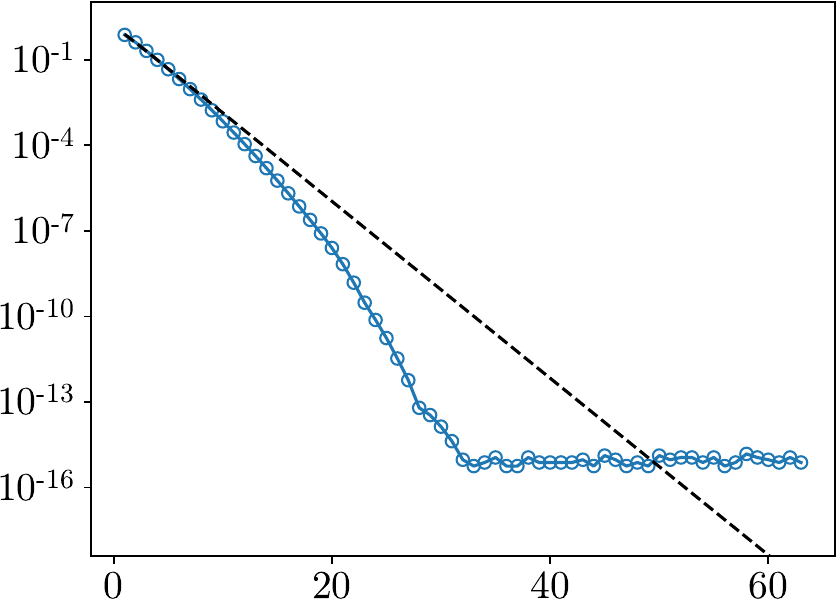}
}    
\subfloat[$f(x)=\int_1^{250} x^{\mu}\sigma_4(\mu) \,d\mu$]{%
\includegraphics[scale=0.47]{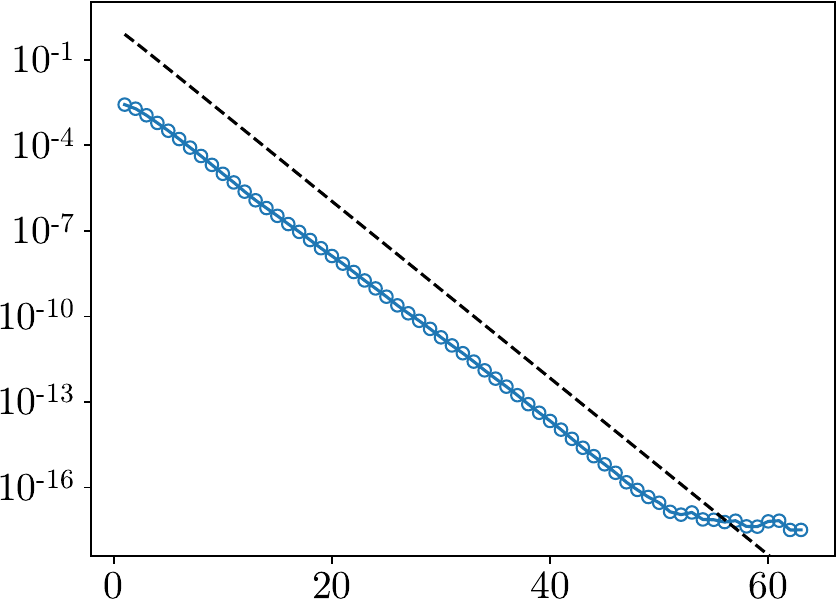}
}    
\caption{The $L^{\infty}$ approximation error over $[0,1]$,
$E_N:=\frac{\|f-\hat{f}_N\|_{L^{\infty}[0,1]}}{|\sigma|}$,
as a function of $n$, for $\gamma=250$.
}
  \label{fig:exp3}
\end{figure}

\subsection{Approximation of Non-integer Powers}
\label{sec:nonint}

In this subsection, 
our goal is to approximate functions 
of the form 
$f(x)=\int_a^b x^{\mu} \sigma(\mu) \, d\mu$, $x \in [0,1]$,
with 
\begin{align}
\sigma_5(\mu)= \delta(\mu-c),
\end{align}
where 
$c \in [a,b]$.
The resulting function is $f(x)=x^c$.

We
fix $N=n$, where  
$\alpha_n \approx \epsilon_0$, and
approximate such functions for 
$1000$ values of $c$ 
distributed logarithmically in the interval $[\frac{a}{1.5},1.5 b]$.
We evaluate $f$ and $\hat{f}_N$ at $1000$ uniformly distributed
points over $[0,1]$ to estimate  
${\|f-\hat{f}_N\|_{L^{\infty}[0,1]}}/{\abs{\sigma}}$.
The results
for $\gamma=10, 50, 250$ are displayed in \Cref{fig:mu}.
It can be observed that the approximation error remains 
accurate up to 
machine precision multiplied by some small constants,
for values of $c$
within the interval $[a,b]$, and grows 
significantly,
for values of $c$
outside $[a,b]$.

We further investigate the approximation error over varying values of $n$,
for $c=a$,
$\frac{a+b}{2}$, $b$ and $\gamma=10$, $50$, $250$, as shown 
in \Cref{fig:df0}. The approximation error is bounded by $\alpha_n$ multiplied
by some small constants, until it stabilizes at a level 
around machine precision.

\begin{figure}[!ht]
\centering
\subfloat[$n=28$, $\gamma=10$: $a=1$, $b=10$]{%
\includegraphics[scale=0.47]{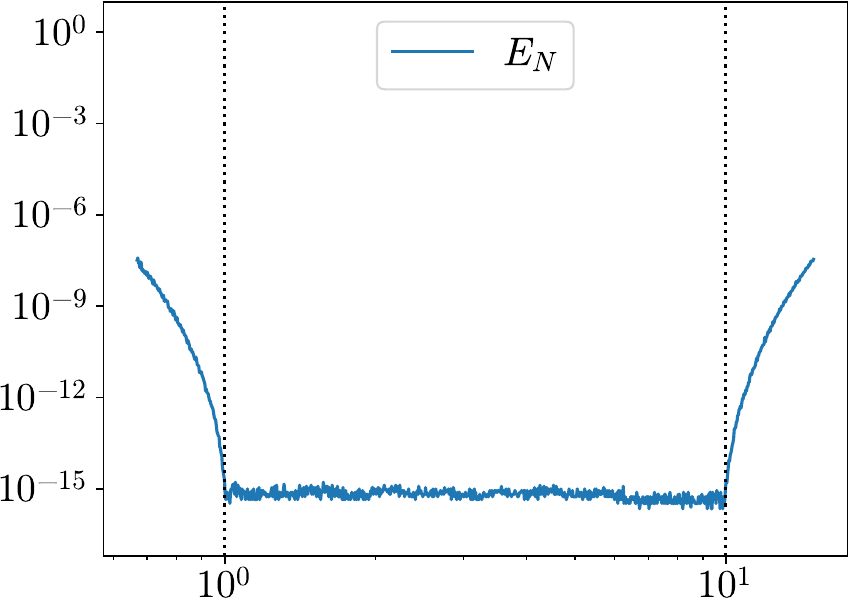}
}    

\subfloat[$n=38$, $\gamma=50$: $a=1$, $b=50$]{%
\includegraphics[scale=0.47]{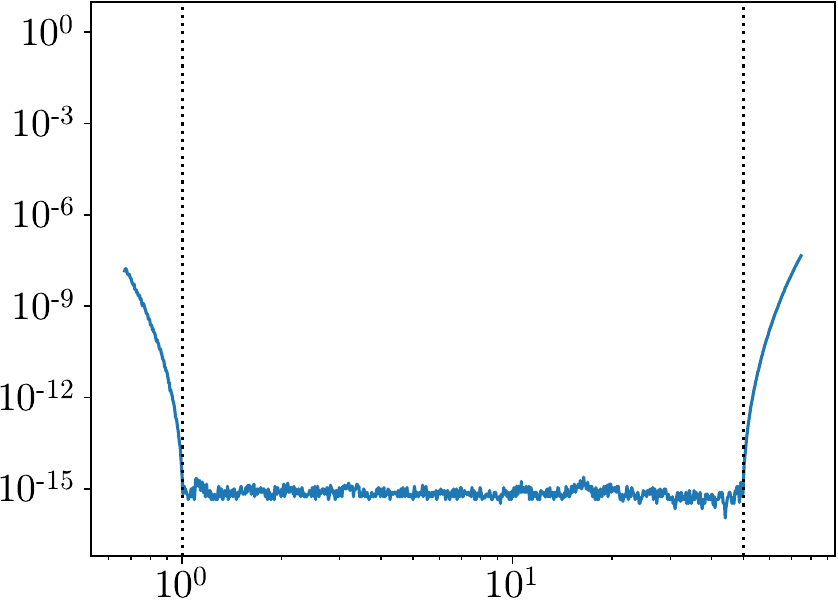}
}

\subfloat[$n=50$, $\gamma=250$: $a=1$, $b=250$]{%
\includegraphics[scale=0.47]{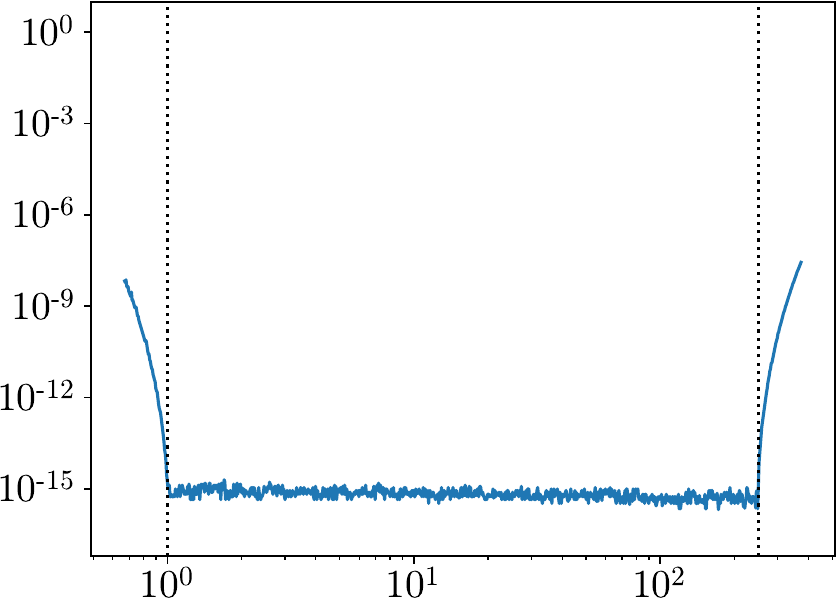}
}
\caption{The $L^{\infty}$ approximation error
of
$f(x)=\int_a^{b} x^{\mu} \sigma_5(\mu) \, d\mu=x^c$ over $[0,1]$,
$E_N:=\frac{\|f-\hat{f}_N\|_{L^{\infty}[0,1]}}{|\sigma|}$,
as a function of $c$, for a fixed $n$ such
that $\alpha_n\approx \epsilon_0$,
and for $\gamma=10$, $50$, $250$.}
  \label{fig:mu}
\end{figure}

\begin{figure}[!ht]
\centering
\subfloat[$\gamma=10$: $a=1$, $b=10$]{%
\includegraphics[scale=0.47]{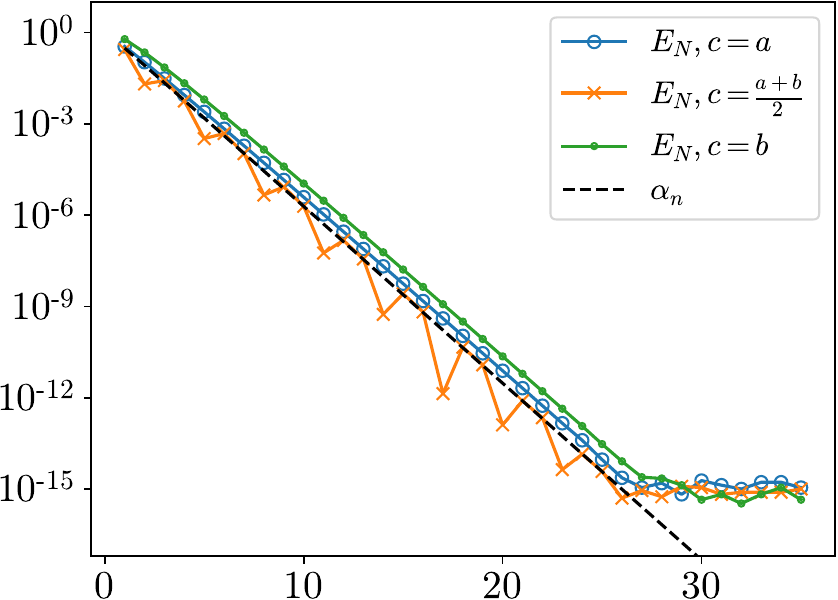}
}    

\subfloat[$\gamma=50$: $a=1$, $b=50$]{%
\includegraphics[scale=0.47]{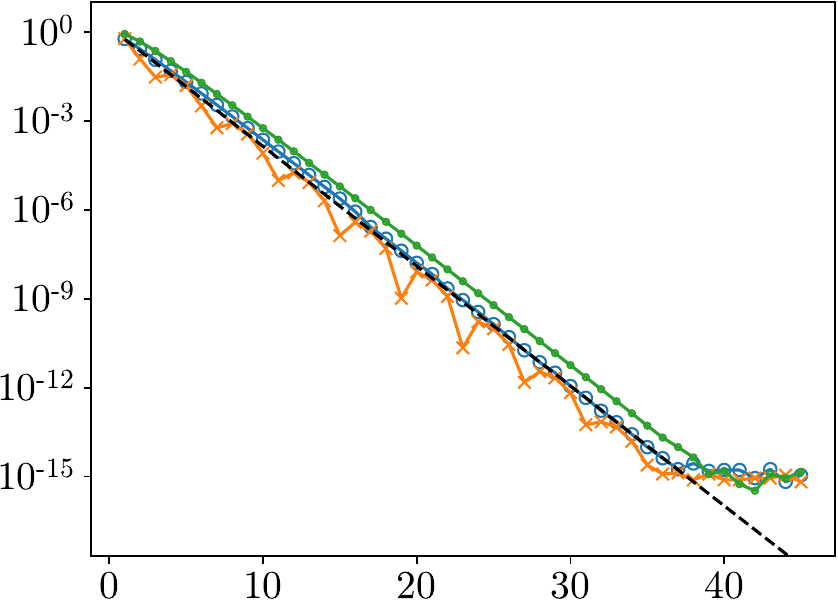}
}

\subfloat[$\gamma=250$: $a=1$, $b=250$]{%
\includegraphics[scale=0.47]{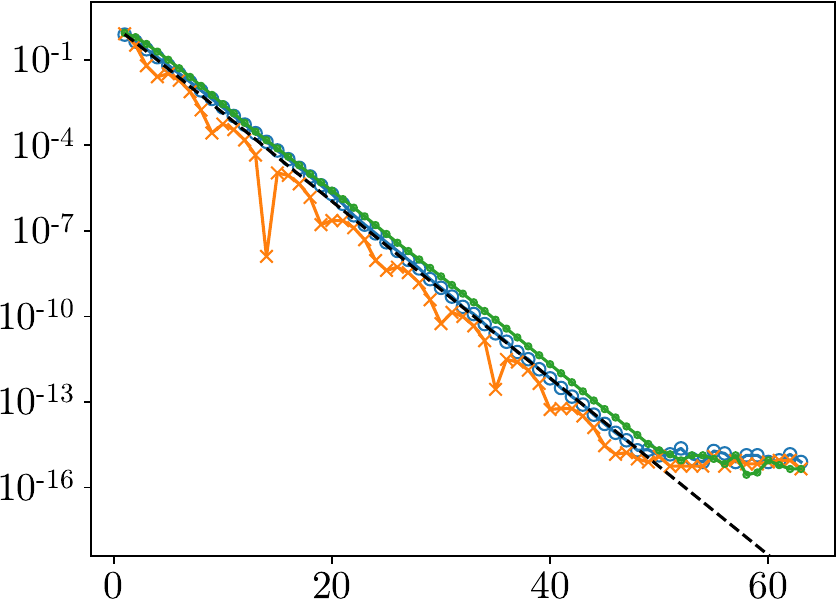}
}
\caption{The $L^{\infty}$ approximation error
of
$f(x)=\int_a^{b} x^{\mu} \sigma_5(\mu) \, d\mu=x^c$ over $[0,1]$,
$E_N:=\frac{\|f-\hat{f}_N\|_{L^{\infty}[0,1]}}{|\sigma|}$,
as a function of $n$,
for $c=a$, $\frac{a+b}{2}$, $b$, and
$\gamma=10$, $50$, $250$.}
\label{fig:df0}
\end{figure}

\subsection{Approximation in the Case of Distributions}
In this subsection, we assume 
$\sigma \in \mathcal{D}'(\R)$ has the form 
\begin{align}
\label{eq:sigmadef}
\sigma_6(\mu)&=(-1)^k \delta^{(k)}(\mu-c),
\end{align}
where $k\geq 0$ is an integer,
$c \in [a,b]$, and $\delta(t)$ is the Dirac delta function.
The resulting function is $f(x)=x^c (\log{x})^k$.
We evaluate $f$ and $\hat{f}_N$ at $2000$ uniformly distributed points in 
$[0,1]$ to estimate
${\|f-\hat{f}_N\|_{L^{\infty}[0,1]}}/{\norm{\sigma}_{C^m([a,b])^*}}$.
The results for $k=1$, $2$, \ldots , $6$, $c=a$, $\frac{a+b}{2}$, $b$,
and $\gamma=10$, $50$,
$250$ are shown in \Cref{fig:exp4,fig:exp5,fig:exp6}. 

In contrast to the previous cases 
where $\sigma$ is a signed Radon measure, 
the approximation error can increase significantly with $k$.
However, the approximation error is still bounded by 
$(\epsilon+\alpha_n)\cdot \max_{0\leq i \leq n-1} \bnorm{u_i\bigl{(}\tfrac{t-a}{b-a}
\bigr{)}}_{C^k([a,b])}$, as stated in \Cref{thm:globalerrdist}.
Furthermore,
we observe that the error grows with $k$, and 
when $c=a$, the error is closely aligned with the estimated bound, since
the function is more singular for smaller $c$ and the approximation error tends to be
larger.

\begin{figure}[!ht]
\centering
\subfloat[$k=1$]{%
\includegraphics[scale=0.47]{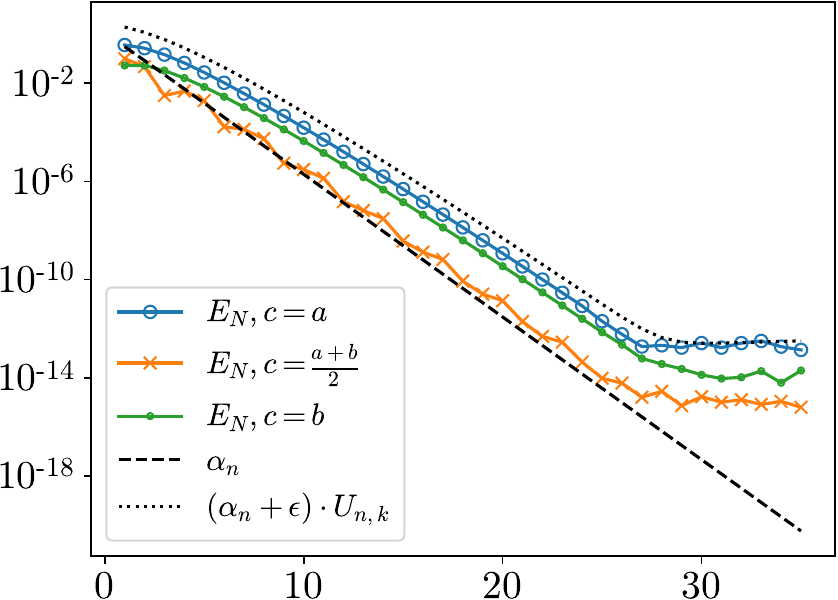}
}    
\subfloat[$k=2$]{%
\includegraphics[scale=0.47]{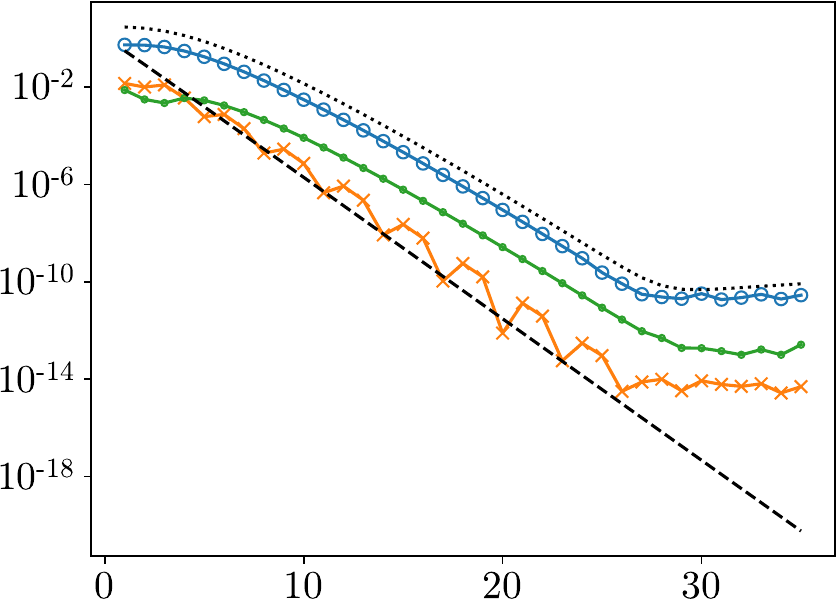}
}
\hfill
\subfloat[$k=3$]{%
\includegraphics[scale=0.47]{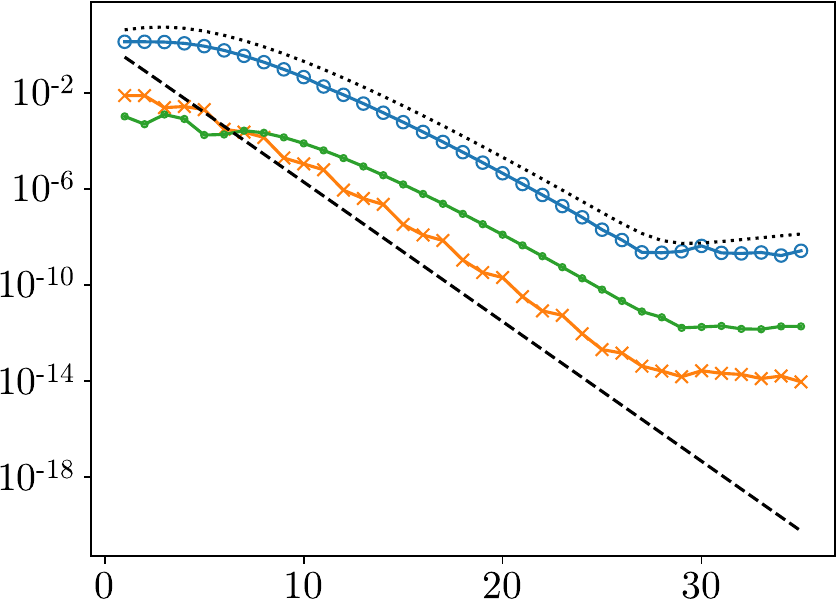}
}
\subfloat[$k=4$]{%
\includegraphics[scale=0.47]{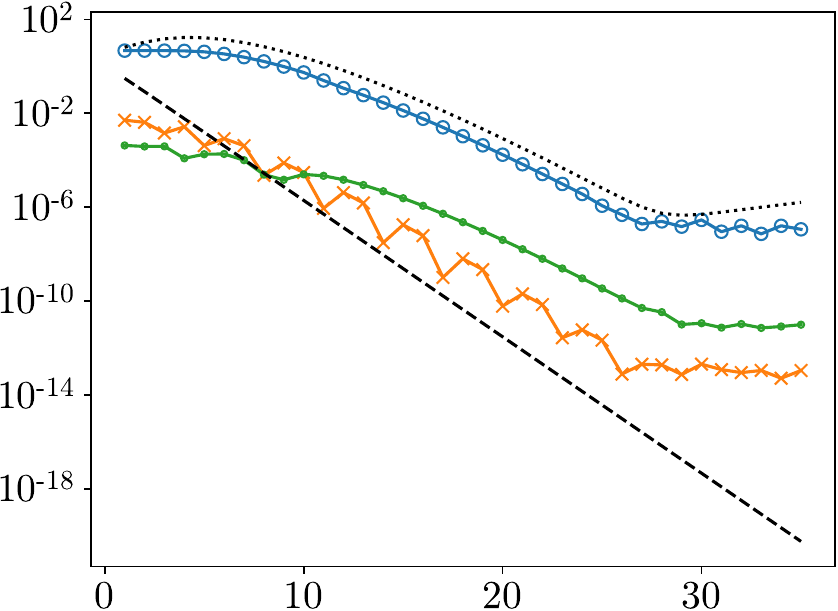}
}
\hfill
\subfloat[$k=5$]{%
\includegraphics[scale=0.47]{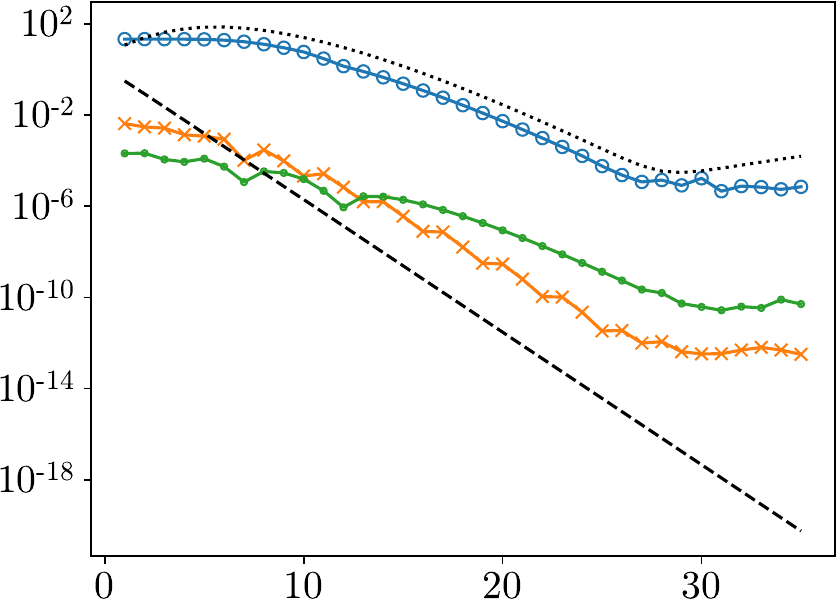}
}
\subfloat[$k=6$]{%
\includegraphics[scale=0.47]{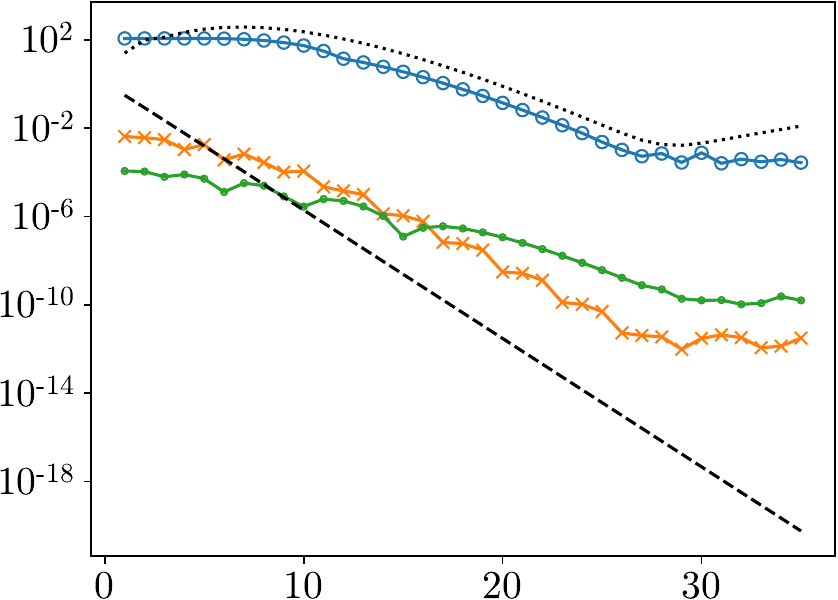}
}
\caption{The $L^{\infty}$ approximation error of 
$f(x)=\int_1^{10} x^{\mu} \sigma_6(\mu) \, d\mu=x^c (\log{x})^k$
over $[0,1]$,
$E_N:=\frac{\|f-\hat{f}_N\|_{L^{\infty}[0,1]}}{\norm{\sigma}_{C^m([a,b])^*}}$,
as a function of $n$, 
for $c=a$, $\frac{a+b}{2}$, $b$, $k=1$, \ldots, $6$, and
$\gamma=10$.
$U_{n,k} :=\max_{0\leq i \leq n-1} \bnorm{u_i\bigl{(}\tfrac{t-a}{b-a}
\bigr{)}}_{C^k([a,b])}$.}
\label{fig:exp4}
\end{figure}

\begin{figure}[!ht]
\centering
\subfloat[$k=1$]{%
\includegraphics[scale=0.47]{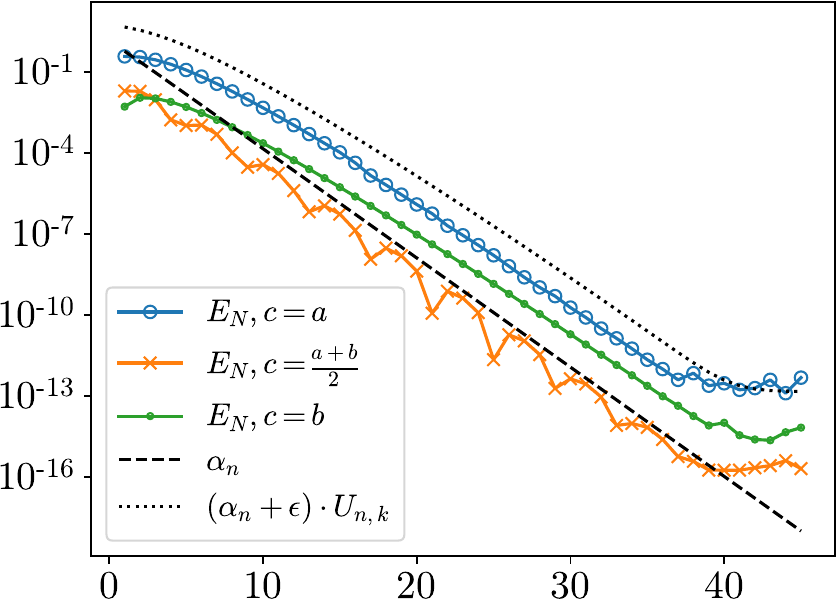}
}    
\subfloat[$k=2$]{%
\includegraphics[scale=0.47]{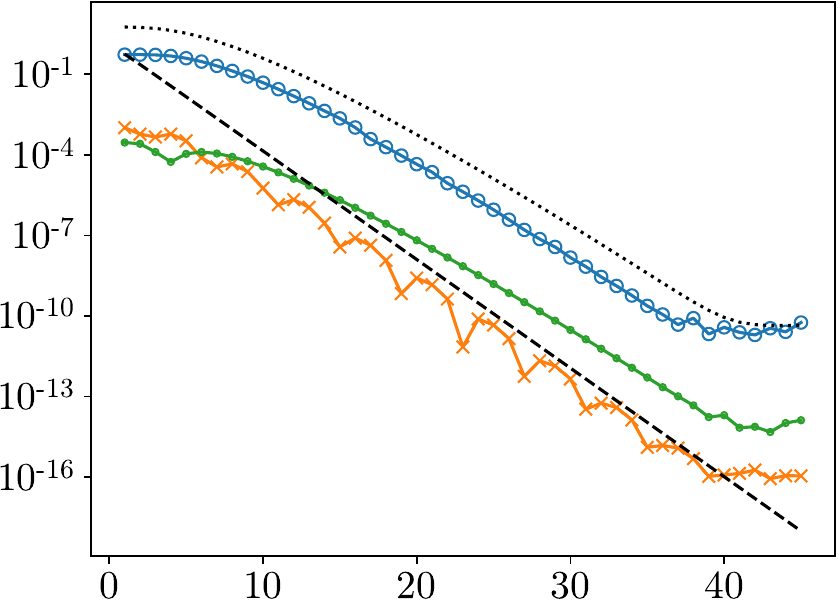}
}
\hfill
\subfloat[$k=3$]{%
\includegraphics[scale=0.47]{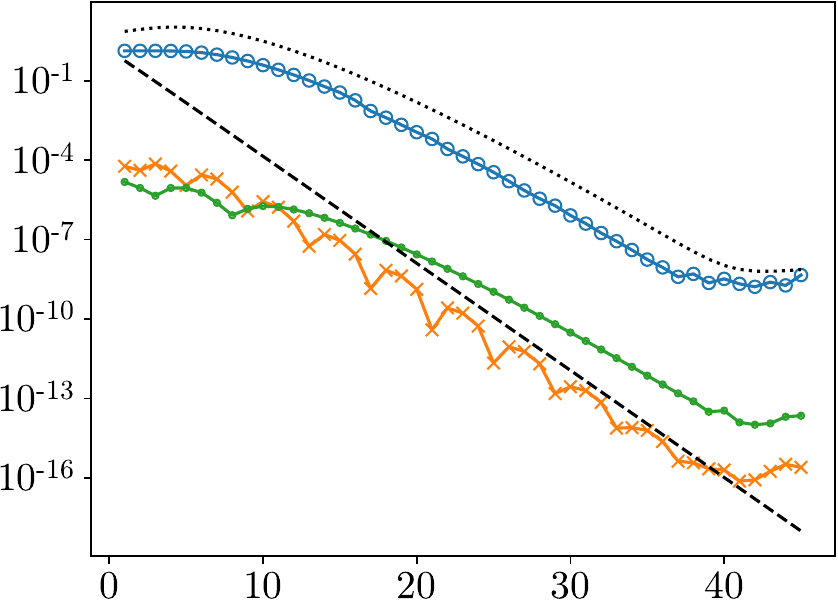}
}
\subfloat[$k=4$]{%
\includegraphics[scale=0.47]{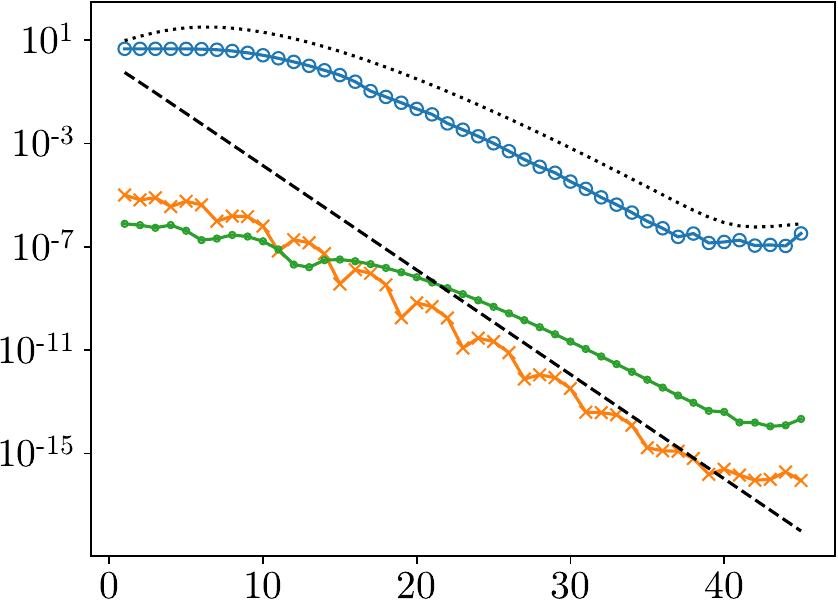}
}
\hfill
\subfloat[$k=5$]{%
\includegraphics[scale=0.47]{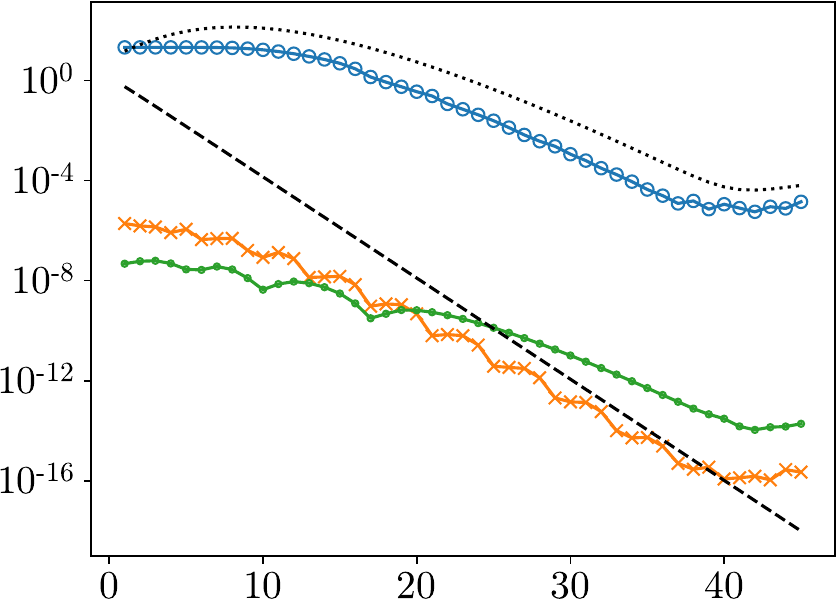}
}
\subfloat[$k=6$]{%
\includegraphics[scale=0.47]{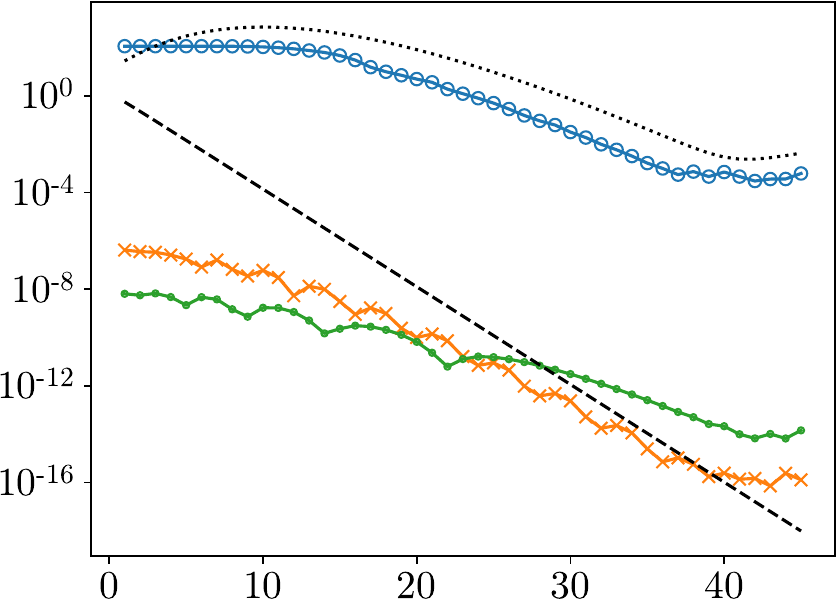}
}
\caption{The $L^{\infty}$ approximation error of 
$f(x)=\int_1^{50} x^{\mu} \sigma_6(\mu) \, d\mu=x^c (\log{x})^k$
over $[0,1]$,
$E_N:=\frac{\|f-\hat{f}_N\|_{L^{\infty}[0,1]}}{\norm{\sigma}_{C^m([a,b])^*}}$,
as a function of $n$, 
for $c=a$, $\frac{a+b}{2}$, $b$, $k=1$, \ldots, $6$, and
$\gamma=50$.
$U_{n,k} :=\max_{0\leq i \leq n-1} \bnorm{u_i\bigl{(}\tfrac{t-a}{b-a}
\bigr{)}}_{C^k([a,b])}$.}
\label{fig:exp5}
\end{figure}

\begin{figure}[!ht]
\centering
\subfloat[$k=1$]{%
\includegraphics[scale=0.47]{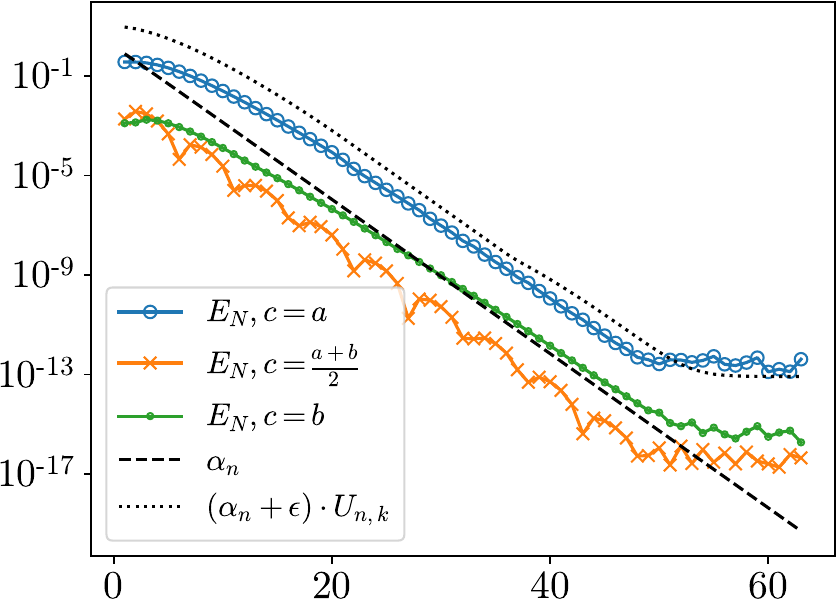}
}    
\subfloat[$k=2$]{%
\includegraphics[scale=0.47]{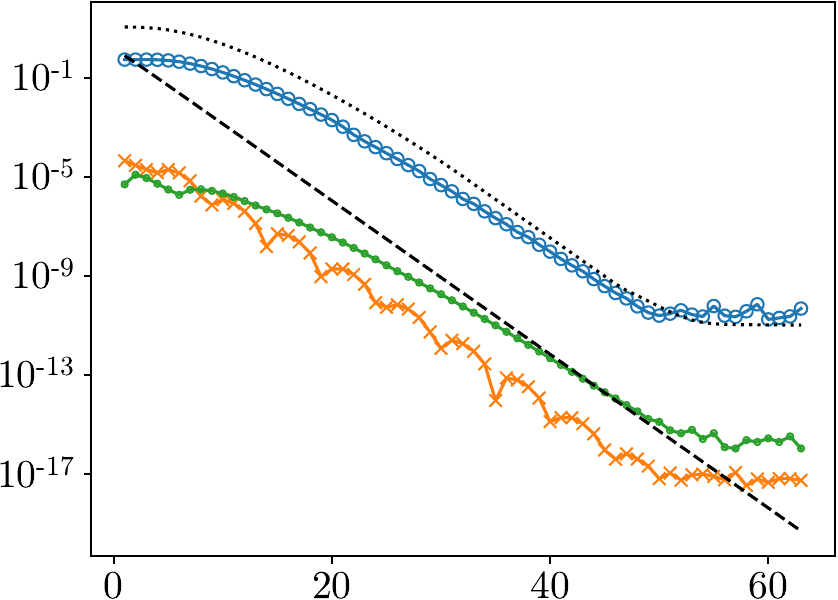}
}
\hfill
\subfloat[$k=3$]{%
\includegraphics[scale=0.47]{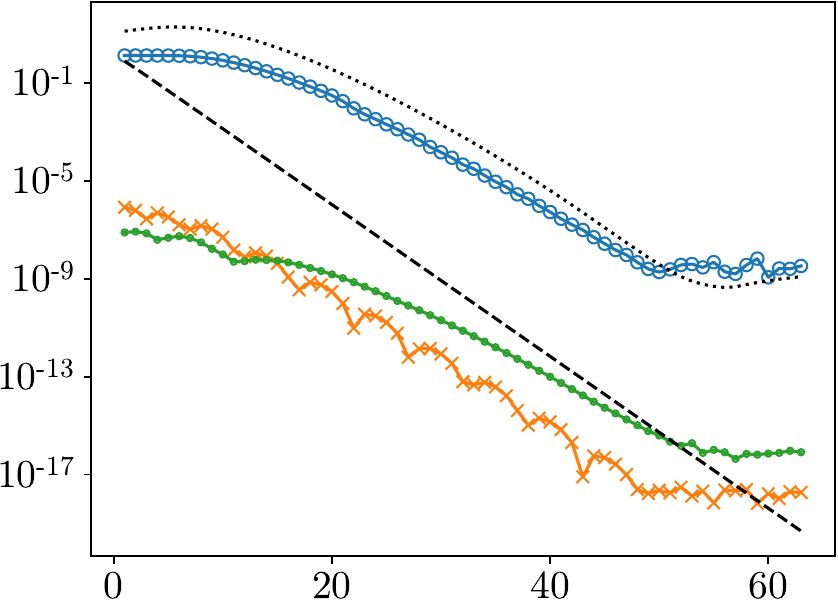}
}
\subfloat[$k=4$]{%
\includegraphics[scale=0.47]{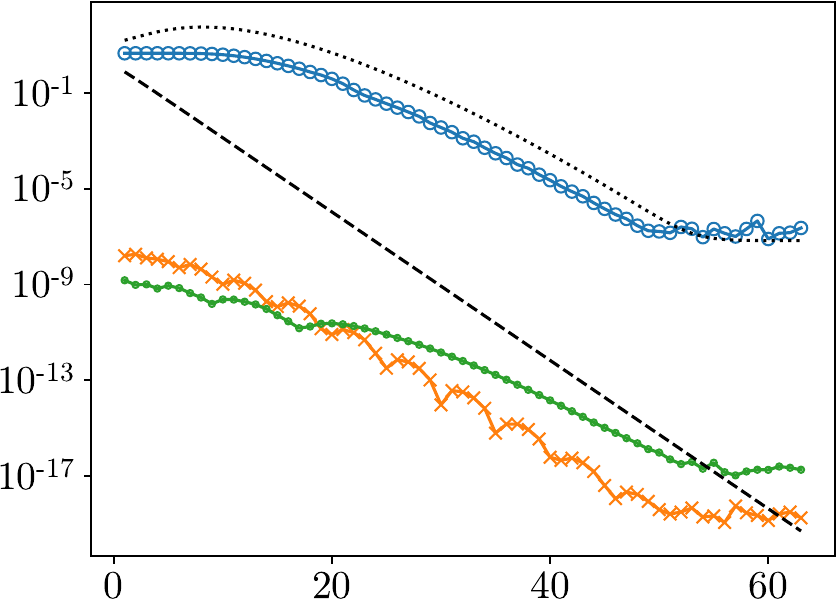}
}
\hfill
\subfloat[$k=5$]{%
\includegraphics[scale=0.47]{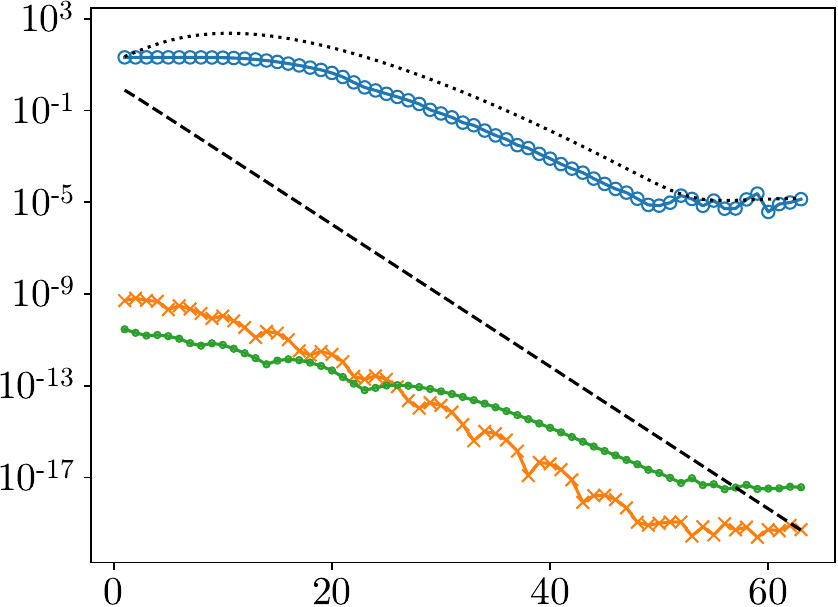}
}
\subfloat[$k=6$]{%
\includegraphics[scale=0.47]{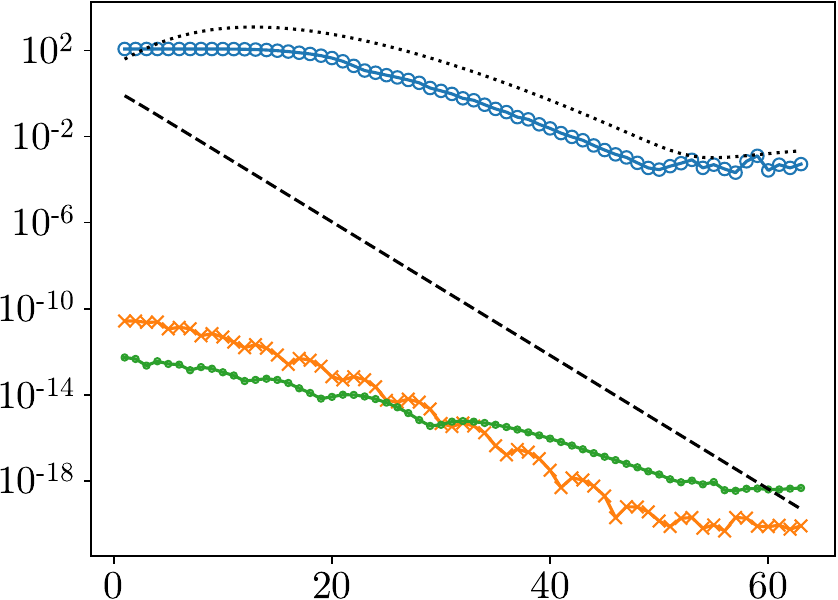}
}
\caption{The $L^{\infty}$ approximation error of 
$f(x)=\int_1^{250} x^{\mu} \sigma_6(\mu) \, d\mu=x^c (\log{x})^k$
over $[0,1]$,
$E_N:=\frac{\|f-\hat{f}_N\|_{L^{\infty}[0,1]}}{\norm{\sigma}_{C^m([a,b])^*}}$,
as a function of $n$, 
for $c=a$, $\frac{a+b}{2}$, $b$, $k=1$, \ldots, $6$, and
$\gamma=250$.
$U_{n,k} :=\max_{0\leq i \leq n-1} \bnorm{u_i\bigl{(}\tfrac{t-a}{b-a}
\bigr{)}}_{C^k([a,b])}$.}
\label{fig:exp6}
\end{figure}

\subsection{Approximation Over a Simple Arc in the Complex Plane}
In this subsection, we investigate the performance of our algorithm 
on simple and smooth arcs in the complex plane. 
Suppose that $\tilde{\gamma}\colon [0,1] \rightarrow \mathbb{C}$,
and let $\Gamma = \tilde{\gamma}([0,1])$. We replace the interpolation matrix V in \Cref{eq:V} 
by a modified interpolation matrix 
$V_{\Gamma}$, defined by 
\begin{align}
\label{eq:V2}
V_{\Gamma}=
\begin{pmatrix}
\tilde{\gamma}(x_1)^{t_1}& \tilde{\gamma}(x_1)^{t_2} & \   \ldots& \ \tilde{\gamma}(x_1)^{t_N}\\
\tilde{\gamma}(x_2)^{t_1}& \tilde{\gamma}(x_2)^{t_2} & \   \ldots& \ \tilde{\gamma}(x_2)^{t_N}\\
\vdots & \vdots& \  \ddots &   \vdots\\
\tilde{\gamma}(x_N)^{t_1}& \tilde{\gamma}(x_N)^{t_2} & \   \ldots& \ \tilde{\gamma}(x_N)^{t_N}
\end{pmatrix} \in \mathbb{C}^{N \times N}.
\end{align}
Specifically, we consider the arcs 
$\tilde{\gamma}(t)=t+ \alpha i(t^2-t)$,
for $\alpha=0.8$, $1.6$ and $2.4$, which are plotted in \Cref{fig:arc}.
Our goal is to approximate functions of the form
\begin{align}
f_{\Gamma}(t):=\int_a^b \tilde{\gamma}(t)^{\mu} \sigma(\mu) \, d\mu,
\end{align}
over the arcs $\tilde{\gamma}(t)$, where $t \in [0,1]$. 
We apply the algorithm
to the functions
$f_{\Gamma}(t)$ where
$\sigma(\mu)$ has the forms
$\sigma_3(\mu)$ and $\sigma_4(\mu)$, as defined in
\Cref{eq:sigma3} and \Cref{eq:sigma4}, respectively.
The experiments are repeated for 
$\gamma=10,50,250$, and the
results are displayed in \Cref{fig:cexp1,fig:cexp2,fig:cexp3}.

We also
investigate the approximation errors for non-integer powers
$f_{\Gamma}(t)=\tilde{\gamma}(t)^c$ over
the arcs $\tilde{\gamma}(t)$, where $c \in [\frac{a}{1.5}, 1.5 b]$,
following the same procedure as the one described in 
\Cref{sec:nonint}.
The results are displayed in \Cref{fig:cmu}.

By analyzing the approximation errors over $\tilde{\gamma}(t)$, for different values of
$\alpha$, we observe that the approximation error grows with $\alpha$, 
and depends on the specific functions being approximated.
Generally, when $\gamma$ is small, the approximation error grows only slightly as the arc
becomes more curved, while for large $\gamma$, it is possible for 
the approximation error to grow significantly larger than $\alpha_n$.
When the arc is slightly curved, the approximation performs similarly to 
the cases where $\tilde{\gamma}(t)=[0,1]$,
with the error bounded by $\alpha_n$.

\begin{figure}[!h]
  \centering
  \includegraphics[scale=0.45]{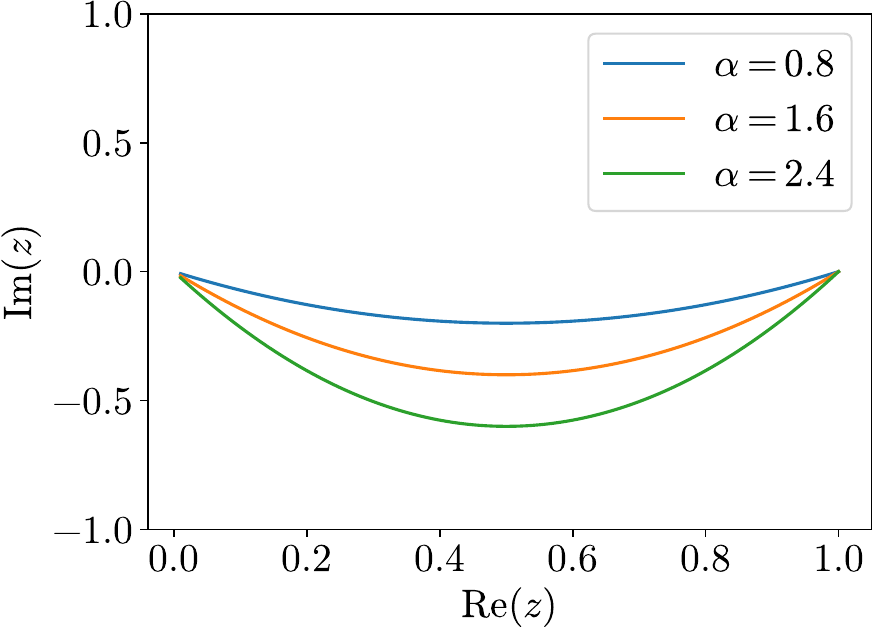}
  \caption{$\tilde{\gamma}(t)=t+\alpha i(t^2-t)$.}
  \label{fig:arc}
\end{figure}

\begin{figure}[!ht]
\centering
\subfloat[$f_{\Gamma}(t)=\int_1^{10} \tilde{\gamma}(t)^{\mu}\sigma_3(\mu) \,d\mu$]{%
\includegraphics[scale=0.47]{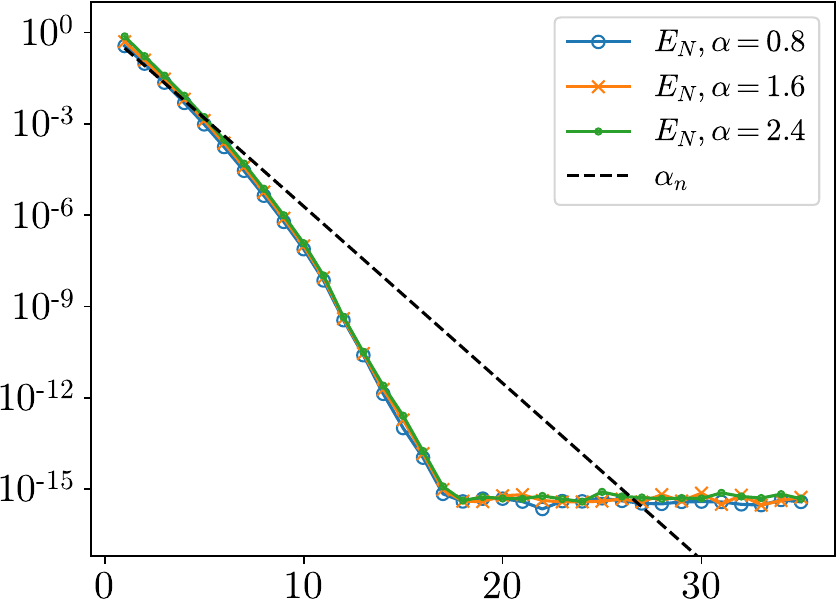}
}    
\subfloat[$f_{\Gamma}(t)=\int_1^{10} \tilde{\gamma}(t)^{\mu}\sigma_4(\mu) \,d\mu$]{%
\includegraphics[scale=0.47]{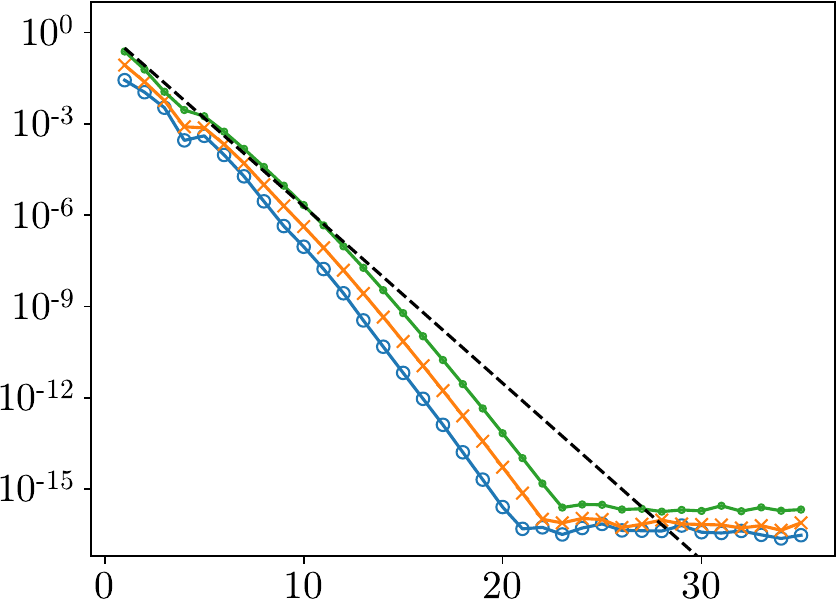}
}
\caption{The $L^{\infty}$ approximation error 
over $\tilde{\gamma}(t)$, 
$E_N:=\frac{\|f_{\Gamma}-\hat{f}_N\|_{L^{\infty}[0,1]}}{|\sigma|}$,
as a function of $n$, 
for $\alpha=0.8$, $1.6$, $2.4$, and $\gamma=10$.}
  \label{fig:cexp1}
\end{figure}

\begin{figure}[!ht]
\centering
\subfloat[$f_{\Gamma}(t)=\int_1^{50} \tilde{\gamma}(t)^{\mu}\sigma_3(\mu) \,d\mu$]{%
\includegraphics[scale=0.47]{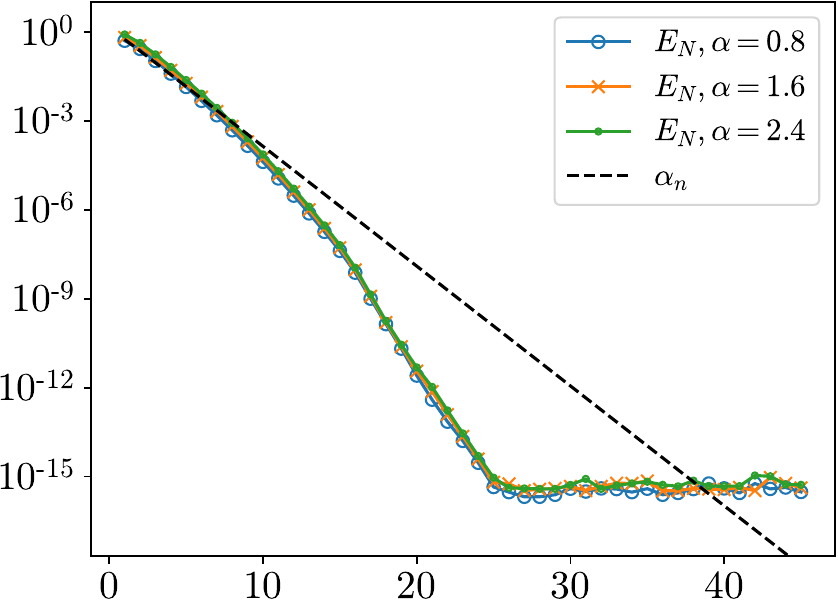}
}    
\subfloat[$f_{\Gamma}(t)=\int_1^{50} \tilde{\gamma}(t)^{\mu}\sigma_4(\mu) \,d\mu$]{%
\includegraphics[scale=0.47]{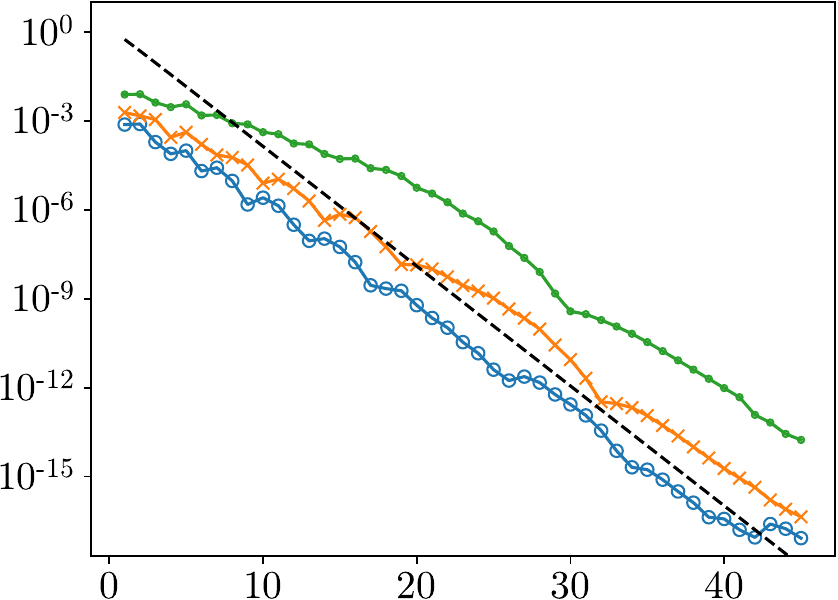}
}
\caption{The $L^{\infty}$ approximation error 
over $\tilde{\gamma}(t)$, 
$E_N:=\frac{\|f_{\Gamma}-\hat{f}_N\|_{L^{\infty}[0,1]}}{|\sigma|}$,
as a function of $n$, 
for $\alpha=0.8$, $1.6$, $2.4$, and $\gamma=50$.}
  \label{fig:cexp2}
\end{figure}

\begin{figure}[!ht]
\centering
\subfloat[$f_{\Gamma}(t)=\int_1^{250} \tilde{\gamma}(t)^{\mu}\sigma_3(\mu) \,d\mu$]{%
\includegraphics[scale=0.47]{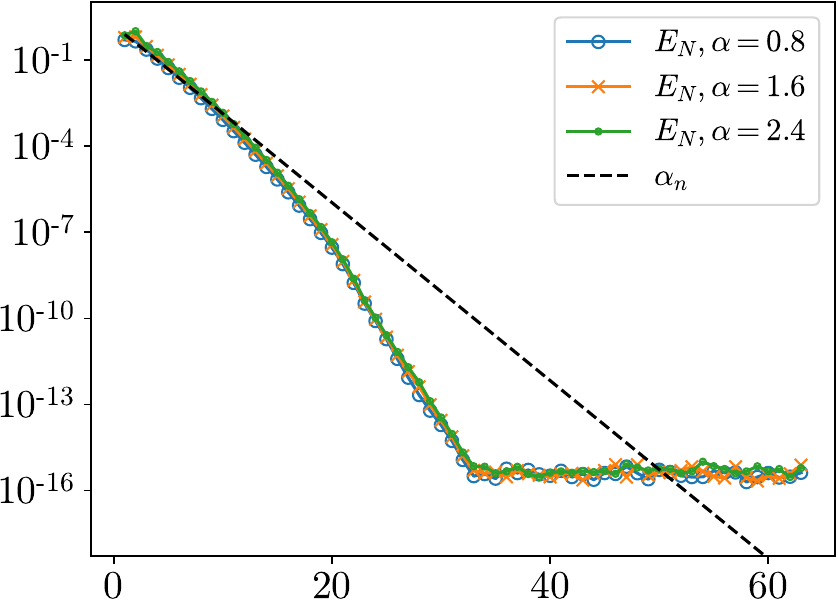}
}    
\subfloat[$f_{\Gamma}(t)=\int_1^{250} \tilde{\gamma}(t)^{\mu}\sigma_4(\mu) \,d\mu$]{%
\includegraphics[scale=0.47]{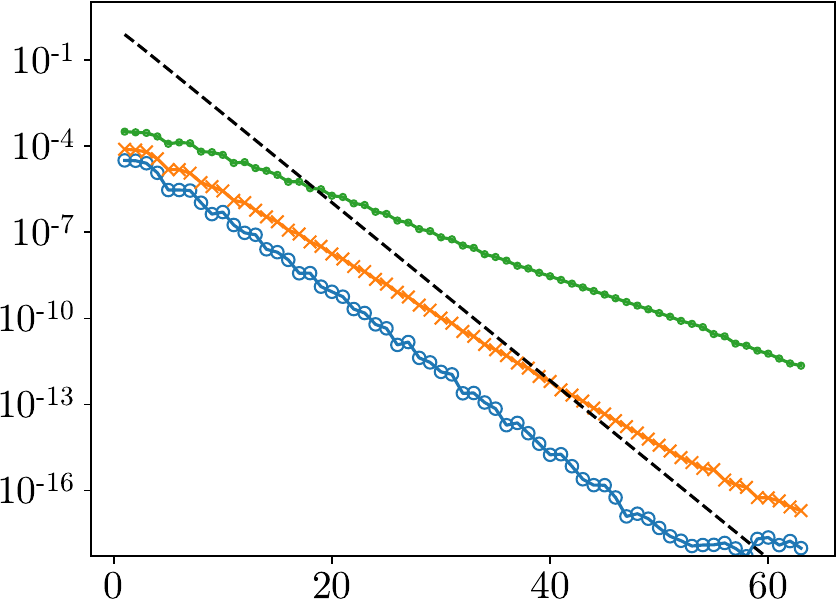}
}
\caption{The $L^{\infty}$ approximation error 
over $\tilde{\gamma}(t)$, 
$E_N:=\frac{\|f_{\Gamma}-\hat{f}_N\|_{L^{\infty}[0,1]}}{|\sigma|}$,
as a function of $n$, 
for $\alpha=0.8$, $1.6$, $2.4$, and $\gamma=250$.}
  \label{fig:cexp3}
\end{figure}

\begin{figure}[!ht]
\centering
\subfloat[$n=32$, $\gamma=10$: $a=1$, $b=10$]{%
\includegraphics[scale=0.47]{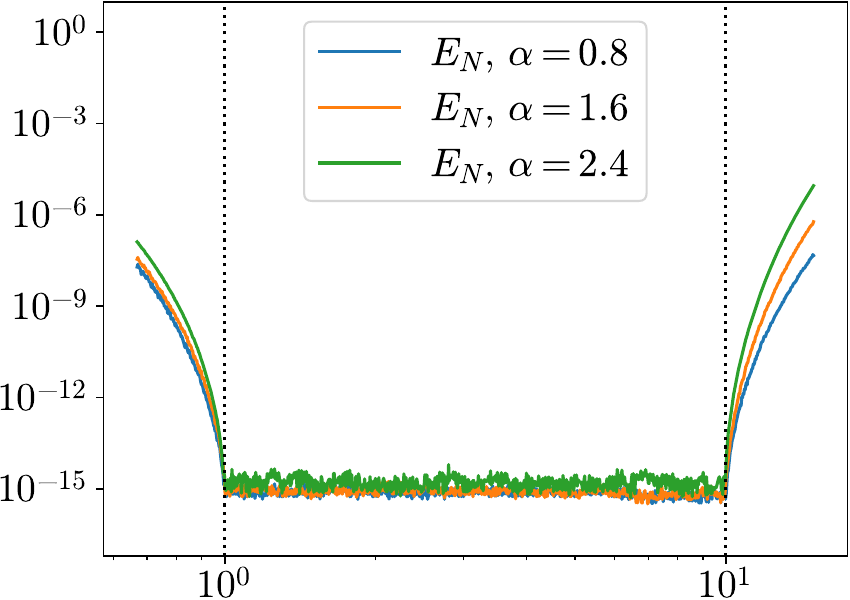}
}    

\subfloat[$n=42$, $\gamma=50$: $a=1$, $b=50$]{%
\includegraphics[scale=0.47]{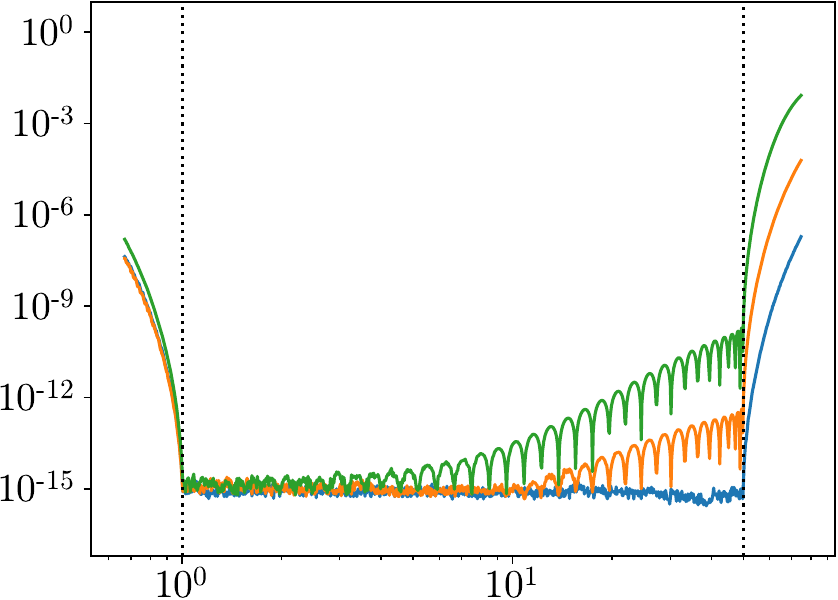}
}

\subfloat[$n=55$, $\gamma=250$: $a=1$, $b=250$]{%
\includegraphics[scale=0.47]{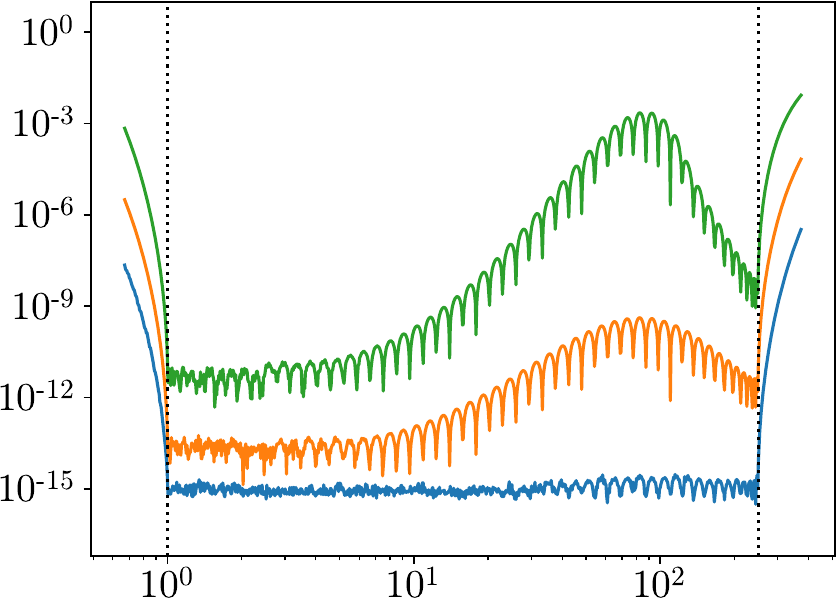}
}
\caption{The $L^{\infty}$ approximation error of
$f_{\Gamma}(t)=\int_a^{b} \tilde{\gamma}(t)^{\mu}\sigma_5(\mu) \,d\mu=
\tilde{\gamma}(t)^{c}$
over $\tilde{\gamma}(t)$,
$E_N:=\frac{\|f_{\Gamma}-\hat{f}_N\|_{L^{\infty}[0,1]}}{|\sigma|}$,
as a function of $c$, for a fixed $n$ such that $\alpha_n\approx \epsilon_0$,
$\alpha=0.8$, $1.6$, $2.4$, 
and $\gamma=10$, $50$, $250$.}
\label{fig:cmu}
\end{figure}

\subsection{Clustering of the Collocation Points}
We analyze the clustering behaviour of the collocation points by plotting
them
for $\gamma=10$, $50$, $250$, and for different values of $n$, 
as demonstrated in \Cref{fig:nx,fig:nx_all}.
We observe that
the collocation points cluster double-exponentially towards zero, for points that
are close to zero, while clustering
at a slower rate, rather than double-exponentially, for points that are away from zero. 
Notably, \Cref{fig:nx} reveals that
the closest collocation point to zero required to achieve an approximation error
of size $\alpha_n$ approaches zero at only an exponential rate as $n$ increases.
For a fixed approximation error, such as $\epsilon_0$,
the closest collocation point to zero
remains at the same distance from zero, for a fixed value of $a$ and varying values of $b$,
as shown in
\Cref{fig:nx_all}.


\begin{figure}[!ht]
\centering
\subfloat[$\gamma=10$: $a=1$, $b=10$]{%
\includegraphics[scale=0.47]{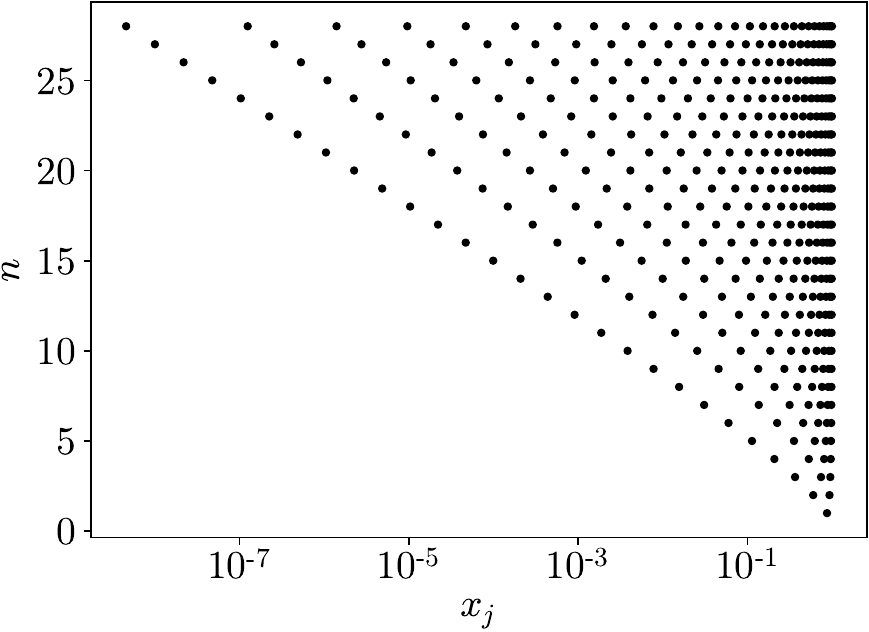}
}    

\subfloat[$\gamma=50$: $a=1$, $b=50$]{%
\includegraphics[scale=0.47]{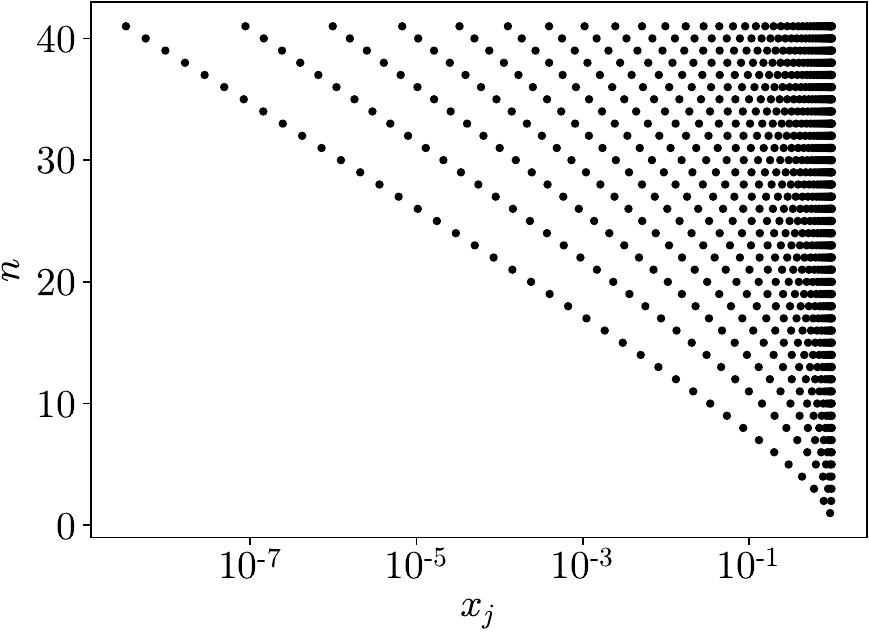}
}

\subfloat[$\gamma=250$: $a=1$, $b=250$]{%
\includegraphics[scale=0.47]{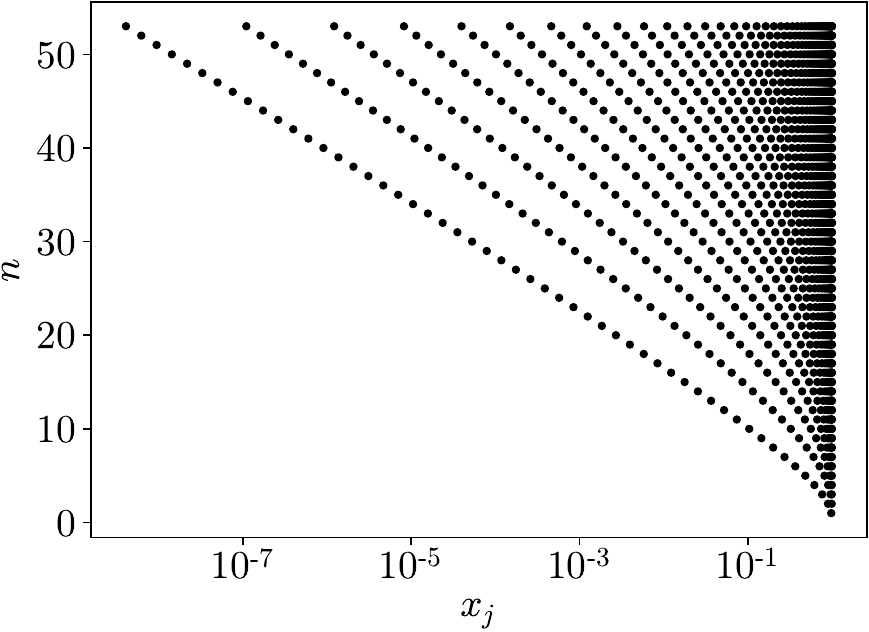}
}
\caption{The distribution of collocation points
$\{x_j\}_{j=1}^N$
over $[0,1]$,
for values of $n$ such that $\alpha_n \gtrsim \epsilon_0$, and 
$\gamma=10$, $50$, $250$.}
\label{fig:nx}
\end{figure}

\begin{figure}
  \centering
  \includegraphics[scale=0.40]{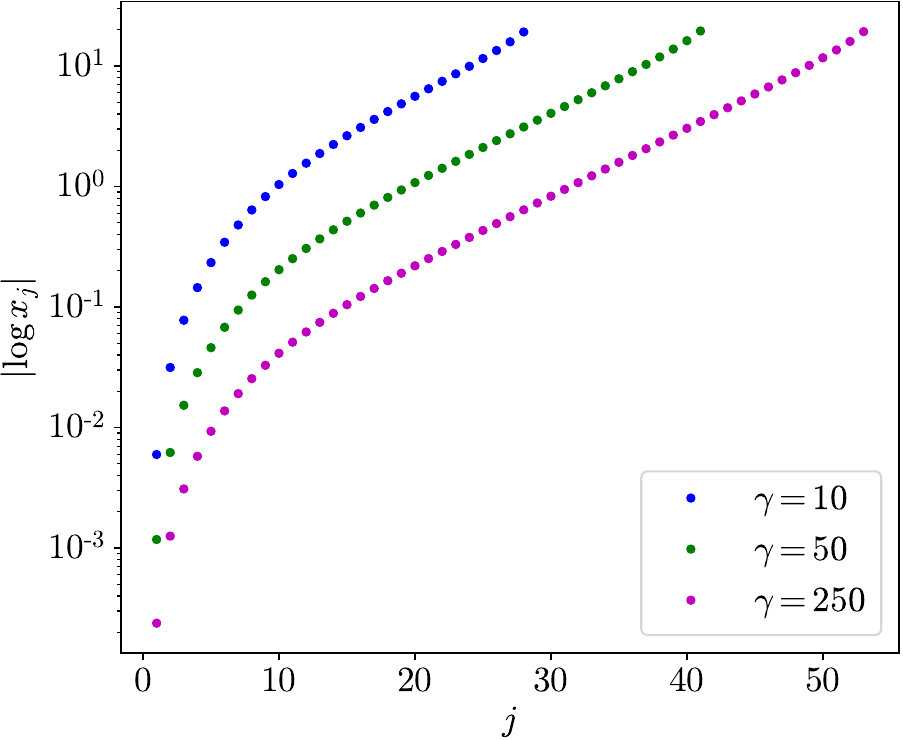}
  \caption{The distribution of collocation points
  $\{x_j\}_{j=1}^N$
  over $[0,1]$,
  for fixed values of $n$ such that $\alpha_n \approx \epsilon_0$, and $a=1$,
  $b=10$, $50$, $250$.}
  \label{fig:nx_all}
\end{figure}

\section{Conclusion}
In this paper, we introduce an approach to approximate functions
of the form
$f(x)=\int_{a}^{b} x^{\mu} \sigma(\mu) \, d \mu$ over the interval $[0,1]$,
by expansions in a small number of singular powers $x^{t_1}$, $x^{t_2}$, \ldots ,
$x^{t_N}$,
where $0<a<b<\infty$ and $\sigma(\mu)$ is some signed Radon measure or some distribution
supported on $[a,b]$. 
Given any desired accuracy $\epsilon$, 
our method guarantees that the uniform approximation error over
the entire interval $[0,1]$
is bounded by $\epsilon$ multiplied by certain small constants. Additionally, 
the number of basis functions $N$ grows asymptotically as $O(\log{\frac{1}{\epsilon}})$,
and the expansion coefficients can be found by
collocating the function at specially chosen collocation points
$x_1$, $x_2$, \ldots , $x_N$ and solving 
an $N \times N$ linear system numerically. 
In practice, when $\frac{b}{a}=10$ and $\sigma$ is a signed Radon measure, 
our method
requires only approximately $N=30$ basis functions
and collocation points in order to achieve 
machine precision accuracy. Numerical experiments 
demonstrate that our method can also be used for approximation 
over simple 
smooth arcs in the complex plane.
A key feature of our method is that both the basis functions
and the collocation points
are determined   
a priori by only the values of $a$, $b$, and $\epsilon$. 
This sets it apart from expert-driven approximation methods, and from
other methods that rely on careful selection of
parameters to determine the basis functions. For example, the basis functions 
used in lightning and reciprocal-log approximation are defined by the
locations of poles, and the SE-Sinc and DE-Sinc approximations
depend on the choices of smooth transformations.  
Compared to the DE-Sinc approximation, which achieves
nearly-exponential rates of convergence at the cost of double-exponentially
clustered collocation points, our method uses collocation points
that cluster double-exponentially only for points that are close
to the singularity,
and at a slower rate, rather than double-exponentially, for points that are further away. 
Moreover, the closest collocation point required to achieve an approximation error 
of size $\alpha_n$ approaches the singularity at only an exponential rate as $n$ increases.
For a fixed desired accuracy $\epsilon$, the closest collocation point 
stays at the same distance from the singularity for a fixed value of $a$ and varying 
values of $b$.
Compared to reciprocal-log approximation, which requires the least-squares
solution of an overdetermined linear system with many collocation points,
our method involves the solution of a small square linear system to
determine the expansion coefficients.

Since our method approximates singular functions
accurately by expansions in singular powers,
it can be used with existing finite element methods
or integral equation methods to 
approximate
the solutions of PDEs on nonsmooth geometries or with discontinuous data.
Typically, the leading singular terms of the asymptotic expansions of solutions
near corners are derived from 
the angles at the corners, and are added to the basis functions 
of finite element methods to enhance the convergence rates  
(see, for example, \cite{tong}, \cite{fix}, 
\cite{olson}).
Now, with only the knowledge that the singular solutions are of the
form~\Cref{eq:repf}, we can enhance the convergence rates of finite element
methods without knowledge of the angles at the corners, by adding all of the
singular powers obtained from our method to the basis functions.  Likewise,
the singular powers obtained from our method can be used in integral
equation methods for PDEs.
In integral equation methods, 
boundary value problems for PDEs are reformulated
as integral equations
for boundary charge and dipole densities which represent their solutions. 
Previously, singular asymptotic expansions of the densities, determined
by the angles at the corners, were used to construct special quadrature rules
to solve these integral equations (see, for example, \cite{corners},
\cite{corners2}).
Using only the fact that the singular
densities are of the form~\Cref{eq:repf}, quadrature rules can instead be developed 
for only the singular
powers obtained from our method, independent of the angles at the corners.


\begin{thebibliography}{99}

\bibitem{MAM2}
Babu\v{s}ka, I., B. Andersson, B. Guo, J. M. Melenk, and H. S. Oh.
``Finite Element Method for Solving Problems with Singular Solutions."
\textit{J. Comput. Appl. Math.} 74 (1996): 51--70.

\bibitem{nm}
Bauer, F.L., and C.T. Fike.
``Norms and Exclusion Theorems."
\textit{Numer. Math.} 2.1 (1960): 137--141.

\bibitem{compact}
Bertero, M., P. Boccacci, and E.R. Pike. 
``On the Recovery and Resolution 
of Exponential Relaxation Rates from Experimental Data:
a Singular-value
Analysis of the Laplace Transform Inversion in the Presence 
of Noise."
\textit{P. Roy. Soc. A-Math. Phy.} 383 (1982): 15--29.

\bibitem{expo}
Beylkin, G., and L. Monz\'{o}n.
``On Approximation of Functions by Exponential Sums."
\textit{Appl. Comput. Harmon. A.} 19.1 (2005): 1063--5203.

\bibitem{spec}
Chen, S., and J. Shen. 
``Enriched Spectral Methods and Applications to Problems with
Weakly Singular Solutions."
\textit{J. Sci. Comput.} 77 (2018): 1468--1489.

\bibitem{coppe}
Copp\'{e}, V., D. Huybrechs, R. Matthysen, and M. Webb.
``The AZ Algorithm for Least Squares Systems with a
Known Incomplete Generalized Inverse."
\textit{SIAM J. Matrix Anal. A.} 41.3 (2020): 1237--1259.

\bibitem{minimax}
Filip, S., Y. Nakatsukasa, L. N. Trefethen, and B. Beckermann.
``Rational Minimax Approximation via Adaptive Barycentric Representations."
\textit{SIAM J. Sci. Comput.} 40.4 (2018): A2427--A2455.

\bibitem{fix}
Fix, G. J., S. Gulati, and G. I. Wakoff.
``On the Use of Singular Functions with Finite Element Approximations."
\textit{J. Comput. Phys.} 13.2 (1973): 209--228.

\bibitem{extend}
Fries, T.-P., and T. Belytschko.
``The Extended/generalized Finite Element Method: an 
Overview of the Method and Its Applications."
\textit{Int. J. Numer. Methods Eng.} 84.3 (2010): 253--304.

\bibitem{rapid}
Gon\v{c}ar, A. A.
``On the Rapidity of Rational Approximation of Continuous Functions with
Characteristic Singularities."
\textit{Math. USSR-Sb.} 2.4 (1967): 561--568.

\bibitem{gopal2}
Gopal, A., and L. N. Trefethen.
``New Laplace and Helmholtz
Solvers."
\textit{Proc. Natl. Acad. Sci.} 116.21 (2019): 10223--10225.

\bibitem{gopal}
Gopal, A., and L. N. Trefethen.
``Solving Laplace Problems
with Corner Singularities via Rational Functions."
\textit{SIAM J. Numer. Anal.} 57.5 (2019): 2074--2094.

\bibitem{unique}
Gutknecht, M. H., and L. N. Trefethen. 
``Nonuniqueness of Best Rational Chebyshev Approximations on the Unit Disk."
\textit{J. Approx. Theory} 39.3 (1983): 275--288.

\bibitem{svd}
Hansen, P.C.
``The truncated SVD as a Method for Regularization."
\textit{BIT.} 27.4 (1987): 534--553.

\bibitem{enriched}
Herremans, A., and D. Huybrechs.
``Efficient Function Approximation in Enriched Approximation Spaces."
\textit{IMA J. Numer. Anal.} drae017, 2024.

\bibitem{resolution}
Herremans, A., D. Huybrechs, and L.N. Trefethen.
``Resolution of Singularities by Rational Functions."
\textit{SIAM J. Numer. Anal.} 61.6 (2023): 2580--2600.

\bibitem{laplace}
Lederman, R.R., and V. Rokhlin.
``On the Analytical and Numerical 
Properties of the Truncated Laplace Transform."
\textit{SIAM J. Numer. Anal.} 53.3 (2015): 1214--1235.

\bibitem{laplace2}
Lederman, R.R., and V. Rokhlin.
``On the Analytical and Numerical 
Properties of the Truncated Laplace Transform. Part II."
\textit{SIAM J. Numer. Anal.} 54.2 (2016): 665--687.

\bibitem{lehman}
Lehman, R. S. 
``Developments at an Analytic Corner of Solutions of Elliptic Partial
Differential Equations."
\textit{J. Math. Mech.} 8.5 (1959): 727--760.

\bibitem{MAM}
Lucas, T. R., and H. S. Oh.
``The Method of Auxiliary Mapping for the Finite Element Solutions
of Elliptic Problems Containing Singularities."
\textit{J. Comput. Phys.} 108.2 (1993): 327--342.

\bibitem{develop}
Mori, M. 
``Discovery of the Double Exponential
Transformation and Its Developments."
\textit{Publ. Res. Inst. Math. Sci.} 41 (2005): 897--935.

\bibitem{real}
Nakatsukasa, Y., and L. N. Trefethen.
``An Algorithm for Real and Complex Rational Minimax Approximation."
\textit{SIAM J. Sci. Comput.} 42.5 (2020): A3157--A3179.

\bibitem{log}
Nakatsukasa, Y., and L. N. Trefethen.
``Reciprocal-log Approximation and Planar PDE Solvers."
\textit{SIAM J. Numer. Anal.} 59.6 (2021): 2801--2822.

\bibitem{aaa}
Nakatsukasa, Y., O. S\`ete, and L. N. Trefethen. 
``The AAA Algorithm for Rational Approximation."
\textit{SIAM J. Sci. Comput.} 40.3 (2018): A1494--A1522.

\bibitem{newman}
Newman, D. J.
``Rational Approximation to $|x|$."
\textit{Mich. Math. J.} 11.1 (1964): 11--14.

\bibitem{sinc}
Okayama, T., T. Matsuo, and M. Sugihara.
``Sinc-collocation Methods
for Weakly Singular Fredholm Integral Equations of the Second Kind."
\textit{J. Comput. Appl. Math.} 234.4 (2010): 1211--1227.

\bibitem{olson}
Olson, L. G., G. C. Georgiou, and W. W. Schultz.
``An Efficient Finite Element Method for Treating Singularities in 
Laplace's Equation."
\textit{J. Comput. Phys.} 96.2 (1991): 391--410.

\bibitem{pseudo}
Roache, P. J.
``A Pseudo-spectral FFT Technique for Non-periodic Problems."
\textit{J. Comput. Phys.} 27.2 (1978): 204--220.

\bibitem{corners}
Serkh, K.
``On the Solution of Elliptic Partial Differential Equations
on Regions with Corners."
\textit{J. Comput. Phys.} 305 (2016): 150--171.

\bibitem{corners2}
Serkh, K., and V. Rokhlin.
``On the Solution of the Helmholtz Equation on Regions with Corners."
\textit{Proc. Natl. Acad. Sci.} 113.33 (2016): 9171--9176.

\bibitem{roots}
Serkh, K., and V. Rokhlin.
``A Provably Componentwise Backward Stable $O(n^2)$ QR Algorithm
for the Diagonalization of Colleague Matrices." 
\textit{ArXiv}, 2021.

\bibitem{best}
Stahl, H.
``Best Uniform Rational Approximation of $x^{\alpha}$ on $[0,1]$."
\textit{Acta. Math.} 190 (2003): 241--306.

\bibitem{stenger2}
Stenger, F.
``Explicit Nearly Optimal Linear Rational Approximation with
Preassigned Poles."
\textit{Math. Comput.} 47.175 (1986): 225--252.

\bibitem{stenger}
Stenger, F. 
``Numerical Methods Based on Sinc and Analytic Functions
in numerical Analysis."
\textit{SSCM} Springer-Verlag, 1993.

\bibitem{tong}
Tong, P., and T. H. H. Pian.
``On the Convergence of the Finite Element Method
for Problems with Singularity."
\textit{Int. J. Solids Struct.} 9.3 (1973): 313--321.

\bibitem{cluster}
Trefethen, L. N., Y. Nakatsukasa, and J. A. C. Weideman. 
``Exponential Node Clustering at Singularities for Rational
Approximation, Quadrature, and PDEs."
\textit{Numer. Math.} 147.1 (2021): 227--254.

\bibitem{nick}
Trefethen, L. N. 
``Approximation Theory and Approximation Practice, Extended Edition."
\textit{Society for Industrial and Applied Mathematics}, 2019.

\bibitem{wasow}
Wasow, W. 
``Asymptotic Development of the Solution of
Dirichlet’s Problem at Analytic Corners."
\textit{Duke Math. J.} 24.1 (1957): 47--56.


\end{thebibliography}
\end{document}